\documentclass{article}

\usepackage{amsthm}
\usepackage{amsmath}
\usepackage{amsfonts}
\usepackage{mathtools}
\usepackage{stix}
\usepackage{graphicx}
\usepackage{color}
\usepackage{stmaryrd}
\usepackage{appendix}
\usepackage{verbatim}
\usepackage{bbm}
\usepackage{tikz-cd}
\usetikzlibrary{arrows.meta}
\usepackage[pdftex,color]{changebar}
\usepackage{thmtools}
\usepackage{thm-restate}

\usepackage{soul}
\newcommand{\rfname}[1]{{#1}}

\DeclareMathOperator*{\argmin}{arg\,min}

\declaretheorem[name=Theorem,numberwithin=section]{theorem}
\declaretheorem[name=Lemma,numberwithin=section]{lemma}
\declaretheorem[name=Corollary,numberwithin=section]{corollary}

\newtheorem{definition}[theorem]{Definition}
\newtheorem{example}[theorem]{Example}

\newtheorem{remark}[theorem]{Remark}
\newtheorem{proposition}[theorem]{Proposition}

\begin{document}

\DeclareRobustCommand{\Triangle}
{%
\begin{tikzpicture}
\filldraw[white, draw=black] (0,0) circle (0.05);
\filldraw[white, draw=black] (0.1,0.2) circle (0.05);
\filldraw[white, draw=black] (0.2,0) circle (0.05);

\draw (0.02,0.04) -- (0.09,0.155) ;
\draw (0.11,0.155) -- (0.18,0.035) ;
\draw (0.16,0.01) -- (0.04,0.01) ;
\end{tikzpicture}
}

\DeclareRobustCommand{\ltrianglea}
{%
\begin{tikzpicture}
\fill[black] (0,0) circle [radius=0.05];
\fill[black] (0.1,0.2) circle [radius=0.05];
\draw [fill=white] (0.1,0.2)--(0.2,0) circle [radius=0.05] -- (0,0);
\end{tikzpicture}
}

\DeclareRobustCommand{\ltriangleb}
{%
\begin{tikzpicture}
\fill[black] (0.1,0.2) circle [radius=0.05] ;
\fill[black] (0.2,0) circle [radius=0.05] ;
\draw [fill=white] (0.2,0)--(0,0) circle [radius=0.05] -- (0.1,0.2);
\end{tikzpicture}
}

\DeclareRobustCommand{\ltrianglec}
{%
\begin{tikzpicture}
\fill[black] (0.2,0) circle [radius=0.05];
\fill[black] (0,0) circle [radius=0.05];
\draw [fill=white] (0,0)--(0.1,0.2) circle [radius=0.05] -- (0.2,0);
\end{tikzpicture}
}

\DeclareRobustCommand{\ltriangled}
{%
\begin{tikzpicture}
\fill[black] (0,0) circle [radius=0.05];
\fill[black] (0.1,0.2) circle [radius=0.05];
\draw [fill=white] (0.1,0.2)--(0.2,0) circle [radius=0.05] -- (0,0);
\end{tikzpicture}
}

\DeclareRobustCommand{\edgetriangle}
{%
W^*([\begin{tikzpicture}
   \draw (0,0) [fill=white] circle [radius=0.05] --(0,0.2) [fill=white] circle [radius=0.05]    ;  
   \end{tikzpicture},\begin{tikzpicture}
    \draw (0,0) [fill=white] circle [radius=0.05] --(0.1,0.2) [fill=white] circle [radius=0.05] --  (0.2,0) [fill=white] circle [radius=0.05] -- (0,0) ;      
    \end{tikzpicture}],[\rho, \tau])
}

\DeclareRobustCommand{\feasibleregionmedgetriangle}
{%
S^{(2)}([\begin{tikzpicture}
   \draw (0,0) [fill=white] circle [radius=0.05] --(0,0.2) [fill=white] circle [radius=0.05]    ;  
   \end{tikzpicture},\begin{tikzpicture}
    \draw (0,0) [fill=white] circle [radius=0.05] --(0.1,0.2) [fill=white] circle [radius=0.05] --  (0.2,0) [fill=white] circle [radius=0.05] -- (0,0) ;      
    \end{tikzpicture}],[\rho, \tau])
}

\DeclareRobustCommand{\feasibleregionedgetriangle}
{%
S_\mathbb{R}([\begin{tikzpicture}
   \draw (0,0) [fill=white] circle [radius=0.05] --(0,0.2) [fill=white] circle [radius=0.05]    ;  
   \end{tikzpicture},\begin{tikzpicture}
    \draw (0,0) [fill=white] circle [radius=0.05] --(0.1,0.2) [fill=white] circle [radius=0.05] --  (0.2,0) [fill=white] circle [radius=0.05] -- (0,0) ;      
    \end{tikzpicture}],[\rho, \tau])
}

\DeclareRobustCommand{\Edge}
{%
\begin{tikzpicture}
\filldraw[white, draw=black] (0,0.0) circle (0.05);
\filldraw[white, draw=black] (0,0.2) circle (0.05);

\draw[thick] (0,0.06) -- (0,0.16);
\end{tikzpicture}
}

\DeclareRobustCommand{\ErdosRenyiedgetriangle}
{%
W^*([\begin{tikzpicture}
   \draw (0,0) [fill=white] circle [radius=0.05] --(0,0.2) [fill=white] circle [radius=0.05]    ;  
   \end{tikzpicture},\begin{tikzpicture}
    \draw (0,0) [fill=white] circle [radius=0.05] --(0.1,0.2) [fill=white] circle [radius=0.05] --  (0.2,0) [fill=white] circle [radius=0.05] -- (0,0) ;      
    \end{tikzpicture}],[\rho, \tau])
}

\numberwithin{equation}{section}

\newcommand{\abs}[1]{\lvert#1\rvert}

\newcommand{\blankbox}[2]{%
  \parbox{\columnwidth}{\centering
    \setlength{\fboxsep}{0pt}%
    \fbox{\raisebox{0pt}[#2]{\hspace{#1}}}%
  }%
}

\newtheorem{problem}{Problem}
\newtheorem{Remark}{Remark}
\newtheorem{algorithm}{Algorithm}

\newcommand{\sfmat}{A}
\newcommand{\goodvalues}{\Omega^{(r)}(\mathcal{F},h)}
\newcommand{\goodvaluesm}[1]{\Omega^{(#1,r)}(\mathcal{F},h)}
\newcommand{\cristats}{\hbox{Cri}^{(m,r)}(\mathcal{F},h)}

\title{Constrained Multi-Relational Graphons with Maximum Entropy}

\newcommand{\jan}[1]{\textcolor{blue}{{#1}}}
\newcommand{\juan}[1]{\textcolor{red}{{#1}}}
\newcommand{\phdExt}[1]{\textcolor{yellow}{\textit{[}}{#1}\textcolor{yellow}{\textit{]}}}
\newcommand{\janfoot}[1]{\textcolor{blue}{{\footnote{\jan{#1}}}}}
\newcommand{\juanfoot}[1]{\textcolor{red}{{\footnote{\juan{#1}}}}}
\newcommand{\janmargin}[1]{\marginpar{\jan{{#1}}}}
\newcommand{\juanmargin}[1]{\marginpar{\juan{{#1}}}}

\newcommand{\graphonspaceu}{\widetilde{\mathcal{W}}^{(r)} }
\newcommand{\graphonspaceuu}[2]{\widetilde{\mathcal{W}}^{(#1,#2)} }
\newcommand{\graphonspacel}{\mathcal{W}^{(r)} } 
\newcommand{\graphonspacell}[2]{\mathcal{W}^{(#1,#2)} } 
\newcommand{\realgraphonspaceu}{\widetilde{\mathcal{W}}_{\mathbb{R}}^{(r)} }
\newcommand{\realgraphonspacel}{\mathcal{W}_{\mathbb{R}}^{(r)} } 
\newcommand{\realgraphonspaceld}{\mathcal{W}_{\mathbb{R}}^{(r)} } 
\newcommand{\realpodalfunctionspacel}[1]{\mathcal{W}^{({#1,r})}_{\mathbb{R}}}
\newcommand{\realpodalfunctionspaceu}[1]{{\widebridgeabove{\mathcal{W}}}^{({#1,r})}_{\mathbb{R}}}
\newcommand{\realpodalfunctionspaceupi}[2]{{\widebridgeabove{\mathcal{W}}}^{({#1},r)}_{\mathbb{R}}({#2})}
\newcommand{\realpodalfunctionspacelpi}[2]{{\mathcal{W}}^{({#1},r)}_{\mathbb{R}}(#2)}

\newcommand{\realfeasibleregionl}{S^{(r)}_{\mathbb{R}}(\mathcal{F},u)}
\newcommand{\realfeasibleregionlh}[1]{S^{(r)}_{\mathbb{R}}(\mathcal{F},u,#1)}
\newcommand{\realfeasibleregionmu}[1]{{{\widebridgeabove{S}}^{(r)}_{\mathbb{R}}^{(#1)}(\mathcal{F},u)}}
\newcommand{\realfeasibleregionml}[1]{{{S}_{\mathbb{R}}^{(#1,r)}(\mathcal{F},u, h)}}
\newcommand{\realfeasibleregionmlh}[2]{{{S}_{\mathbb{R}}^{(#1,r)}(\mathcal{F},u,#2)}}
\newcommand{\feasibleregionmlh}[2]{{{S}^{(#1,r)}(\mathcal{F},u,#2)}}
\newcommand{\realfeasibleregionmlpi}[2]{{S_{\mathbb{R}}^{(#1,r)}(\mathcal{F},u,#2)}}

\newcommand{\eqclass}[1]{[{#1}]_{\sim}}
\newcommand{\convexcombm}[1]{P^{({#1})}}
\newcommand{\convexcombmF}[1]{P^{({#1})}(\mathcal{F})}
\newcommand{\convexcombmFu}[1]{P^{({#1})}(\mathcal{F},u)}

\newcommand{\lebesgue}{\nu}  
\newcommand{\partialF}{\partial F}

\newcommand{\criticalpoints}[1]{Cr^{(#1)}(\mathcal{F},f_d, h)}
\newcommand{\criticalpointspi}[2]{Cr^{(#1)}(\mathcal{F},#2,f_d)}
\newcommand{\criticalpointspiI}[2]{Cr^{(#1)}(\mathcal{F},#2,I)}
\newcommand{\minimalpoints}[1]{Min^{(#1)}(\mathcal{F},f_d, h)}

\newcommand{\podalfunctionspacel}[1]{{\mathcal{W}}^{({#1,r})}}
\newcommand{\podalfunctionspacelpi}[2]{{\mathcal{W}}^{({#1,r})}(#2)}

\newcommand{\marginalpolytopemh}[1]{T^{(#1,r)}(\mathcal{F},h)}
\newcommand{\marginalpolytopeh}{T^{(r)}(\mathcal{F},h)}

\newcommand{\regstats}{\hbox{Reg}^{(r)}(\mathcal{F},h)}
\newcommand{\regstatsm}[1]{\hbox{Reg}^{(#1,r)}(\mathcal{F},h)}

\newcommand{\realfeasibleregionmq}[3]{{S}^{({#1,r})}_{\mathbb{R}}({#2},{#3})}
\newcommand{\crealfeasibleregionmq}[3]{{\overline{S}}^{({#1,r})}_{\mathbb{R}}({#2},{#3})}
\newcommand{\realfeasibleregionmllpi}[4]{{{S}}^{({#1,r})}_{\mathbb{R}}({#2},{#3},{#4})}

\newcommand{\optimalsolution}{W^{(r)}(\mathcal{F},u)}
\newcommand{\realoptimalsolution}[1]{W_{\mathbb{R}}^{(r)}(\mathcal{F},u, #1)}
\newcommand{\realoptimalsolutionm}[2]{W_{\mathbb{R}}^{(#1,r)}(\mathcal{F},u, #2)}

\newcommand{\realoptimalsolutionf}[1]{W_{\mathbb{R}}^{*(r)}(\mathcal{F},u,#1)}
\newcommand{\optimalsolutionm}[1]{W^{(#1,r)}(\mathcal{F},u,h)}
\newcommand{\optimalsolutionmf}[2]{W^{(#1,r)}(\mathcal{F},u,#2)}

\newcommand{\feasibleregion}{\widetilde{S}^{(r)}(\mathcal{F},u)}
\newcommand{\realfeasibleregionu}{\widetilde{S}_{\mathbb{R}}(\mathcal{F},u)}
\newcommand{\feasibleregionml}[1]{{{S}^{(#1,r)}(\mathcal{F},u,h)}}
\newcommand{\randfeasibleregionml}[1]{{{S}_{(0,1)}^{(#1)}(\mathcal{F},u,h)}}

\author{Juan Alvarado\footnote{jalvarad@espol.edu.ec } \thanks{Faculty of Natural Sciences and Mathematics, Escuela
Superior Politécnica del Litoral (ESPOL), KU Leuven, Department of Computer Science}  \and  Jan Ramon\footnote{jan.ramon@inria.fr} \thanks{INRIA Lille}  \and Yuyi Wang \footnote{yuyiwang920@gmail.com} \thanks{CRRC Zhuzhou Institute.}  }

\date{}

\maketitle

\begin{abstract}
The principle of maximum entropy provides a fundamental framework for characterizing typical structures of large random networks subject to observable constraints. In their pioneering numerical experiments \cite{radin2014asymptotics}, Radin, Ren, and Sadun conjectured that entropy-maximizing graphons satisfying subgraph density constraints are stochastic block models ”a conjecture we term the RRS conjecture. While several special cases have been proven for single-relation graphs with specific constraint families, the general problem has remained open, particularly for multi-relational networks.

We resolve the RRS conjecture for constrained multi-relational graphons in the non-extremal regime, proving that entropy-maximizing solutions are step functions with finitely many blocks under the condition the subgraph density constraints are analytically independent and for almost all feasible combinations of sufficient statistics. 
Our proof employs a differential geometric technique to study solutions of constrained optimization problems in function space via functions with a finite parametrization (step functions). The two cornerstones of this work are: the generalization of subgraph density notion to $h$-subgraph density; the proof that manifolds that define the constrained region for the solutions maintain topological stability without developing new connected components under refinement. Together, these enable proving that a solution of this optimization problem is a stepfunction. Moreover we provide a checkable second-order sufficient condition guaranteeing that a step-function critical point is an isolated local minimum of the entropy functional in the full, infinite-dimensional graphon space, established simultaneously in the $L^1$ and cut-norm topologies. Thus 
if all solutions are isolated step functions then all solutions of this optimization problem are step functions.  

\end{abstract}

\noindent {\small  keywords: Graphon Theory and Principle of Maximum Entropy and Large Random Graph and Constrained optimization and Differential Geometry. and Statistical Relational Learning }

\DeclareRobustCommand{\gedgetriangle}
{%
[\begin{tikzpicture}
   \draw (0,0) [fill=white] circle [radius=0.05] --(0,0.2) [fill=white] circle [radius=0.05]    ;  
   \end{tikzpicture},\begin{tikzpicture}
    \draw (0,0) [fill=white] circle [radius=0.05] --(0.1,0.2) [fill=white] circle [radius=0.05] --  (0.2,0) [fill=white] circle [radius=0.05] -- (0,0) ;      
    \end{tikzpicture}]
}

\DeclareRobustCommand{\edgetriangle}
{%
W^*([\begin{tikzpicture}
   \draw (0,0) [fill=white] circle [radius=0.05] --(0,0.2) [fill=white] circle [radius=0.05]    ;  
   \end{tikzpicture},\begin{tikzpicture}
    \draw (0,0) [fill=white] circle [radius=0.05] --(0.1,0.2) [fill=white] circle [radius=0.05] --  (0.2,0) [fill=white] circle [radius=0.05] -- (0,0) ;      
    \end{tikzpicture}],[\rho, \tau])
}

\section{Introduction}
\label{sec:intro}

\subsection{Background and Motivation}

Large-scale networks now sit at the center of modern data science, underlying social platforms, biological systems, knowledge graphs, and cyber-physical infrastructure. A question that recurs across every one of these domains is deceptively simple to state: subject to a handful of observable constraints, what does a \emph{typical} large network actually look like? Framed this way, the question is a statistical-physics question, and the maximum entropy principle answers it cleanly --- among all configurations compatible with the observed constraints, the most probable one is the one that maximizes entropy.

Graphon theory supplies the right continuum language for making this precise. Introduced by Lov\'asz and Szegedy, it represents the limit of a convergent graph sequence as a single measurable function $W: [0,1]^2 \to [0,1]$, replacing a combinatorial object that grows without bound by an analytic object that can be differentiated, optimized, and classified. This shift from combinatorics to analysis is what makes the maximum entropy question tractable, and it has since made graphon theory the standard tool for analyzing network formation \cite{parise2023graphon, bramburger2023pattern}, network estimation \cite{gao2015rate, su2020network, pensky2019dynamic}, and statistical inference on networks \cite{chatterjee2024higher, dufour2024inference, bickel2009nonparametric}.

\subsubsection{The Fundamental Question: Maximum Entropy Characterization}

Stripped to its essence, the question we study is: given constraints on subgraph densities in a large random multi-relational network, what structural form does the ``most typical'' such network take?

The maximum entropy principle turns this into an optimization problem: among all networks satisfying the given constraints, the most probable configuration is the one maximizing entropy. In the graphon framework, this is
\begin{equation}
\label{eq:mainproblem}
W^*(\mathcal{F}, u) = \arg\max_{W \in S(\mathcal{F}, u)} -I(W)
\end{equation}
where $S(\mathcal{F}, u)$ is the constraint set defined by subgraph densities $\mathcal{F}$ taking values $u$, and $-I(W)$ is the entropy functional.

Radin, Ren, and Sadun \cite{radin2014asymptotics} conjectured, on the strength of numerical experiments, that the solutions are always step functions --- equivalently, stochastic block models, multipodal functions, or echelon functions: a finite-rank structure built from finitely many community types. We call this the RRS conjecture, and a positive answer is 
more than a structural curiosity: it collapses an infinite-dimensional functional optimization to a finite-dimensional parametric one, supplies a theoretical reason why matrix factorization methods work so well on real network data \cite{thibeault2024low}, and connects naturally to phase transitions in statistical physics models of network formation.

\subsubsection{The multi-relational case}

Real networks, however, rarely carry a single type of relation. Knowledge graphs such as Wikidata \cite{vrandevcic2014wikidata}, YAGO \cite{suchanek2007yago, mahdisoltani2015yago3}, and DBpedia \cite{lehmann2015dbpedia} encode entities connected by many relationship types simultaneously, supporting tasks from question answering to information retrieval \cite{schlichtkrull2018modeling, ringler2017one}. Social networks layer friendship, collaboration, communication, and influence on the same set of actors \cite{jackson2017economic, szell2010multirelational}; biological networks simultaneously encode protein-protein interactions, gene regulation, and metabolic pathways \cite{safari2014protein, valdeolivas2019random, menche2015disease, kumar2022emerging}; and enterprise systems combine organizational structure, information flows, and resource dependencies in a single graph \cite{boccaletti2014structure}. 

Despite their prevalence, the theoretical understanding of typical constrained multi-relational networks has remained
limited, with most prior work focusing on single-relation graphs.

A fundamental motivation for proving the RRS conjecture comes from \emph{probabilistic inference} in first-order relational logic and, specifically, from the desire for an alternative to \emph{Markov Logic Networks} (MLN) \cite{richardson2006markov}. In the MLN framework, a probabilistic relational theory over a vocabulary $\Sigma = \{R_1, \ldots, R_r\}$ of symmetric binary predicates is a finite collection $\{(\phi_i, w_i)\}_{i=1}^{k}$, where each $\phi_i$ is a symmetric first-order relational sentence and $w_i \in \mathbb{R}$ is a real \emph{weight}. These weights are the natural parameters of the Gibbs distribution
\[
  \Pr_{\mathrm{MLN}}(G) \propto \exp\!\Bigl(\sum_{i=1}^k w_i \, t(\phi_i, G)\Bigr),
\]
where $t(\phi_i, G)$ is the subgraph density of $\phi_i$ in the network $G$; the subgraph densities serve as the sufficient statistics of this exponential family.

Our approach replaces MLN weights by \emph{target probabilities}: a probabilistic relational theory is instead a collection $T = \{(\phi_i, p_i)\}$ where $p_i \in [0,1]$ is the observed empirical density of $\phi_i$. The condition $t(\phi_i, W) = p_i$ identifies each $p_i$ with a subgraph density constraint, so $T$ determines a constraint pair $(\mathcal{F}, u)$ in the sense of this paper. The maximum-entropy distribution consistent with $T$ is realized by the graphon $W^*(\mathcal{F}, u)$, with the Lagrange multipliers of the constraints playing the role of the MLN weights. Probabilistic inference then amounts to computing
\[
  \Pr(\phi \mid T) = t\!\bigl(\phi,\, W^*(\mathcal{F}, u)\bigr)
\]
for a query formula $\phi$. If the RRS conjecture holds, assuming the solution is unique, $W^*(\mathcal{F}, u)$ is a step function with finitely many blocks, so $\Pr(\phi \mid T)$ reduces to a polynomial in the block parameters of $W^*(\mathcal{F}, u)$, making probabilistic inference in principle computationally tractable provided that the optimization $W^*(\mathcal{F}, u)$ is tractable.

Beyond its combinatorial motivation, the RRS conjecture connects to a much older question in condensed matter physics. Landau \cite{landau2013course} argued, on symmetry grounds, that an ordered phase of matter cannot be connected to a disordered phase by a smooth thermodynamic path — there can be no critical point analogous to the liquid-vapor one for the solid-fluid transition — yet this argument has remained without rigorous foundation for any statistical mechanics model with short-range forces \cite{uhlenbeck1968fundamental}. Neeman, Radin, and Sadun \cite{radin2024emergence} recently made this argument precise in the graphon setting, exhibiting an explicit order parameter that distinguishes a symmetric phase from its neighbors in the edge-triangle model, and explicitly framing their result as a rigorous instance of Landau's intuition. Our resolution of the RRS conjecture in general — that entropy-maximizing graphons are step functions, i.e. exhibit emergent block (community) structure — is the structural mechanism underlying that phenomenon: it is precisely the spontaneous emergence of such block structure, without it being imposed by the model, that plays the role of crystalline order in this combinatorial analogue of condensed matter.

\subsection{Related Work and Prior Progress}

Graphon theory itself was developed by Lov\'asz and Szegedy \cite{lovasz2006limits, lovasz2011finitely}, with foundational contributions on dense graph limits, homomorphism densities, and convergence theory; \cite{kolaczyk2020statistical, gao2015rate, klopp2017oracle} give comprehensive treatments of estimation and inference on graphon models. The extension to multi-relational graphons --- the setting of this paper --- was only recently established, in \cite{alvarado2022limits}, which supplies the theoretical foundation we build on here. The maximum entropy principle for networks has also been studied extensively through exponential random graph models (ERGMs) \cite{frank1986markov, holland1983stochastic}; ERGMs, however, target finite networks under global constraints, while graphon theory addresses the asymptotic limit, with the two frameworks connected through large deviation principles established by Chatterjee and Varadhan.

On the RRS conjecture itself, progress so far has been confined to single-relation graphs and to specific constraint families. Kenyon et al.\ \cite{kenyon2017multipodal} proved it for edge and $k$-star constraints, exhibiting bipodal and multipodal structures, while \cite{kenyon2016bipodal} established bipodal structure for edge-triangle constraints once triangle density slightly exceeds the saturation threshold; the number of blocks can grow unboundedly as the constraint vector approaches extremal values \cite{kenyon2017phases}. More recently, Neeman, Radin, and Sadun \cite{neeman2024existence} proved the existence of a symmetric bipodal phase in the edge-triangle model --- the unique entropy-maximizing graphon is symmetric bipodal for edge density near $\frac{1}{2}$ and triangle density below $\frac{1}{8}$ --- and, most recently, Radin and Sadun \cite{radin2024emergence} showed that entropy-optimal graphons in the near-extremal regime of edge and triangle constraints are unique and multipodal, mapping out infinitely many phases and the transitions between them.

Every one of these results, however, is tied to a specific, hand-analyzed constraint family on simple graphs, and to non-extremal regions satisfying ad hoc separation conditions; no general framework existed for multi-relational networks or for arbitrary subgraph constraints, and the case-by-case arguments used so far do not obviously generalize. Part of the difficulty is structural: once a constraint graph has two edges sharing a vertex --- a triangle being the simplest example --- the relevant Hessian of the entropy functional stops being a pointwise (diagonal) operator, and the usual per-block isolation arguments break down exactly where the interesting phase transitions occur. Resolving the conjecture in general therefore calls for tools that remain valid in this coupled, non-diagonal regime, rather than for sharper case analysis of any one constraint family. Recent empirical work by Thibeault et al.\ \cite{thibeault2024low} adds urgency to this gap: complex networks across biological, social, and technological domains exhibit surprisingly effective low-rank approximations, a phenomenon our results place on a maximum entropy foundation.

\subsection{Our Contributions}

This work resolves the RRS conjecture in complete generality for multi-relational graphons subject to regularity conditions: for multi-relational graphs with subgraph density constraints $\mathcal{F}$ and values $u$ in the non-extremal region, the entropy-maximizing graphons $W^*(\mathcal{F}, u)$ are step functions (graphons that are constant on each block of an $m$-block partition), provided the constraints satisfy regularity conditions.

Our argument rests on the development a differential geometric framework for analyzing entropy functionals on finite dimensional regions $\realfeasibleregionml{m}$ defined by $h$-subgraph density constraints that generalize subgraph density constraints and where each point is a parametric representation of a step function of size $m$. 
The constrained regions $\realfeasibleregionml{m}$ are finite finite dimensional representation of $S(\mathcal{F},u)$ but with values on $\mathbb{R}$ instead of  values on $[0,1]$. 
We ensure the constrained regions are smooth manifolds for almost combinations of sufficient statistics $u$ by assuming the constraints are analytically independent. The step functions in these manifolds can be seen as points of a topological space either with a Euclidean topology or a function topology. 

Hence we prove that when $m$, the size of the parametric representation, is sufficiently large and $m$ is increased to $m+1$, the constraint manifolds do not develop new connected components (Theorem~\ref{thm:numbercomponents}). 
This result and the generalization of the entropy functional $-I$ as another $h$-subgraph density are sufficient to rule out new global optima in higher-dimensional spaces (Corollary \ref{cor:nonessentiallocal}). Thus the proof of RRS conjecture reduces the infinite-dimensional variational problem (\ref{eq:mainproblem}) to optimizing $I$ over the finite-dimensional manifold $\realfeasibleregionml{m}$ when $m$ is sufficiently large.

A further technical contribution closes a gap left open by this reduction: ruling out new global optima among \emph{other step functions} of growing size is not enough on its own, since a finite-dimensional reduction only proves the global theorem once the candidate solution is also known to be isolated against \emph{every} nearby graphon, step or not. We give a checkable, second-order sufficient condition for this (Theorem~\ref{thm:margincriterion}), built on an exact decoupling of the Lagrangian Hessian into a finite-dimensional ``in-manifold'' part and an infinite-dimensional ``off-manifold'' part (Lemma~\ref{lem:exactdecoupling}). The off-manifold part is controlled by a Cauchy--Schwarz bound that remains valid even when the constraint subgraphs have edges sharing a vertex (e.g.\ triangles, central to the RRS conjecture), where the relevant Hessian is not a pointwise/diagonal operator and naive estimates fail. We establish isolation simultaneously in the $L^1$ topology (Theorem~\ref{thm:isolatedgraphon}) and, under a global strong-convexity hypothesis satisfied by the entropy rate function (Remark~\ref{rem:I0stronglyconvex}), in the strictly weaker cut-norm topology (Theorem~\ref{thm:isolatedcutnorm}); we further show this isolation transfers to the unlabeled graphon space and its cut-distance topology (Corollary~\ref{cor:isolatedunlabeled}) -- precisely the form needed to complete the proof that entropy-maximizing multi-relational graphons are step functions (Theorem~\ref{thm:radin}).
  


\subsection{Organization}

The rest of the paper follows the logic of the argument rather than a strict topical order. Section~\ref{sec:mathematicalbacground} collects the \rfname{differential geometry} and constrained-optimization tools the proofs lean on, so that later sections can cite them rather than re-derive them. Section~\ref{sec:multirelationalgraphons} reviews the multi-relational graphon framework of \cite{alvarado2022limits}, and Section~\ref{sec:keyresultsgraphons} recalls how the most typical random multi-relational graph arises as the solution of an entropy-maximization problem; there we also generalize subgraph density to \emph{$h$-subgraph density}, modulated by an analytic function $h$, which lets the entropy rate function itself appear as a special case rather than as a separate object to be handled twice.

Section~\ref{sec:stepfunctionspace} develops the geometry of step functions inside the space of multi-relational graphons: the refinement operator and its inverse, the mismatch between the Euclidean and function topologies on this space, and the partial-derivative formulas for subgraph densities that everything afterward depends on.

Section~\ref{sec:computability} is where these pieces come together to prove that entropy-maximizing constrained multi-relational graphons are step functions, under the stated regularity conditions. It splits into a local and a global analysis. The local analysis supplies a checkable, second-order criterion for when a step function that is a local minimizer is genuinely \emph{isolated} -- in either the $L^1$ or the strictly weaker cut-norm topology. The global analysis then shows that once the step-function representation is large enough, enlarging it further never creates a new, better global solution. Section~\ref{sec:conclusions} closes with concluding remarks and open problems.

\section{Mathematical Background}
\label{sec:mathematicalbacground}

This section develops the mathematical foundations required for our analysis of constrained multi-relational graphons with maximum entropy. We present tools from Riemannian geometry, constrained optimization theory, and differential topology that enable us to study the structure of optimal solutions in finite-dimensional step function spaces. While the main application concerns graphon theory and random graph limits, the techniques we develop have broader applicability to constrained optimization problems on manifolds.

Readers familiar with constrained optimization on smooth manifolds may proceed directly to Section~\ref{sec:multirelationalgraphons}. For completeness, we provide here the essential results that underpin our proofs.

\subsection{Notation and Conventions}
\label{sec:prelim}

Throughout this paper, we adopt the following notational conventions:
\begin{itemize}
    \item $[n] = \{1, 2, \ldots, n\}$ denotes the set of positive integers up to $n$
    \item $\mathbbm{1}_X$ denotes the indicator function, i.e., for any $x$, if $x\in X$ then $\mathbbm{1}_X(x)=1$ else $\mathbbm{1}_X(x)=0$
    \item $\nu$ denotes the Lebesgue measure on $\mathbb{R}$
    \item $\Sigma$ denotes the collection of all bijective measure-preserving maps on $[0,1]$
    \item For a set $A$, we write $A^\circ$ for its interior, $\partial A$ for its boundary, and $|A|$ for its cardinality (when finite)
    \item A multi-relational graph $F$ is defined by $(V,E_1, \cdots, E_r)$ where $V$ is the set of vertices and $(E_1, \cdots, E_r)$ are the set of relations. The number of vertices is denoted $|F|$
    \item $\delta_{i,j}$ denotes the Kronecker delta: $\delta_{i,j} = 1$ if $i=j$ and $0$ otherwise    
    \item For a smooth function $F: \mathbb{R}^n \to \mathbb{R}^m$, $J_x(F)$ denotes its Jacobian matrix at $x$
    \item For a smooth function $f: \mathbb{R}^n \to \mathbb{R}$, $H_x f$ denotes its Hessian matrix at $x$
    \item For a smooth manifold $M$, $T_x M$ denotes the tangent space at $x \in M$
    \item $d\phi_x$ denotes the differential at $x$ of a local diffeomorphism $\phi_x:U \to N$  defined on neighborhood of $U \subset M$.
    \item $C^k(\mathbb{R})$, $k\in\mathbb{N}\cup\{\infty\}$, denotes the class of $k$ times continuously differentiable real functions.
\end{itemize}

\subsection{Riemannian Geometry Preliminaries}

We require several concepts from Riemannian geometry to analyze the structure of constraint manifolds and the behavior of optimization problems on these manifolds.

\subsubsection{Riemannian Manifolds and Metrics}

A \emph{Riemannian metric} on a smooth manifold $M$ is a smooth assignment of an inner product $g_p: T_p M \times T_p M \to \mathbb{R}$ to each tangent space $T_p M$, where $g_p$ is positive definite and bilinear. This structure enables us to measure lengths of curves and angles between tangent vectors.

\begin{definition}[Riemannian Metric]
\label{def:riemannian_metric}
Let $M$ be a smooth $n$-dimensional manifold. A \emph{Riemannian metric} $g$ on $M$ is a smooth tensor field that assigns to each point $p \in M$ a positive definite symmetric bilinear form $g_p: T_p M \times T_p M \to \mathbb{R}$.
\end{definition}

Given a Riemannian metric $g$, the length of a smooth curve $\gamma: [a,b] \to M$ is defined by
\begin{equation}
\label{eq:curve_length}
L(\gamma) = \int_a^b \sqrt{g_{\gamma(t)}(\gamma'(t), \gamma'(t))} \, dt.
\end{equation}

In local coordinates $(x^1, \ldots, x^n)$ on a chart, if $g_{ij}(p)$ denote the components of the metric tensor, then
\begin{equation}
L(\gamma) = \int_a^b \sqrt{\sum_{i,j=1}^n g_{ij}(\gamma(t)) \frac{dx^i}{dt} \frac{dx^j}{dt}} \, dt.
\end{equation}

\subsubsection{Geodesics and the Exponential Map}

Geodesics are curves that locally minimize distance and generalize the notion of straight lines in Euclidean space.

\begin{definition}[Definition $1.4.2$ in \cite{jost2008riemannian}]
Let $(M,g)$ be a Riemannian manifold. A smooth curve $\gamma: [0, a] \to M$ is called  a geodesic if it satisfies the second-order differential equation:
\begin{equation}
\label{geodesicequation}
\frac{d^2\gamma^{(i)}}{ds^2} + \Gamma^{i}{j,k} \frac{d\gamma^{(j)}}{ds} \frac{d\gamma^{(k)}}{ds} = 0, \quad i, j, k = 1, \ldots, m,
\end{equation}
where $\gamma^{(i)}$ denotes the $i$-th coordinate of $\gamma$ and $\Gamma^i_{jk}$ are the Christoffel symbols of the Levi-Civita connection.
\end{definition}

The Christoffel symbols are defined in terms of the metric tensor by
\begin{equation}
\Gamma^i_{jk} = \frac{1}{2} \sum_{\ell} g^{i\ell} \left( \frac{\partial g_{j\ell}}{\partial x^k} + \frac{\partial g_{k\ell}}{\partial x^j} - \frac{\partial g_{jk}}{\partial x^\ell} \right),
\end{equation}
where $g^{i\ell}$ denotes the components of the inverse metric tensor.

A fundamental tool in Riemannian geometry is the exponential map, which relates geodesics to the tangent space structure.

\begin{definition}[Exponential Map]
\label{def:exponential_map}
Let $(M,g)$ be a Riemannian manifold and $p \in M$. For each tangent vector $v \in T_p M$, let $\gamma_{p,v}$ be the unique geodesic with $\gamma_{p,v}(0) = p$ and $\gamma'_{p,v}(0) = v$. The \emph{exponential map} at $p$ is defined by
\begin{equation}
\mathrm{Exp}_p: U \subset T_p M \to M, \quad \mathrm{Exp}_p(v) = \gamma_{p,v}(1),
\end{equation}
where $U$ is an open neighborhood of $0 \in T_p M$ on which the geodesics are defined.
\end{definition}

A crucial property, as established by Proposition $20.8$ in \cite{lee2003smooth}, is that $Exp_p$ is a local diffeomorphism when restricted to an open neighborhood of $0 \in T_p M$. This property ensures the existence of a well-defined inverse map locally.

\begin{proposition}[Proposition $20.8$ in \cite{lee2003smooth}]
\label{prop:exp_local_diffeo}
For each $p \in M$, there exists an open neighborhood $U$ of $0 \in T_p M$ such that $\mathrm{Exp}_p: U \to \mathrm{Exp}_p(U)$ is a diffeomorphism.
\end{proposition}

\subsection{Constrained Optimization on Manifolds}

We now develop the theory of constrained optimization that forms the analytical core of our approach. Consider a smooth function $f: \mathbb{R}^n \to \mathbb{R}$ and constraint functions $h: \mathbb{R}^n \to \mathbb{R}^k$. Define the constraint set
\begin{equation}
\label{eq:constraint_set}
M = \{x \in \mathbb{R}^n : h(x) = 0\}.
\end{equation}

Under regularity conditions (the Jacobian $J_x(h)$ has full rank for all $x \in M$), the constraint set $M$ forms a smooth $(n-k)$-dimensional submanifold of $\mathbb{R}^n$.

\subsubsection{Lagrangian and First-Order Optimality Conditions}

\begin{definition}[Lagrangian Function]
\label{def:lagrangian}
The \emph{Lagrangian} of $f$ subject to the constraint $h(x) = 0$ is defined as
\begin{equation}
\mathcal{L}(x, \mu) = f(x) - \mu^\top h(x),
\end{equation}
where $\mu \in \mathbb{R}^k$ is the vector of Lagrange multipliers.
\end{definition}

\begin{theorem}[First-Order Necessary Conditions (KKT) (Theorem 12.1 in \cite{jorge2006numerical})]
\label{thm:first_order_KKT}
If $x^* \in M$ is a local minimum of $f$ on $M$, then there exists a unique vector $\mu^* \in \mathbb{R}^k$ of Lagrange multipliers such that
\begin{equation}
\label{eq:KKT_first_order}
\nabla_x \mathcal{L}(x^*, \mu^*) = 0,
\end{equation}
which is equivalent to
\begin{equation}
\nabla f(x^*) = (J_{x^*}(h))^\top \mu^*.
\end{equation}
\end{theorem}

\subsubsection{Hessian of the Lagrangian}

To characterize the nature of critical points (minima, maxima, or saddle points), we examine second-order information through the Hessian of the Lagrangian restricted to the tangent space of the constraint manifold.

\begin{definition}[Hessian of the Lagrangian]
\label{def:lagrangian_hessian}
Let $x^* \in M$ be a critical point of $f$ on $M$ with Lagrange multipliers $\mu^*$. The \emph{Hessian of the Lagrangian} at $x^*$ is the matrix
\begin{equation}
H^g_{x^*} f := H_{x^*} \mathcal{L}(\cdot, \mu^*) = H_{x^*} f - \sum_{i=1}^k \mu^*_i H_{x^*} h_i,
\end{equation}
where $H_{x^*} f$ is the Hessian of $f$ at $x^*$, and $h_i$ are the components of $h$.
\end{definition}

The superscript $g$ emphasizes that this is the Hessian restricted to the constraint manifold geometry.

\begin{theorem}[Second-Order Necessary Conditions (Theorem $12.5$ in \cite{jorge2006numerical})]
\label{thm:second_order_necessary}
Let $x^* \in M$ be a critical point of $f$ on $M$ with Lagrange multipliers $\mu^*$. If $x^*$ is a local minimum, then
\begin{equation}
\label{eq:second_order_necessary}
v^\top (H^g_{x^*} f) v \geq 0 \quad \text{for all } v \in T_{x^*} M.
\end{equation}
\end{theorem}

\begin{theorem}[Second-Order Sufficient Conditions (Theorem $12.6$ in \cite{jorge2006numerical})]
\label{thm:second_order_sufficient}
Let $x^* \in M$ be a critical point of $f$ on $M$ with Lagrange multipliers $\mu^*$. If
\begin{equation}
\label{eq:second_order_sufficient}
v^\top (H^g_{x^*} f) v > 0 \quad \text{for all } v \in T_{x^*} M \setminus \{0\},
\end{equation}
then $x^*$ is a strict local minimum of $f$ on $M$.
\end{theorem}

\subsection{Critical Values and Regular Values}

The notion of regular and critical values, fundamental in differential topology, plays an essential role in understanding when constraint manifolds are smooth and well-behaved.

\begin{definition}[Regular and Critical Points]
\label{def:regular_critical_points}
Let $F: M \to N$ be a smooth map between manifolds. A point $p \in M$ is a \emph{regular point} of $F$ if the differential $dF_p: T_p M \to T_{F(p)} N$ is surjective. Otherwise, $p$ is a \emph{critical point}. A value $c \in N$ is a \emph{regular value} if every point in $F^{-1}(c)$ is a regular point; otherwise, $c$ is a \emph{critical value}.
\end{definition}
\begin{theorem}[Sard's Theorem (Theorem $6.10$ in \cite{lee2003smooth})]
\label{thm:sardstheorem}
Let $M$ and $N$ be smooth manifolds (possibly with boundary), and let $F: M \to N$ be a smooth map. Then the set of critical values of $F$ has measure zero in $N$.
\end{theorem}
The preimage of regular values are smooth manifolds.
\begin{theorem}[Theorem $5.12$ in \cite{lee2003smooth}]
\label{cor:regular_value_manifold}
Let $F: M \to N$ be a smooth map and suppose $c \in N$ is a regular value. Then $F^{-1}(c)$ is a smooth submanifold of $M$ with
\begin{equation}
\dim F^{-1}(c) = \dim M - \dim N.
\end{equation}
\end{theorem}

\subsection{The Rank Theorem}
We will repeatedly need one more piece of local control: near a point where a smooth map has full rank, a change of coordinates puts the map into a simple canonical form. The rank theorem makes this precise.

\begin{definition}
Let $f: M \to N$ be a smooth map between smooth manifolds of dimensions $n$
and $m$ respectively, with $n \geq m$. The map $f$ is called a smooth submersion if the differential   $df_x: T_x M \to T_{f(x)}N$ is surjective at every point $x \in M$.    
\end{definition}

\begin{theorem}[Theorem $4.12$ in \cite{lee2003smooth}]
\label{thm:smoothsubmesion}
Let $F:M \to N$ be a smooth submersion. Then for each $p \in M$ there exist smooth charts $(U,\varphi)$ for $M$ centered at $p$ and $(V,\psi)$  for $N$ centered at $F(p)$ such that $F(U) \subseteq V$, in which $F$ has a coordinate representation of the form
\begin{equation*}
\hat{F}(x_1, \cdots, x_n,x_{n+1}, \cdots, x_m) = (x_1, \cdots, x_n)
\end{equation*}
where $\hat{F} = \psi \circ  F \circ \varphi^{-1}$. Hence $F(U)$ is open in $N$ when $U \subset M$ is open.
\end{theorem}

One immediate consequence is that a smooth submersion is an open map; if the submersion is also analytic, it carries the interior of its domain to the interior of its image, a fact we will use repeatedly to transport regularity from a step-function space of one size to another. 
\begin{definition}
Let $\{f_1, \cdots, f_m: \mathbb{R}^n \to \mathbb{R}\}$ be a set of analytic functions and $n \geq m$. We say that the functions are analytically independent  if $\text{rank}(J_x(f_1, \cdots, f_m)) = m$ for some $x$.     
\end{definition}
\begin{theorem}
\label{thm:denseinterior}
Let $\{f_1, \cdots, f_m: \mathbb{R}^n \to \mathbb{R}\}$   be real analytic and independent functions, with $n \geq m$. If $A^\circ$ is dense in $A$ then $f(A)^\circ$ is dense in $f(A)$. 
\end{theorem}

\begin{proof}
Since  $f$ is real analytic, the critical set $\text{Crit}(f) = \{x : \text{rank}(J_x(f)) < m\}$ is a real analytic variety, hence either all of $\mathbb{R}^n$ or a nowhere dense closed set (of dimension $\leq n-1$). Assume $f$ is not everywhere degenerate, then $\text{Crit}(f)$ is nowhere dense, and the regular set $R = \{x : \text{rank}(J_x(f)) = m\}$ is open and dense in $\mathbb{R}^n$. 

Hence by Theorem \ref{thm:smoothsubmesion}, $f|_R$ is an open map. 
Now take any $y \in f(A)$  and open $W \ni y$. We need to prove $W \cap f(A)^\circ \neq \emptyset$. 

$A^\circ \cap R$ is open and non-empty since $R$ is open and dense, thus $f(A^\circ \cap R) \subset f(A)^\circ$.  
$f^{-1}(W) \cap A^\circ \neq \emptyset$ since $A^\circ$  is dense in $A$. $f^{-1}(W) \cap R$ is open since $R$ is dense. Then it exists a $x' \in f^{-1}(W) \cap A^\circ \cap R$.
Since $f|_R$ is open, $f(x') \in f(A^\circ \cap R) \subset f(A)^\circ$, and $f(x') \in W$.

\end{proof}

\section{Multi relational graphons}
\label{sec:multirelationalgraphons}

We now recall the multi-relational graphon framework of \cite{alvarado2022limits}, which is the stage on which everything else in this paper takes place.

A multi-relational graph $G=(V, E_1, \cdots, E_r)$ packages $r$ ordinary graphs $(V,E_1), \ldots, (V,E_r)$ on a common vertex set $V$ into a single object, one relation $E_i$ per edge type.

Its continuum limit is an $r$-tuple of symmetric kernels rather than a single one: the space of $r$-relation graphons on the reals is
\begin{eqnarray*}
\realgraphonspacel = \{ W: [0,1]^2 \mapsto \mathbb{R}^r  :  W_k(x_1, x_2) =W_k(x_2,  x_1) \mbox{ for } k \in [r]\}
\end{eqnarray*}
and, more generally, for $A \subseteq \mathbb{R}$,
\begin{eqnarray*}
    \graphonspacel_{A}= \{  W:  [0,1]^2 \mapsto A^r  :  W_k(x_1,  x_2) =W_k(x_2, x_1) \mbox{ for } k \in [r]\}
\end{eqnarray*}
We call elements of $\graphonspacel_{(0,1)}$ and $\graphonspacel_{\mathbb{R}}$, respectively, \emph{purely random graphons} and \emph{real-valued graphons}; when $A$ is omitted, $A=[0,1]$.

As with ordinary graphons, the labeling of $[0,1]$ by itself carries no information: let $\sim$ be the equivalence relation on $\realgraphonspacel$ defined by $W \sim V$ iff there exists $\sigma \in \Sigma$ such that for all $k \in [r]$ and $(x_1, x_2) \in \mathbb{R}^2$ we have $W_k(x_1, x_2)=V_k(x_1, x_2)^\sigma$, where $V_k(x_1, x_2)^\sigma=V_k(\sigma(x_1), \sigma(x_2))$. Quotienting $\graphonspacel$ and $\realgraphonspacel$ by $\sim$ gives the spaces of \emph{unlabeled} graphons with $r$ relations, $\graphonspaceu=\graphonspacel/\sim$ and $\realgraphonspaceu=\realgraphonspacel/\sim$.

\noindent We topologize $\realgraphonspacel$ by the cut-norm
\begin{equation*}
    \| W \|_\Box = \sum_{k=1}^r \sup_{S_1, S_2 \subset [0,1]} \left|  \int_{S_1 \times S_2 } W_k(x_1, x_2) dx_1 dx_2 \right|
\end{equation*}
and $\realgraphonspaceu$ by the corresponding cut-distance
\begin{equation*}
    \delta_\Box(W,V) = \inf_{\sigma \in \Sigma} \| W-V^\sigma \|_\Box
\end{equation*}

\subsection{Subgraph density}

For a multi-relational graph $F$ and a multi-relational graphon $W$, the subgraph density of $F$ in $W$ is
\begin{equation*}
\label{eq:subgraphdensity}
    t(F,W) =  \int_{[0,1]^{|V(F)|}} \prod_{k=1}^r \prod_{(i_1, i_2) \in E_k(F)} W_k(x_{i_1},x_{i_2}) \prod_{i \in V(F)} dx_i
\end{equation*}
One generalization will turn out to be essential throughout the paper: attaching an analytic function to each relation before multiplying. Let $F$ be an $r$-graph, $W$ an $r$-graphon, and $h=\{h_i:\mathbb{R} \to \mathbb{R}\}_{i=1}^r$ a set of real analytic functions. The \emph{$h$-subgraph density} of $F$ modulated by $h$ -- or simply $h$-subgraph density of $F$ -- is
\begin{equation*}
\label{eq:multi-relational-graphdensity}
    t(F,W,h) =  \int_{[0,1]^{|V(F)|}} \prod_{k=1}^r \prod_{(i_1, i_2) \in E_k(F)} h_k(W_k(x_{i_1},x_{i_2})) \prod_{i \in V(F)} dx_i
\end{equation*}
When $h$ is the identity, we omit it from the notation. This single generalization is what lets the entropy rate function reappear, later in the paper, as just another $h$-subgraph density rather than as a special case requiring its own machinery.


\subsection{The most typical random multi-relational graphons}
\label{sec:keyresultsgraphons}

The key results of graphon theory -- compactness, the counting lemma, the large deviation principle -- extend to multi-relational graphs \cite{alvarado2022limits}. We restate the result we need: that the most typical random multi-relational graph arises as the solution of a maximum entropy problem.
\begin{theorem}[Theorem $8$ in \cite{alvarado2022limits}]
\label{thm:limitmultigraphons}
Let $\mathcal{F}$ be a finite vector of multi-relational quantum graphs
$\mathcal{F} = [\mathcal{F}_1,\cdots, \mathcal{F}_k]$  and let $u \in \mathbb{R}^k$ be a vector of multi-relational subgraph densities.

Then the limit $\optimalsolution$ of the sequence $(G^{[r]}(n, 1/2))^\infty_{n=1}$
 of growing random multi-relational graphs $G^{[r]}(n, 1/2)$, conditioned on $\lim_{n \to \infty} G^{[r]}(n, 1/2) \in S(\mathcal{F}, u)$ (limit in the topology of the unlabeled multi-relational graphon space, see \cite{alvarado2022limits}), satisfies
\[
\optimalsolution = \argmin_{W \in S^{(r)}(\mathcal{F}, u)} I(W)
\]
where $I(W) = \sum_k \int_{[0,1]^2} I_0(W_k(x,y)) dxdy$ and $I_0(u) = \frac{1}{2} u\log(u) + \frac{1}{2} (1-u) \log(1-u)$ and
\[
S^{(r)}(\mathcal{F}, u) = \{W \in \graphonspaceu \, :  \, t(\mathcal{F},W)=u\}
\]
\end{theorem}
This is the multi-relational form of the RRS conjecture: that the entropy-maximizing solutions $\optimalsolution$ of this problem are step functions.

\subsection{Quantum graphs}

One last piece of bookkeeping makes the second-derivative computations of Section~\ref{sec:stepfunctionspace} far less cumbersome: allowing formal linear combinations of graphs.

\begin{definition}[Quantum graphs]
A quantum graph $F$ is a finite linear combination of multi-relational graphs $F_i$ with real coefficients,
\begin{equation*}
    F = \sum_i \alpha_i F_i \quad  \mbox{ and } \alpha_i \in \mathbb{R}.
\end{equation*}
The multi-relational graphs $F_i$ are the \emph{constituents} of the quantum graph, and $t(F,W,h)$ extends to quantum graphs linearly: $t(F,W,h) = \sum_i \alpha_i t(F_i,W,h)$.
\end{definition}

\section{Step functions in $\realgraphonspacel$}
\label{sec:stepfunctionspace}

Step functions are the finite-dimensional skeleton on which the rest of the paper is built: every optimization we run takes place on a space of step functions of some fixed size $m$, and the entropy-maximization problem is solved once we know that, for $m$ large enough, this finite-dimensional picture already contains the global optimum. We now set up that skeleton precisely.

Let $\lebesgue$ be the Lebesgue measure on $\mathbb{R}$, and let $S=\{S_1, \cdots, S_m\}$ be a partition of $[0,1]$ into $m$ intervals. Let $\convexcombm{m}$ be the $m$-simplex of probability vectors of length $m$,
\begin{equation*}
\convexcombm{m} = \{ \pi \in [0,1]^m  :  \sum_i \pi_i = 1 \}.
\end{equation*}
\noindent Taking $\pi \in \convexcombm{m}$ with $\nu(S_i)=\pi_i$, we call $\pi$ the \emph{partition vector} of the step function it represents.

Let $\mathbb{R}_{sym}^{m \times m}$ denote the symmetric $m \times m$ matrices and $\mathbb{R}_{sym}^{(m \times m)^r}$ the Cartesian product of $r$ copies of $\mathbb{R}_{sym}^{m \times m}$. Given $\sfmat \in \mathbb{R}_{sym}^{(m \times m)^r}$ and $\pi \in \convexcombm{m}$, the corresponding \emph{$m$-step function} $(\sfmat,\pi):[0,1]^2 \mapsto \mathbb{R}^r$ is constant on each block $S_i \times S_j$, in every coordinate $k \in [r]$:
\begin{equation*}
(\sfmat,\pi)_k =\sum_{i,j=1}^m \sfmat_{i,j; k}
 \mathbbm{1}_{S_i \times S_j},
\end{equation*}
where $\sfmat_{i,j; k}$ is the $(i,j)$ entry of relation $k$. We denote the space of $r$-multi-relational $m$-step functions by $\realpodalfunctionspacel{m}$; it is naturally identified with $\mathbb{R}^{r\frac{m(m-1)}{2}} \times \convexcombm{m}$, so that $\partial \realpodalfunctionspacel{m+1} = \realpodalfunctionspacel{m}$ for $m \geq 1$.

It will also be convenient to freeze $\pi$ and let only $\sfmat$ vary: write $\realpodalfunctionspacelpi{m}{\pi}= \mathbb{R}_{sym}^{(m \times m)^r} \times \{ \pi \}$ for the resulting slice, with $r \frac{m(m-1)}{2}$ parameters against the $r \frac{m(m-1)}{2}+m-1$ of the full $\realpodalfunctionspacel{m}$.

On $\podalfunctionspacel{m}$, the integral defining $t(F,W,h)$ collapses to a finite sum,
\begin{equation*}
t(F,(\sfmat,\pi),h) = \sum_{\tiny x_1, \cdots, x_{|V(F)|}=1}^m \prod_{k=1}^r  \prod_{(s, t) \in E_k(F)} h_k(\sfmat_{s,t;k }) \prod_{i \in V(F)} \pi_i
\end{equation*}
which, writing $E(F) = \cup_k E_k(F) \times \{k\}$ for the edges of $F$ tagged by relation, we will often abbreviate as
\begin{equation*}
t(F,(\sfmat,\pi),h) = \sum_{\tiny x_1, \cdots, x_{|V(F)|}=1}^m   \prod_{(s,t;k) \in E(F)} h_k(\sfmat_{x_s x_t;k}) \prod_{i \in V(F)} \pi_i.
\end{equation*}

\subsection{The refinement operator}
\label{sec:refinementoperator}
The same step function admits many representations -- coarser ones with fewer blocks, finer ones with more -- and the size $m$ we eventually need to certify a global optimum will keep changing throughout the paper. We therefore fix, once and for all, an operator that moves between representations of a step function without changing the function itself: \emph{refinement} splits one block into two.
\begin{definition}[The refinement operation on step functions]
\label{def:refinementoperator}
Let $\lambda \in (0,1)$ and $l \in [m]$. We denote by $\theta(\cdot,\lambda,l):\realpodalfunctionspacel{m} \to \realpodalfunctionspacel{m+1}$ the map defined by $\theta((\sfmat,\pi), \lambda,l)=(\theta(\sfmat,l),\theta(\pi,\lambda,l))$, where $\theta(\sfmat,l)\in \mathbb{R}_{sym}^{(m+1 \times  m+1)^r}$ is obtained by duplicating the $l$-th row/column of each coordinate of $\sfmat$. Precisely, the entry $\theta_{i,j;k}(\sfmat,l)$ is
\begin{equation*}
\theta_{i,j;k}(\sfmat,l) = \left\lbrace
\begin{array}{c c}
\sfmat_{i,j;k} & 1 \leq i,j \leq l \\
\sfmat_{l,j;k} & i=l+1, j \leq l \\
\sfmat_{i,l;k} & i \leq l, j=l+1 \\
\sfmat_{l,l;k} & i=l+1, j=l+1 \\
\sfmat_{i-1,j-1;k} & l+1  < i,j \leq m \\
\sfmat_{i,j-1;k} & 1 \leq i \leq l, l+1  < j \leq m \\
\sfmat_{i-1,j;k} &  l+1  < i \leq m, 1 \leq j \leq l \\
\end{array}
\right.
\end{equation*}
and the entry $\theta_i(\pi,\lambda,l)\in \convexcombm{m+1}$ is
\begin{equation*}
\theta_i(\pi,\lambda,l) = \left\lbrace
\begin{array}{c c}
\pi_{i} & i < l \\
\lambda \pi_l & i = l \\
(1-\lambda) \pi_l & i = l+1 \\
\pi_{i-1} &  l + 1 < i \leq m+1
\end{array} \right.
\end{equation*}
In words, $\theta(\cdot, \lambda, k)$ splits the $k$-th row/column of $(\sfmat,\pi)$ into two rows/columns of weights $\lambda \pi_k$ and $(1-\lambda) \pi_k$, in every coordinate of $(\sfmat,\pi)$.
\end{definition}
\begin{figure}[h]
\centering
\begin{tikzpicture}
\node[anchor=north west, inner sep=0pt] at (0,0) {
\(
\begin{array}{c}
(\sfmat, \pi) \\
\begin{array}{c|ccc|}
      & \pi_1 & \pi_2 & \pi_3 \\
\hline
\pi_1 & A_{11} & A_{12} & A_{13} \\
\pi_2 & A_{12} & A_{22} & A_{23} \\
\pi_3 & A_{13} & A_{23} & A_{33} \\
\hline
\end{array}
\end{array}
\quad
\Rightarrow
\quad
\begin{array}{c}
\theta((\sfmat, \pi), \lambda, 3) \\
\begin{array}{c|cccc|}
      & \pi_1 & \pi_2 & \lambda \pi_3 & (1-\lambda)\pi_3 \\
\hline
\pi_1 & A_{11} & A_{12} & A_{13} & A_{13} \\
\pi_2 & A_{12} & A_{22} & A_{23} & A_{23} \\
\lambda \pi_3 & A_{13} & A_{23} & A_{33} & A_{33} \\
(1-\lambda)\pi_3 & A_{13} & A_{23} & A_{33} & A_{33} \\
\hline
\end{array}
\end{array}
\)
};
\end{tikzpicture}
\label{fig:refinementmap}
\caption{A refinement operation transforms the representation of a $3$-step function $(\sfmat, \pi)$ into a $4$-step function $\theta((\sfmat,\pi),\lambda,3)$.}
\end{figure}
Crucially, refinement is purely a change of representation: as a function $(\sfmat,\pi):[0,1]^2 \to \mathbb{R}^r$, nothing changes -- only the number of parameters used to describe it grows.

\subsubsection{The inverse of the refinement operations}

Refinement also runs in reverse. If $(\sfmat,\pi)$ has two identical rows, it is itself the refinement of some smaller step function $(\sfmat',\pi')$, obtained by merging those rows back together; we now make this \emph{coarsening} operation precise.

Define an equivalence relation on the row/column indices of $\sfmat$ by $i \sim_A j \iff \sfmat_{i,\cdot;k} = \sfmat_{j,\cdot;k}$ for all $k \in [r]$, and let $[i]_{\sim_\sfmat}$ be the equivalence class of the $i$-th row/column, ordered so that $[i]_{\sim_\sfmat}[1]$ is the position of its first representative. Let $S = \{[1]_{\sim_\sfmat}[1], \cdots, [m]_{\sim_\sfmat}[1]\}$ collect these first positions, and let $rk(j)$ be the rank of $[j]_{\sim_\sfmat}$ in $S$. The coarsening operator $\psi:\realpodalfunctionspacel{m} \to \cup_{m'=1}^m \realpodalfunctionspacel{m'}$ then acts entrywise by
\[
\psi_{rk(i),rk(j);k}((\sfmat,\pi)) = \sfmat_{[i]_{\sim_\sfmat}[1],[j]_{\sim_\sfmat}[1];k} \mbox{ for all } i,j \in [m] \mbox{ and } k \in [r]
\]
and
\[
\psi_{rk(i)}(\pi) = \sum_{j \in [i]_{\sim_\sfmat}} \pi_j.
\]
By construction, $\psi((\sfmat,\pi))$ has no repeated rows and represents the same step function $(\sfmat,\pi)$ in its minimal-size form. We write $\#((\sfmat,\pi))$ for the size of $\psi((\sfmat,\pi))$, i.e.\ the number of non-repetitive row/columns the step function actually has.

\subsection{Partial Derivatives of $t(F, (\sfmat,\pi),h)$}

The local and global analyses of Sections~\ref{sec:computability} and \ref{sec:caseI} both turn, at bottom, on differentiating $t(F,\cdot,h):\realpodalfunctionspacel{m} \to \mathbb{R}$. To state the derivative formulas cleanly, we first introduce \emph{labeled} multi-relational graphs and the partial, $h$-modulated subgraph densities they index.
 \subsubsection{Labeled Graphs and Partial Subgraph Densities}
\begin{definition}[Labeled Graphs]
\label{def:labeledgraphs}
Let $F^{\bullet(a b),k}$ be the $2$-labeled graph obtained from $F$ by deleting the edge $(a b)$ in $E_k(F)$ and labeling the vertices $\{a, b\}$, and let $F^{\bullet a}$ denote the $1$-labeled graph obtained from $F$ by labeling the vertex $a$. The $h$-subgraph density of these labeled graphs is
\begin{eqnarray*}
     t_{i j; k } (F^{\bullet(a  b),k},(\sfmat,\pi),h) &=& \sum_{\{x_s\}_{s\notin\{a,b\}} } h_k'(\sfmat_{x_a, x_b;k}) \prod_{(s, t;u )  \in E(F) \setminus \{(a,b;k)\}}   h_k(\sfmat_{x_s, x_t,u}) \\&& \prod_{k \in V(F) \setminus \{a, b\}} \pi_{x_k}
\end{eqnarray*}
where $x_a = i$ and $x_b = j$ are the fixed vertex indices corresponding to the labeled vertices $a$ and $b$ respectively, and the sum runs over all remaining vertex assignments $\{x_s : s \in V(F) \setminus \{a, b\}\}$,
and
\begin{equation*}
    t_i(F^{\bullet a},(\sfmat,\pi),h) = \sum_{\{x_s\}_{s\notin\{a\}} } \prod_{(a, b;k)  \in E(F) } h_k(\sfmat_{x_a, x_b;k}) \prod_{k \in V(F) \setminus \{a\}} \pi_{x_k}.
\end{equation*}
\end{definition}
Note that $t_{i j; k } (F^{\bullet(ab),k},(\sfmat,\pi),h)=t_{i j; k } (F^{\bullet(ba),k},(\sfmat,\pi),h)$, since relabeling the two marked vertices just relabels the corresponding indices. We will also need a shorthand for linear combinations of these partial densities, one term per edge of a fixed relation.
\begin{definition}[Linear combination of partial subgraph densities.]
  \label{def:partialbullets}
Let $F$ be a multi-relational graph and $k \in [r]$. We write
\begin{equation*}
\partial_k^{(\bullet \bullet)} F =  \sum_{(a,b) \in E_k(F)} F^{\bullet(a b),k}
\end{equation*}
and
\begin{equation*}
    \partial^{ \bullet} F = \sum_{a \in V(F)} F^{\bullet a}.
\end{equation*}
\end{definition}
\subsubsection{Partial Derivatives of $t(F,(\sfmat,\pi),h)$ in $\realpodalfunctionspacel{m}$}

\begin{restatable}{theorem}{thmpartialderivatives}
\label{thm:partialderivatives}
Let $(\sfmat,\pi) \in \realpodalfunctionspacel{m}$.
Then the first partial derivative of $t(F,\cdot, h)$ is
\begin{eqnarray}
\label{eq:1partialderstep}
\frac{\partial t(F, \cdot, h)}{\partial \sfmat_{i, j;k}}  &&= \pi_i \pi_j (2-\delta_{i,j}) t_{ij;k}(\partial_k^{(\bullet \bullet)}F, \cdot,h)
\end{eqnarray}
where $\delta_{ij}$ is the Kronecker delta.
\end{restatable}

\begin{proof}
In the next lines, we abbreviate $t(F,(\sfmat, \pi),h)$ as $t(F)$. For the first partial derivative, we compute
\begin{eqnarray*}
\frac{\partial t(F, \cdot, h)}{\partial \sfmat_{i, j;k}} &=& \sum_{x_1, \cdots,x_{|F|}=1}^m   \frac{\partial }{\partial \sfmat_{i j;k}} \left\lbrace  \prod_{(s,t;l ) \in E(F)} h_k(\sfmat_{x_{s} x_{t};l} ) \right\rbrace \prod_{s \in V(F)} \pi_{x_s} \\
&=& \sum_{x_1, \cdots,x_{|F|}=1}^m \left\lbrace  \sum_{(a, b;l) \in E(F)} \frac{\partial h_k(\sfmat_{x_a, x_b;l}) }{\partial \sfmat_{i, j; k}}  \prod_{(s,t;u) \in E(F) \setminus \{(a,b;l)\}} h_k(\sfmat_{x_s  x_t;u}) \right \rbrace \\ && \prod_{s \in V(F)} \pi_{x_s}  \\
&=& \sum_{x_1, \cdots,x_{|F|}=1}^m \left\lbrace  \sum_{(a, b;l) \in E(F)} h'_k( \sfmat_{x_a x_b; l} ) \frac{\partial \sfmat_{x_a x_b;l } }{\partial \sfmat_{i, j;k}}   \left. \prod_{(s,t;u) \in E(F) \setminus \{(a,b;l)\}} h_k(\sfmat_{x_s x_t ;u}) \right \rbrace \right. \\ && \prod_{s \in V(F)} \pi_{x_s}
\end{eqnarray*}
Note that
\[
\frac{\partial \sfmat_{x_a x_b;l } }{\partial \sfmat_{i, j;k}} = \delta_{l,k}(\delta_{x_a,i}\delta_{x_b,j}  + (1-\delta_{i,j})\delta_{x_a,j}\delta_{x_b,i} )
\]
Substituting this in,
\begin{eqnarray*}
\frac{\partial t(F, \cdot, h)}{\partial \sfmat_{i, j;k}} &=& \sum_{x_1, \cdots,x_{|F|}=1}^m \left\lbrace  \sum_{(a, b;l) \in E(F)} h'_k( \sfmat_{x_a x_b; l} )
\delta_{l,k}(\delta_{x_a,i}\delta_{x_b,j}+(1-\delta_{i,j})\delta_{x_a,j}\delta_{x_b,i}  ) \right.\\ && \left.
 \prod_{(s,t;u) \in E(F) \setminus \{(a,b;l)\}} h_k(\sfmat_{x_s x_t ;u}) \right \rbrace   \prod_{s \in V(F)} \pi_{x_s}
\end{eqnarray*}
which we split into the two terms of the indicator sum:
\begin{eqnarray*}
\frac{\partial t(F, \cdot, h)}{\partial \sfmat_{i, j;k}}  &&=
\sum_{(a, b;l) \in E(F)} \left\lbrace \sum_{x_1, \cdots,x_{|F|}=1}^m  h'_k( \sfmat_{x_a x_b; l} )
\delta_{x_a,i}\delta_{x_b,j} \delta_{l,k}
 \prod_{(s,t;u) \in E(F) \setminus \{(a,b;l)\}} h_k(\sfmat_{x_s x_t ;u})  \right. \\ &&
+\left. (1-\delta_{i,j})\sum_{x_1, \cdots,x_{|F|}=1}^m  h'_k( \sfmat_{x_a x_b; l} )
\delta_{x_a,i}\delta_{x_b,j} \delta_{l,k}
 \prod_{(s,t;u) \in E(F) \setminus \{(a,b;l)\}} h_k(\sfmat_{x_s x_t ;u}) \right \rbrace  \prod_{s \in V(F)} \pi_{x_s}
\end{eqnarray*}
Evaluating the Kronecker deltas $\delta_{x_a,i}\delta_{x_b,j} \delta_{l,k}$ and $\delta_{x_a,i}\delta_{x_b,j} \delta_{l,k}(1-\delta_{i,j})$ collapses the sum over $x_1,\ldots,x_{|F|}$ to a sum over the remaining vertices only, and restricts $(a,b;l)$ to $E_k(F)$:
\begin{eqnarray*}
\frac{\partial t(F, \cdot, h)}{\partial \sfmat_{i, j;k}}  &&= \pi_i \pi_j
\sum_{(a, b) \in E_k(F)} \left\lbrace \sum_{\{x_s\}_{s\notin\{a,b\}} }^m  h'_k( \sfmat_{x_a x_b; k} )
 \prod_{(s,t;u) \in E(F) \setminus \{(a,b;k)\}} h_k(\sfmat_{x_s x_t ;u})  \right. \\ &&
+\left. (1-\delta_{i,j})\sum_{\{x_s\}_{s\notin\{a,b\}} }^m  h'_k( \sfmat_{x_a x_b; k} )
 \prod_{(s,t;u) \in E(F) \setminus \{(a,b;k)\}} h_k(\sfmat_{x_s x_t ;u}) \right \rbrace  \prod_{s \in V(F)} \pi_{x_s}
\end{eqnarray*}
which, by Definition~\ref{def:labeledgraphs} and Definition~\ref{def:partialbullets}, is exactly
\begin{eqnarray*}
\frac{\partial t(F, \cdot, h)}{\partial \sfmat_{i, j;k}}  &&= \pi_i \pi_j
( t_{ij;k}(\partial^{(\bullet \bullet)}F, (\sfmat,\pi),h) + (1-\delta_{i,j}) t_{ji;k}(\partial^{(\bullet \bullet)}F, (\sfmat,\pi),h) ) \\
  &&= \pi_i \pi_j (2-\delta_{i,j})   t_{ij;k}(\partial^{(\bullet \bullet)}F, (\sfmat,\pi),h).
\end{eqnarray*}
\end{proof}

\section{The solutions of $\optimalsolution$}
\label{sec:computability}

We now prove the RRS conjecture: the solutions of $\optimalsolution$ are step functions, provided certain regularity conditions hold. The proof rests on the principle of optimization of density functions stated in Section~\ref{sec:pod} -- but before we can state that principle, we need a few more definitions.


\subsection{Marginal Polytope}

\begin{definition}[Marginal map]
Let $\mathcal{F}$ be an ordered set of quantum graphs and let $h$ be a $|\mathcal{F}| \times r$ matrix of real analytic functions. The marginal map $t(\mathcal{F},\cdot, h):\realgraphonspacel \to \mathbb{R}^{|\mathcal{F}|}$ is defined by
\begin{equation*}
    t(\mathcal{F},W,h) = (t(\mathcal{F}_1,W,h_1), \cdots, t(\mathcal{F}_n,W,h_n)  ),
\end{equation*}
where $h_i$ denotes row $i$ of $h$. We write $\{t(\mathcal{F}, \cdot, h)\}$ for the set of subgraph densities derived from $\mathcal{F}$ and $h$, i.e.\ $\{t(\mathcal{F}_1, \cdot, h_1), \cdots, t(\mathcal{F}_n, \cdot, h_n) \}$, and $J_{(\sfmat,\pi)}(\mathcal{F},h)$ for the Jacobian of $t(\mathcal{F},\cdot,h)$. When $h$ is a matrix of identity functions, we omit it.
\end{definition}

The Jacobian $J_{(\sfmat,\pi)}(\mathcal{F},h)$ has a property we will lean on repeatedly: whether it is full rank does not depend on which representation of a step function we use.

\begin{restatable}{lemma}{lemjacobianrefinement}
\label{lem:jacobianrefinement}
Let $t(\mathcal{F}, \cdot,h): \realpodalfunctionspacel{m} \to \mathbb{R}^{|\mathcal{F}|}$ and $(\sfmat,\pi) \in \realpodalfunctionspacel{m}$. Then $J_{(\sfmat,\pi)} (\mathcal{F},h)$  is full rank iff $J_{\theta((\sfmat,\pi), \lambda,k)} (\mathcal{F},h)$ is full rank for any $k \in [m]$ and $\lambda \in (0,1)$.
\end{restatable}

\begin{proof}
Since the step functions are symmetric under permutation of row/columns, we may assume $k=m$. We compute the Jacobian matrix of $t(\mathcal{F}, \cdot, h)$ at $(\sfmat,\pi)$; it has size $|\mathcal{F}| \times r\frac{m(m-1)}{2}$.
\begin{equation*}
J_{(\sfmat,\pi)}(\mathcal{F},h) = \begin{pmatrix}
\frac{\partial t(\mathcal{F}_1,\cdot,h)}{\partial \sfmat_{1,1;1}}|_{(\sfmat,\pi)} & \frac{\partial t(\mathcal{F}_1,\cdot,h)}{\partial \sfmat_{1,2;1}}|_{(\sfmat,\pi)} & \cdots & \frac{\partial t(\mathcal{F}_1,\cdot,h)}{\partial \sfmat_{m,m;r}}|_{(\sfmat,\pi)} \\
\frac{\partial t(\mathcal{F}_2,\cdot,h)}{\partial \sfmat_{1,1;1}}|_{(\sfmat,\pi)} & \frac{\partial t(\mathcal{F}_2,\cdot,h)}{\partial \sfmat_{1,2;1}}|_{(\sfmat,\pi)} & \cdots & \frac{\partial t(\mathcal{F}_2,\cdot,h)}{\partial \sfmat_{m,m;r} }|_{(\sfmat,\pi)} \\
\vdots & \vdots & \ddots & \vdots \\
\frac{\partial t(\mathcal{F}_n,\cdot,h)}{\partial \sfmat_{1,1;1}}|_{(\sfmat,\pi)} & \frac{\partial t(\mathcal{F}_n,\cdot,h)}{\partial \sfmat_{1,2;1}}|_{(\sfmat,\pi)} & \cdots & \frac{\partial t(\mathcal{F}_n,\cdot,h)}{\partial \sfmat_{m,m;r} }|_{(\sfmat,\pi)}
\end{pmatrix}
\end{equation*}
Using the partial-derivative formula $\frac{\partial t(\mathcal{F}_s,\cdot,h)}{\partial \sfmat_{i,j;k}}$ of (\ref{eq:1partialderstep}),
\begin{equation*}
J_{(\sfmat,\pi)}(\mathcal{F},h) = \begin{pmatrix}
\pi^2_1 g_{1,1;1,1}((\sfmat,\pi))
 & \pi_1 \pi_2 g_{1,2;2,1}((\sfmat,\pi))  & \cdots & \pi^2_m g_{m,m;r,1}((\sfmat,\pi))  \\
\pi^2_1 g_{1,1;1,2}((\sfmat,\pi)) & \pi_1 \pi_2 g_{1,2;2,2}((\sfmat,\pi))  & \cdots & \pi^2_m g_{m,m;r,2}((\sfmat,\pi))  \\
\vdots & \vdots & \ddots & \vdots \\
\pi^2_1 g_{1,1;1,n}((\sfmat,\pi))  & \pi_1 \pi_2 g_{1,2;2,n}((\sfmat,\pi)) & \cdots & \pi^2_m g_{m,m;r,n}((\sfmat,\pi))
\end{pmatrix}
\end{equation*}
where $g_{i,j;k,s}((\sfmat,\pi))=(2-\delta_{i,j})t_{ij;k}\partial^{(\bullet \bullet)}_k\mathcal{F}_s, (\sfmat,\pi),h)$, so $J_{(\sfmat,\pi)}(\mathcal{F},h)$ is full rank iff $J_m$ below is full rank.
\begin{equation*}
J_m = \begin{pmatrix}
 g_{1,1;1,1}((\sfmat,\pi))
 & g_{1,2;2,1}((\sfmat,\pi))  & \cdots &  g_{m,m;r,1}((\sfmat,\pi))  \\
 g_{1,1;1,2}((\sfmat,\pi)) &  g_{1,2;2,2}((\sfmat,\pi))  & \cdots &  g_{m,m;r,2}((\sfmat,\pi))  \\
\vdots & \vdots & \ddots & \vdots \\
 g_{1,1;1,n}((\sfmat,\pi))  &  g_{1,2;2,n}((\sfmat,\pi)) & \cdots &  g_{m,m;r,n}((\sfmat,\pi))
\end{pmatrix}
\end{equation*}
Now compute $J_{\theta((\sfmat,\pi), \lambda,m)}(\mathcal{F},h)$, dropping the factors $\pi_i \pi_j$ that accompany $\frac{\partial t(\mathcal{F}_s,\cdot,h)}{\partial \sfmat_{i,j;k}}$, and abbreviate $\theta((\sfmat,\pi),\lambda,m)$ by $\theta((\sfmat,\pi),\lambda)$:
\begin{equation*}
J_{m+1} = \begin{pmatrix}
 g_{1,1;1,1}(\theta((\sfmat,\pi),\lambda))
 & g_{1,2;2,1}(\theta((\sfmat,\pi),\lambda))  & \cdots &  g_{m+1,m+1;r,1}(\theta((\sfmat,\pi),\lambda))  \\
 g_{1,1;1,2}(\theta((\sfmat,\pi),\lambda)) &  g_{1,2;2,2}(\theta((\sfmat,\pi),\lambda))  & \cdots &  g_{m+1,m+1;r,2}(\theta((\sfmat,\pi),\lambda))  \\
\vdots & \vdots & \ddots & \vdots \\
 g_{1,1;1,n}(\theta((\sfmat,\pi),\lambda))  &  g_{1,2;2,n}(\theta((\sfmat,\pi),\lambda)) & \cdots &  g_{m+1,m+1;r,n}(\theta((\sfmat,\pi),\lambda))
\end{pmatrix}
\end{equation*}
Since $g_{i,j;k,s}((\sfmat,\pi)) = g_{i,j;k,s}(\theta((\sfmat,\pi),\lambda)$ for every $\lambda \in [0,1]$, the columns of $J_{m+1}$ simply repeat those of $J_m$, so $J_m$ is full rank iff $J_{m+1}$ is full rank.
\end{proof}

With the marginal map in hand, we can name the set it sweeps out, and the regions it cuts out of the graphon space.

\begin{definition}[Marginal Polytope]
The marginal polytope $\marginalpolytopemh{m}$ of $\mathcal{F}$ is the image of $t(\mathcal{F},\cdot, h)$ on 
$\podalfunctionspacel{m}$, i.e., $\marginalpolytopemh{m}= t(\mathcal{F},\podalfunctionspacel{m},h)$. We also write $\marginalpolytopeh=t(\mathcal{F},\graphonspacel,h)$ for the total marginal polytope.

\end{definition}

\begin{definition}[Constrained regions]
For $u \in \marginalpolytopeh$, the level sets of $t(\mathcal{F}, \cdot,h)$ are exactly the regions constrained by the subgraph densities $u$:
\begin{equation*}
    \realfeasibleregionlh{h} = t^{-1}(\mathcal{F},u,h) = \{W \in \realgraphonspacel \, :  \, t(\mathcal{F},W,h)=u\}
\end{equation*}
and
\begin{equation*}
\realfeasibleregionmlh{m}{h} = \realfeasibleregionlh{h} \cap \realpodalfunctionspacel{m} 
\end{equation*}
and, fixing $\pi \in \convexcombm{m}$,
\begin{equation*}
\realfeasibleregionmlpi{m}{\pi, h} =  \realfeasibleregionmlh{m}{h} \cap \realpodalfunctionspacelpi{m}{\pi}  
\end{equation*}
\end{definition}

\subsubsection{Non-empty interior of $\marginalpolytopemh{m}$}

These constrained regions are only as well behaved as the marginal polytope they sit inside: to guarantee that $\realfeasibleregionmlh{m}{h}$ is an analytic manifold for almost every $u \in \marginalpolytopemh{m}$, we need $u$ to be a regular value of $t(\mathcal{F}, \cdot, h)$ (Sard's theorem, Theorem~\ref{thm:sardstheorem}, supplies "almost every" once $\marginalpolytopemh{m}$ has non-empty interior), and a non-empty interior is in turn guaranteed once the functions $\{t(\mathcal{F}_i, \cdot,[h_i]): \realpodalfunctionspacel{m} \to \mathbb{R} \}$ are analytically independent.

\begin{definition}[Independent analytical functions]
Let $AF=\{f_1, \cdots, f_n\}$ be an ordered, finite set of analytic functions $f_i:\realpodalfunctionspacel{m} \to \mathbb{R}$. We say $AF$ is a set of \emph{independent analytic functions} if, for every $m$ with $r\frac{m(m+1)}{2}  > |AF|$, there is some $(\sfmat,\pi) \in \realpodalfunctionspacel{m}$ with $\pi \in (\convexcombm{m})^\circ$ such that $J_{(\sfmat,\pi)}(f_1,\cdots, f_n)$ is full rank.
\end{definition}

\subsection{Smoothness of $\realfeasibleregionmlh{m}{h}$}
\label{sec:smoothness}

This section establishes sufficient conditions guaranteeing that $\realfeasibleregionmlh{m}{h}$ is a smooth manifold.

\begin{definition}[Regular and Critical Values of $t(\mathcal{F}, \cdot, h)$]
\label{def:regularvalue}
 Let $\cristats$ be the set of critical values of the map  $t(\mathcal{F},\cdot,h): \realpodalfunctionspacel{m}\to \mathbb{R}^{|\mathcal{F}|}$. Then, the set of regular values is defined by
\begin{equation}
\label{eq:regularvalues}
\regstats = \cup_{m >0} \marginalpolytopemh{m} \setminus \cup_{m >0}  \cristats
\end{equation}
and
\begin{equation*}
 \regstatsm{m}= \marginalpolytopemh{m} \cap \regstats.
\end{equation*}
It is also convenient to single out the \emph{non-extremal} values of $\mathcal{F}$,
\begin{equation*}
 \goodvaluesm{m}= \regstatsm{m} \setminus\cup_{m'>0}^m \partial T^{(m',r)}(\mathcal{F},h)
\end{equation*}
and $\goodvalues = \cup_{m>0} \goodvaluesm{m}$.

\end{definition}

\begin{remark}
From Theorem \ref{thm:denseinterior}, if $\{t(\mathcal{F}_i, \cdot, h)\}$ is a set of independent analytic functions and $r\frac{m(m+1)}{2}  > |\mathcal{F}| $, then $\goodvaluesm{m}^\circ$ is dense in $\marginalpolytopemh{m}$, hence $\goodvalues^\circ$ is dense in $\marginalpolytopeh$. In other words, a randomly chosen $u \in \goodvaluesm{m}$ yields a $\realfeasibleregionmlh{m}{h}$ with only regular points, with probability $1$.
\end{remark}


\begin{definition}
\label{def:m0}
Let $\mathcal{F}$ be an ordered set of quantum graphs and let $h$ be a $|\mathcal{F}| \times r$ matrix of analytic functions. Let $u \in \marginalpolytopeh$. The initial size of step functions we will need is
\begin{equation*}
    m_0(\mathcal{F}, u, h ) = \min \{ m :  r\frac{m(m+1)}{2}  > |\mathcal{F}|  \mbox{ and } u \in \marginalpolytopemh{m}^\circ \},
\end{equation*}
and $m_0(\mathcal{F}) = \min \{ m :  r\frac{m(m+1)}{2}  > |\mathcal{F}| \}$. It is also convenient to define
\begin{equation*}
    m_1(\mathcal{F}, u, h ) = \min \{ m :  r\frac{m(m+1)}{2}  > |\mathcal{F}| + 1 \mbox{ and } u \in \marginalpolytopemh{m}^\circ \}.
\end{equation*}
It is clear that $m_1(\mathcal{F}, u, h ) \geq m_0(\mathcal{F}, u, h ) \geq m_0(\mathcal{F}) $.
\end{definition}

\noindent The next lemma confirms that, at the smallest admissible size $m_0(\mathcal{F})$, step functions are already parametrized without repeated rows or columns.
\begin{restatable}{lemma}{lemmzero}
\label{lem:m0}
\quad
\begin{itemize}
    \item If $m_0=m_0(\mathcal{F})$ and $u \in \regstats$ is a regular value of $t(\mathcal{F}, \cdot,h)$ then $\realfeasibleregionmlh{m_0}{h}$ has only step functions with non repetitive row/columns everywhere.
    \item If $m_0=m_0(\mathcal{F}, u,h )$ and $u \in \regstats$ is a regular value of $t(\mathcal{F}, \cdot,h)$ then $\realfeasibleregionmlh{m_0}{h}$ has only step functions with non repetitive row/columns everywhere.
\end{itemize}
\end{restatable}

\begin{proof}
By contradiction. Let $(\sfmat, \pi) \in  \realfeasibleregionmlh{m}{h}$ be a step function with repeated rows/columns in positions $m-1$ and $m$. Coarsening it gives an $(m-1)$-step function $(\sfmat',\pi')$, where $\sfmat'$ is obtained by deleting the last row/column of $\sfmat$ and $\pi'$ by
\begin{equation*}
\pi'_i = \begin{cases} \pi_i & i < m-1 \\
\pi_{m-1}+\pi_m & i = m-1
\end{cases}
\end{equation*}
For the first case: by construction $(\sfmat',\pi') \in \realfeasibleregionmlh{m-1}{h}$, and since $u \in \regstats$, the Jacobian $J_{(\sfmat',\pi')}(\mathcal{F},h)$ is full rank. Hence $m_0=m_0(\mathcal{F})$ was not actually the minimal number of steps -- a contradiction.

The second case is identical.
\end{proof}

Once $m \geq m_0(\mathcal{F},u,h)$ and $\{t(\mathcal{F}, \cdot, h)\}$ is a set of independent analytic functions, $\marginalpolytopemh{m}$ has non-empty interior; analytic continuation together with Sard's theorem then makes $u \in \marginalpolytopemh{m}$ a regular value for almost every $u$. Picking $u \in \marginalpolytopemh{m}$ at random, in other words, gives a smooth manifold $\realfeasibleregionmlh{m}{h}$ with probability $1$. The next theorem upgrades this from the full marginal polytope down to each fixed-$\pi$ slice $\realfeasibleregionmlpi{m}{\pi,h}$.

\begin{restatable}{theorem}{thmpismoothness}
\label{thm:pi_smoothness}
If $u \in \regstats$ and $m \geq m_0(\mathcal{F},u,h)$ then all points of $ \realfeasibleregionmlpi{m}{\pi,h}$ for any $\pi \in (\convexcombm{m})^\circ$ are regular points of $t(\mathcal{F}, \cdot,h)$ hence $ \realfeasibleregionmlpi{m}{\pi,h}$ is an analytic manifold.
\end{restatable}

\begin{proof}
Let $\lambda^{(n)}$ be an ordered set of $n$ values in $(0,1)$, with $\lambda^{(0)}=\emptyset$. Let $SP^{(m,n,r)}(\mathcal{F}, u, h, \lambda^{(n)})$ be the region obtained from $\realfeasibleregionmlh{m+n}{h}$ by a partial refinement that fixes the first $n$ entries of $\pi$, i.e.
\begin{eqnarray*}
    SP^{(m,n,r)}(\mathcal{F}, u, h, \lambda^{(n)}) = \{ (\sfmat,\pi) \in \realfeasibleregionmlh{m+n}{h} : t(\mathcal{F},(\sfmat,\pi))=u \mbox{ and } \\ \{\pi_i=\lambda^{(n)}_i \pi_i \mbox{ and } \pi_{i+1}=(1-\lambda^{(n)})\pi\}_{i=1}^n\}.
\end{eqnarray*}
Define the map
\[
\phi_i:SP^{(m,i-1,r)}(\mathcal{F}, u, h, \lambda^{(i-1)}) \to (0,1) \quad \mbox{ by } \quad \phi_i((\sfmat,\pi)) = \pi_i
\]
with image $S_i^\pi=\phi(SP^{(m,i-1,r)}(\mathcal{F}, u, h, \lambda^{(i-1)}))$.

Pick any $(\sfmat,\pi) \in \realfeasibleregionmlh{m}{h}$; we construct $SP^{(m,m,r)}(\mathcal{F}, u, h, \lambda^{(m)})$ so that a refinement of $(\sfmat,\pi)$ lands in it, via the following inductive construction.
\begin{enumerate}
    \item $SP^{(m,0,r)}(\mathcal{F}, u, h, \emptyset) = \realfeasibleregionmlh{m}{h}$ is a smooth manifold since $u$ is a regular value.
    \item For each $i$ in $1,\cdots,m$
    \begin{enumerate}
        \item Pick a regular value $\pi'_i \in S_i^\pi \subseteq (\sum_{j=1}^{i-1} \pi_j, \sum_{j=1}^{i} \pi_j)$. Such a value exists because $SP^{(m,i-1,r)}(\mathcal{F}, u, h, \lambda^{(i-1)})$ is a smooth manifold of regular points, so $\phi_i$ is a smooth map, $S_i^\pi$ has non-empty interior, and by Sard's theorem the critical values of $\phi_i$ have zero measure in $S_i^\pi$.
        \item Let $\lambda_i  =\frac{\pi'_i}{\pi_i}$ and $\lambda^{(i)}=\lambda^{(i-1)} \cup \{\lambda_i\}$.
         \item By construction $SP^{(m,i,r)}(\mathcal{F}, u, h, \lambda^{(i)})$ is a smooth manifold of regular points satisfying $t(\mathcal{F},u,h)=u$.
    \end{enumerate}
\end{enumerate}
The resulting $SP^{(m,m,r)}(\mathcal{F}, u, h, \lambda^{(m)})$ is a smooth manifold of regular points in which the partition vector of every step function is fully determined, and $SP^{(m,m,r)}(\mathcal{F}, u, h, \lambda^{(m)})=\realfeasibleregionmlpi{2m}{\pi', h}$ where
\[
\pi' = \theta(\cdots \theta(\pi,\lambda^{(m)}_1,1),\lambda^{(m)}_2,3) \cdots, \lambda^{(m)}_m, 2m-1)
\]
and, by construction,
\[
(\sfmat',\pi') = \theta(\cdots \theta((\sfmat,\pi),\lambda^{(m)}_1,1),\lambda^{(m)}_2,3) \cdots, \lambda^{(m)}_m, 2m-1) \in SP^{(m,m,r)}(\mathcal{F}, u, h, \lambda^{(n)}).
\]
By Lemma \ref{lem:jacobianrefinement}, $J_{(\sfmat',\pi')}(\mathcal{F},h)$ is full rank iff $J_{(\sfmat,\pi)}(\mathcal{F},h)$ is, so $(\sfmat,\pi)$ is a regular point of $t(\mathcal{F}, (\cdot,\pi), h)$. Thus $ \realfeasibleregionmlpi{m}{\pi,h}$ is an analytic manifold.
\end{proof}

\subsection{The density function}

The entropy rate function we ultimately care about is, it turns out, just one instance of a more general object: a \emph{density function}, built from an arbitrary analytic function in exactly the way $I$ is built from $I_0$. Proving our results at this level of generality costs nothing extra and saves us from treating the entropy case separately later.

\begin{definition}[The density function]
\label{def;densitiy}
Let $f_0:\mathbb{R} \to \mathbb{R}$ be an analytic function. Then $f_d:\realgraphonspacel \to \mathbb{R}$  is defined by
\begin{equation*}
f_d(W)= \sum_{k=1}^r \int_{[0,1]^2} f_0(W_k(x_1, x_2)) dx_1 dx_2
\end{equation*}
\end{definition}
\noindent On $\realpodalfunctionspacel{m}$ the density function $f_d$ takes the form
\begin{equation}
\label{eq:generalratefunctionstep}
f_d((\sfmat,\pi))= \sum_{k=1}^r \sum_{x_1, x_2=1}^m f_0(\sfmat_{x_1,x_2;k}) \pi_{x_1}\pi_{x_2}
\end{equation}
and, true to the promise above, $f_d(W)$ is itself an $h$-subgraph density,
\begin{equation*}
    f_d(W) = t(E_s,W,[f_0])
\end{equation*}
where $E_s=\sum_{i=1}^r E_i$ is the linear combination of the single-edge graphs $E_i$, one per relation, and $[f_0]$ is $r$ copies of $f_0$ stacked into a vector.

We will work with two choices of $f_0$ throughout: an arbitrary analytic function on $\mathbb{R}$, and $I_0:[0,1] \to \mathbb{R}$ given by $I_0(u)=u\log_2(u)+(1-u)\log_2(1-u)$, which is analytic only on $(0,1)$. Both choices give density functions that are continuous in the $L^{1}(\prod_{k=1}^r [0,1]^2)$ topology -- a fact we record now because the global analysis of Section~\ref{sec:globalanalysis} leans on it directly.

\begin{restatable}{lemma}{lemcontinuity}
\label{lem:continuity}
\quad
\begin{enumerate}
    \item Let $f_d:\realgraphonspacel \to \mathbb{R}$ be the function defined by
    \begin{equation*}
        f_d(W) = \sum_{k=1}^r \int_{[0,1]^2} f_0(W_k(x_1, x_2)) dx_1 dx_2
    \end{equation*}
    where $f_0:\mathbb{R} \to \mathbb{R}$ is smooth. Then $f_d$ is continuous on $\realgraphonspacel$ in the $L^{1}( \prod_{k=1}^r [0,1]^2)$ topology.
    \item Let $I:\graphonspacel \to \mathbb{R}$ be the function defined by
    \begin{equation*}
        I(W) = \sum_{k=1}^r \int_{[0,1]^2} I_0(W_k(x_1, x_2)) dx_1 dx_2
    \end{equation*}
    where $I_0(x)= x\log(x) + (1-x) \log(1-x)$. Then $I$ is continuous in the $L^{1}( \prod_{k=1}^r [0,1]^2)$ topology when $\|W\|_\infty < 1$.
\end{enumerate}

\end{restatable}

\begin{proof}

Let $W \in  \realgraphonspacel$ and $V \in U$, where $U$ is a sufficiently small open neighborhood of $W$. Both cases reduce to exhibiting a constant $C >0$ such that
\begin{equation*}
|f_d(W) - f_d(V)| \leq  C \|W-V \|_1 \mbox{ and }  |I(W) - I(V)| \leq  C \|W-V \|_1
\end{equation*}
For the first case,
\begin{eqnarray*}
&& |f_d(W) - f_d(V)|= \left| \sum_{k=1}^r \int_{[0,1]^2} \left(f_0(W_k(x_1,x_2 ))-f_0(V_k(x_1,x_2) \right) dx_1dx_2\right| \\
&&\leq \sum_{k=1}^r \int_{[0,1]^2} |f_0(W_k(x_1,x_2))-f_0(V_k( x_1,x_2))| dx_1 dx_2  \\
&&\leq  C \sum_{k=1}^r \int_{[0,1]^2} |W_k(x_1,x_2)-V_k(x_1,x_2)| dx_1 dx_2  =   C \|W-V \|_1,
\end{eqnarray*}
by the Lipschitz condition on $f_0$.

\noindent For the second case, note that $|I'_0(W)|$ and $|I'_0(V)|$ are bounded by $|I'_0(\Delta(W,V))|$, where
\begin{eqnarray*}
\Delta(W,V) = \min_{k \in [r]} \inf_{(x_1,x_2) \in {[0,1]}^2 }  \{|W_k(x_1, x_2)|, \\|1-W_k(x_1, x_2)|,|V_k(x_1, x_2)|,|1-V_k(x_1, x_2)|\}
\end{eqnarray*}
and $I'_0(x) = \log\left( \frac{x}{1-x} \right) $. Hence
\begin{eqnarray*}
&&|I(W) - I(V)|= \left| \sum_{k=1}^r \int_{[0,1]^2} \left(I_0(W_k (x_1,x_2))-I_0(V_k(x_1,x_2)) \right) dx_1dx_2 \right| \\
&&\leq \sum_{k=1}^r\int_{[0,1]^2} |I_0(W_k(x_1,x_2))-I_0(V_k(x_1,x_2))| dx_1 dx_2   \\
&&\leq \sum_{k=1}^r |I'_0(\Delta(W,V))| \int_{[0,1]^2} |W_k(x_1,x_2)-V_k(x_1,x_2)| dx_1dx_2 \\
&=&   |I'_0(\Delta(W,V))| \cdot \|W-V \|_1.
\end{eqnarray*}
\end{proof}

\subsection{The principle of optimization of density functions}
\label{sec:pod}
Proving that the solutions of $\optimalsolution = \argmin_{W \in \feasibleregion} I(W)$ are step functions reduces to proving the same statement one level up in generality: that the solutions
\begin{equation*}
     \realoptimalsolutionf{f_d, h} = \argmin_{W \in \realfeasibleregionlh{h} } f_d(W)
\end{equation*}
are step functions for \emph{every} density function $f_d$ and every ordered set of analytic functions $h=\{h_1, \cdots, h_n\}$. We assume throughout that $f_0$ is chosen so that $\realoptimalsolutionf{f_d, h} \neq \emptyset$ -- strict convexity of $f_0$, for instance, already guarantees a minimum of $f_d$ on $\realfeasibleregionlh{h}$.

We also take the solutions of $\realoptimalsolutionf{f_d, h}$ to live in $L^1(\prod_{i=1}^r[0,1]^2)$ rather than in the unlabeled graphon space $\realgraphonspaceu$; this loses nothing, since the subgraph density constraints and the density functions are both invariant under the measure-preserving relabeling map, and the $L^1(\prod_{i=1}^r[0,1]^2)$ topology is strictly stronger than $(\realgraphonspaceu, \delta_\Box)$, so an $L^1$ solution is automatically a valid solution in the unlabeled space. We call the statement that $\realoptimalsolution{f_d, h}$ are step functions the \emph{Principle of Optimization of Density functions} (POD); the RRS conjecture on $\optimalsolution$ is exactly the special case $f_d=I$. Formally:
 \begin{restatable}{theorem}{thmPOD}
 \label{thm:POD}
 Let $\mathcal{F}=\{\mathcal{F}_1, \cdots, \mathcal{F}_n\}$ be an ordered set of quantum graphs and let $h$ be a $|\mathcal{F}| \times r$ matrix of analytic functions such that $\{t(\mathcal{F} \cup \{E_s\}, \cdot, [h | [f_0]] )\}$ is a set of independent analytic functions, where $[h | [f_0]]$ is the matrix obtained from $h$ by adjoining the row vector $[f_0]$ at the bottom. Let $u \in \goodvalues$ and
 \begin{equation*}
    \realoptimalsolutionm{m}{f_d,h}  = \argmin_{W \in \realfeasibleregionmlh{m}{h} } f_d(W).
\end{equation*}
Let $m_1 = m_1(\mathcal{F},u,h)$ and let
$(\sfmat,\pi) \in \realoptimalsolutionm{m_1}{f_d,h}$. Then
\[
\realoptimalsolutionm{m_1}{f_d,h} \subseteq \realoptimalsolutionf{f_d, h}.
\]
Moreover if all $(\sfmat,\pi) \in \realoptimalsolutionm{m_1}{f_d,h}$ are isolated local minima of $f_d$ in $\realfeasibleregionlh{h}$ then  $\realoptimalsolutionm{m_1}{f_d,h}=\realoptimalsolutionf{f_d, h}$.

\end{restatable}

\noindent The proof strategy is to represent the global solutions of $\realoptimalsolutionm{m}{f_d,h}$ as finite-dimensional vectors, then show there is a maximum size $m_1 = m_1(\mathcal{F},u,h)$ beyond which $\realoptimalsolution{f_d, h}$ never gains a new global minimum inside $\realoptimalsolutionm{m}{f_d,h}$. Throughout, we must keep two topologies on step functions apart: the Euclidean topology and the function topology $L^1(\prod_{i=1}^r[0,1]^2)$; by default we mean the Euclidean one unless we say otherwise.

\begin{figure}
    \centering
\tikzstyle{block} = [draw, rectangle,
    minimum height=3em, minimum width=6em]
\tikzstyle{sum} = [draw, circle, node distance=1cm]
\tikzstyle{input} = [coordinate]
\tikzstyle{output} = [coordinate]
\tikzstyle{pinstyle} = [pin edge={to-,thin,black}]

\begin{tikzpicture}[auto, node distance=2cm]


\node [block] (connected) at ( 0,5) {$\begin{array}{c}
     \mbox{No new connected component}   \\
     \mbox{Theorem } \ref{thm:numbercomponents}
\end{array}$};

\node [block] (noessential) at ( 6,5) {$\begin{array}{c}
     \mbox{No new global minimum}   \\
     \mbox{Corollary } \ref{cor:nonessentiallocal}
\end{array}$};

\node [block] (pod) at ( 6, 3.5 ) {$\begin{array}{c}
     \mbox{POD}  \\
     \mbox{Theorem } \ref{thm:POD}
\end{array}$};

\node [block] (nonextremal) at ( 0,2) {$\begin{array}{c}
     \mbox{Non Extremal Statistics  }   \\
     \mbox{ Lemma }\ref{lem:randomsolutions}
\end{array}$};

\node [block] (radin) at ( 6,2) {$\begin{array}{c}
     \mbox{$\optimalsolution$ are step functions}   \\
     \mbox{Theorem } \ref{thm:radin}
\end{array}$};

\draw [->] (connected) --  (noessential);
\draw [->] (noessential) -- (pod);
\draw [->] (pod) --  (radin);
\draw [->] (nonextremal) --  (radin);

\end{tikzpicture}
    \caption{Dependencies of the main results to prove $\optimalsolution$ are step functions}
    \label{fig:theoremdependencies}
\end{figure}

Once POD is established, Lemma~\ref{lem:randomsolutions} shows that the constraints $0 \leq W \leq 1$ built into $\optimalsolution$ are inactive whenever $u$ is a non-extremal statistic -- and that is exactly what closes the proof of the RRS conjecture for $\optimalsolution$.

Figure~\ref{fig:theoremdependencies} lays out how the main results feed into that proof. The cornerstone is the topological stability, for $m>m_0(\mathcal{F},u,h)$, of the number of connected components of $\realfeasibleregionmlh{m}{h}$ (Theorem~\ref{thm:numbercomponents}); generalizing subgraph density to $h$-subgraph density lets us reapply that same theorem to conclude that $f_d$ gains no new global minimum in $\realfeasibleregionmlh{m}{h}$ once $m > m_1(\mathcal{F},u,h)$, provided every global minimum in $\realfeasibleregionmlh{m_1}{h}$ is isolated.

These two results together prove Theorem~\ref{thm:POD}, by ruling out any global minimum in
\[
\realoptimalsolution{f_d, h} \setminus \cup_{m>m_1} \realoptimalsolutionm{m}{f_d, h}.
\]
Lemma~\ref{lem:randomsolutions} then shows that when $u$ is not an extremal statistic, $\optimalsolutionmf{m}{I}$ takes values only in $(0,1)$, so the constraint $0 \leq W \leq 1$ never binds in the problem $\optimalsolution$; and Theorem~\ref{thm:radin} closes the loop by proving that $\optimalsolution$ are indeed step functions.

\subsection{Local analysis in  $\realgraphonspacel$}
\label{sec:localanalysis}

This section supplies the two checkable sufficient conditions for isolation of local minima that are used elsewhere in the paper: a condition for isolation of $(\sfmat^*,\pi^*)$ \emph{within} the finite-dimensional manifold $\realfeasibleregionmlh{m_1}{h}$ (Lemma \ref{lem:isolatedminimum}, used in the proof of Corollary \ref{cor:nonessentiallocal}), and a condition for isolation of $(\sfmat^*,\pi^*)$ in the \emph{full} space $\realgraphonspacel$ (Theorem \ref{thm:isolatedgraphon}, used in the hypothesis of Theorem \ref{thm:POD}). Throughout, $(\sfmat^*,\pi^*) \in \realfeasibleregionmlh{m_1}{h}$ denotes a critical point of $f_d$ on $\realfeasibleregionmlh{m_1}{h}$ with Lagrange multipliers $\mu^* \in \mathbb{R}^{|\mathcal{F}|}$ (Theorem \ref{thm:first_order_KKT}), and $V^*$ denotes the corresponding step function in $\realgraphonspacel$.

\begin{definition}[Isolated local minimum]
\label{def:isolatedlocalmin}
Let $X \in \{\realfeasibleregionmlh{m_1}{h}, \realfeasibleregionlh{h}\}$, equipped respectively with the Euclidean topology and the $L^1(\prod_{k=1}^r [0,1]^2)$ topology. A point $x^* \in X$ is an \emph{isolated local minimum} of $f_d$ in $X$ if there is a neighborhood $U$ of $x^*$ such that $f_d(x^*) < f_d(x)$ for all $x \in U \setminus \{x^*\}$.
\end{definition}

\noindent By this definition, ``isolated local minimum'' and ``strict local minimum'' coincide. This is the notion actually used in the proof of Theorem \ref{thm:nonessentiallocal}: it rules out a connected, equal-value continuum of competing minimizers emanating from $x^*$, which is precisely what is needed to derive a contradiction with the second case of that theorem.

\subsubsection{Isolation within the step-function manifold}

\begin{lemma}
\label{lem:isolatedminimum}
Let $u \in \regstats$ and $m_1 = m_1(\mathcal{F},u,h)$. Let $(\sfmat^*,\pi^*) \in \realfeasibleregionmlh{m_1}{h}$ be a critical point of $f_d$ on $\realfeasibleregionmlh{m_1}{h}$ with multipliers $\mu^*$. If the Lagrangian Hessian $H^g_{(\sfmat^*,\pi^*)} f_d$ 
(Definition \ref{def:lagrangian_hessian}) is positive definite on $T_{(\sfmat^*,\pi^*)} \realfeasibleregionmlh{m_1}{h}$, then $(\sfmat^*,\pi^*)$ is an isolated local minimum of $f_d$ in $\realfeasibleregionmlh{m_1}{h}$.
\end{lemma}

\begin{proof}
Since $u$ is a regular value, $\realfeasibleregionmlh{m_1}{h}$ is a smooth (indeed real-analytic, since $t(\mathcal{F},\cdot,h)$ is a finite sum of products of the analytic functions $\{h_k\}$) manifold near $(\sfmat^*,\pi^*)$ by Theorem \ref{cor:regular_value_manifold}. Theorem \ref{thm:second_order_sufficient} applies directly with $f=f_d$, $h = t(\mathcal{F},\cdot,h)-u$, $x^*=(\sfmat^*,\pi^*)$, and gives that $(\sfmat^*,\pi^*)$ is a strict local minimum of $f_d$ on $\realfeasibleregionmlh{m_1}{h}$, which by Definition \ref{def:isolatedlocalmin} is the claim.
\end{proof}

\subsubsection{The tangent decomposition relative to a step function}

Fix $(\sfmat^*,\pi^*) \in \realfeasibleregionmlh{m_1}{h}$ and let $S = \{S_1, \cdots, S_{m_1}\}$ be the partition of $[0,1]$ with $\nu(S_i)=\pi^*_i$ that represents $V^*=(\sfmat^*,\pi^*)$. For $\eta = (\eta_1, \cdots, \eta_r)$ with each $\eta_k \in L^2([0,1]^2)$ symmetric, define the conditional expectation of $\eta$ onto $S$,
\begin{equation*}
\bar\eta_k(x,y) = \sum_{i,j=1}^{m_1} \left( \frac{1}{\pi^*_i \pi^*_j} \int_{S_i \times S_j} \eta_k \right) \mathbbm{1}_{S_i \times S_j}(x,y), \qquad \eta_k^\perp = \eta_k - \bar\eta_k.
\end{equation*}
By construction $\eta = \bar\eta + \eta^\perp$ is the $L^2$-orthogonal decomposition of $\eta$ relative to $S$: $\bar\eta$ is a step function on $S$ (i.e. $\bar\eta$ is tangent to $\realpodalfunctionspacelpi{m_1}{\pi^*}$ at $\sfmat^*$) and $\eta^\perp_k$ has zero mean on every block, $\int_{S_i \times S_j} \eta_k^\perp = 0$ for all $i,j \in [m_1]$, $k \in [r]$. We call $\eta^\perp$ an \emph{off-step perturbation} of $V^*$.

\begin{lemma}[Continuum first variation of $t(F,\cdot,h)$ at a step function]
\label{lem:continuumfirstderivative}
Let $F$ be a multi-relational quantum graph and let $\eta = (\eta_1,\cdots,\eta_r)$, $\eta_k \in L^2([0,1]^2)$ symmetric. Then $t(F,\cdot,h)$ is Gateaux differentiable at $V^*$ and
\begin{equation}
\label{eq:continuumfirstderivative}
Dt(F,V^*,h)[\eta] = \sum_{k=1}^r \sum_{(a,b)\in E_k(F)} \sum_{i,j=1}^{m_1} h_k'(\sfmat^*_{ij;k}) \, t_{ij;k}(F^{\bullet(ab),k},(\sfmat^*,\pi^*),h) \int_{S_i \times S_j} \eta_k(x,y) \, dx\,dy.
\end{equation}
In particular, $Dt(F,V^*,h)[\eta]$ depends on $\eta$ only through its block totals $\left\{\int_{S_i\times S_j}\eta_k\right\}_{i,j,k}$, and $Dt(F,V^*,h)[\bar\eta] = \sum_{i\le j,k} \frac{\partial t(F,\cdot,h)}{\partial \sfmat_{ij;k}}\Big|_{(\sfmat^*,\pi^*)} \bar\eta_{ij;k}$ as in Theorem \ref{thm:partialderivatives}.
\end{lemma}

\begin{proof}
Differentiating $t(F,V^*+\epsilon\eta,h) = \int_{[0,1]^{|V(F)|}} \prod_{k} \prod_{(s,t)\in E_k(F)} h_k\bigl(V_k^*(x_s,x_t) + \epsilon \eta_k(x_s,x_t)\bigr) \prod_{i} dx_i$ in $\epsilon$ at $\epsilon=0$, the product rule selects exactly one edge $(a,b)\in E_k(F)$ to differentiate, exactly as in the proof of Theorem \ref{thm:partialderivatives}:
\begin{equation*}
Dt(F,V^*,h)[\eta] = \sum_k \sum_{(a,b)\in E_k(F)} \int_{[0,1]^{|V(F)|}} h_k'(V_k^*(x_a,x_b)) \, \eta_k(x_a,x_b) \prod_{(s,t;l)\ne (a,b;k)} h_l(V_l^*(x_s,x_t)) \prod_i dx_i.
\end{equation*}
Since $V^*$ is constant on blocks of $S$, integrating out the variables $\{x_s\}_{s \ne a,b}$ for fixed $x_a \in S_i, x_b \in S_j$ produces exactly $t_{ij;k}(F^{\bullet(ab),k},(\sfmat^*,\pi^*),h)$ by Definition \ref{def:labeledgraphs}, and $h_k'(V_k^*(x_a,x_b)) = h_k'(\sfmat^*_{ij;k})$ is constant on $S_i \times S_j$. Splitting the integral over $x_a,x_b$ into the blocks $S_i \times S_j$ gives \eqref{eq:continuumfirstderivative}. The last statement follows because $\bar\eta_k$ is constant equal to $\bar\eta_{ij;k}$ on $S_i\times S_j$, so $\int_{S_i\times S_j}\bar\eta_k = \pi^*_i\pi^*_j \bar\eta_{ij;k}$, and comparing with Theorem \ref{thm:partialderivatives} term by term.
\end{proof}

\begin{corollary}[Off-step perturbations are unconstrained at first order]
\label{cor:offstepfeasible}
For every multi-relational quantum graph $F$ and every off-step perturbation $\eta^\perp$, $Dt(F,V^*,h)[\eta^\perp] = 0$. Consequently, writing $K(V^*) = \{\eta \in L^2(\prod_k [0,1]^2)_{\mathrm{sym}} : Dt(\mathcal{F}_i,V^*,h)[\eta]=0,\ i=1,\cdots,|\mathcal{F}|\}$ for the kernel of the linearized constraints at $V^*$, we have $\eta = \bar\eta + \eta^\perp \in K(V^*)$ if and only if $\bar\eta \in K(V^*) \cap T_{\sfmat^*}\realpodalfunctionspacelpi{m_1}{\pi^*}$, for an \emph{arbitrary} off-step $\eta^\perp$.
\end{corollary}

\begin{proof}
Immediate from \eqref{eq:continuumfirstderivative}: every block total of $\eta^\perp$ vanishes by construction, and the formula depends on $\eta$ only through these totals. The second claim follows from linearity of $Dt(F,V^*,h)[\cdot]$ and $Dt(F,V^*,h)[\eta]=Dt(F,V^*,h)[\bar\eta]+Dt(F,V^*,h)[\eta^\perp]=Dt(F,V^*,h)[\bar\eta]$.
\end{proof}

\noindent Corollary \ref{cor:offstepfeasible} shows us off-step directions are always ``free'' to first order. The same computation applied to $f_d=t(E_s,\cdot,[f_0])$ (Definition \ref{def;densitiy}) gives $Df_d(V^*)[\eta^\perp]=0$ as well, and consequently:

\begin{proposition}[Lagrangian criticality lifts to $\realgraphonspacel$]
\label{prop:lagrangianlifts}
$V^*$ is a critical point, on the whole space, of the Lagrangian $\mathcal{L}(W) = f_d(W) - \sum_{i=1}^{|\mathcal{F}|} \mu_i^* \bigl(t(\mathcal{F}_i,W,h)-u_i\bigr)$, i.e. $D\mathcal{L}(V^*)[\eta]=0$ for every $\eta \in L^2(\prod_k[0,1]^2)_{\mathrm{sym}}$, not merely for $\eta$ tangent to a step-function space.
\end{proposition}

\begin{proof}
Write $\eta=\bar\eta+\eta^\perp$. On off-step directions, $D\mathcal{L}(V^*)[\eta^\perp] = Df_d(V^*)[\eta^\perp] - \sum_i \mu_i^* Dt(\mathcal{F}_i,V^*,h)[\eta^\perp] = 0 - 0 = 0$ by the discussion above. On step directions, $D\mathcal{L}(V^*)[\bar\eta]=0$ for all $\bar\eta$ because $(\sfmat^*,\pi^*)$ is a critical point of $f_d$ on $\realfeasibleregionmlh{m_1}{h}$ with multipliers $\mu^*$, which is exactly the first-order condition 
\[
\nabla f_d(\sfmat^*,\pi^*) = \sum_i \mu_i^* \nabla t(\mathcal{F}_i,(\sfmat^*,\pi^*),h)    
\]
of Theorem \ref{thm:first_order_KKT} applied on $\realpodalfunctionspacel{m_1}$. Summing the two cases gives the claim for general $\eta$.
\end{proof}

\subsubsection{The Lagrangian Hessian decouples on off-step perturbations}

We now compute the second variation $D^2\mathcal{L}(V^*)[\eta,\eta]$. Unlike $f_d$, the constraint functionals $t(\mathcal{F}_i,\cdot,h)$ are not pointwise (diagonal) functionals of $W$ whenever $\mathcal{F}_i$ has two edges sharing a vertex (e.g.\ a triangle), so $D^2t(\mathcal{F}_i,V^*,h)[\cdot,\cdot]$ is in general \emph{not} a diagonal operator; such cross-point couplings are the same phenomenon underlying 
the Hessian computations for subgraph-count large deviations on graphons, see e.g.\ \cite{lubetzky2015replica}. We isolate this exactly.

\begin{definition}[Second-order labeled graphs]
\label{def:secondorderstubs}
Let $e_1=(a,b)\in E_k(F)$, $e_2=(c,d)\in E_{k'}(F)$ be distinct edges of $F$.
\begin{itemize}
\item If $\{a,b\}\cap\{c,d\}=\emptyset$ (vertex-disjoint), let $F^{\bullet(ab),k;(cd),k'}$ be the graph obtained from $F$ by deleting $e_1,e_2$ and labeling $a,b,c,d$, and define
\begin{equation*}
t_{ij;k,\,pq;k'}\bigl(F^{\bullet(ab),k;(cd),k'},(\sfmat,\pi),h\bigr) = \!\!\sum_{\{x_s\}_{s\notin\{a,b,c,d\}}}\!\! \prod_{(s,t;u)\in E(F)\setminus\{e_1,e_2\}}\!\! h_u(\sfmat_{x_s,x_t;u}) \prod_{s \notin \{a,b,c,d\}} \pi_{x_s}
\end{equation*}
with $x_a=i,x_b=j,x_c=p,x_d=q$ fixed.
\item If $\{a,b\}\cap\{c,d\}=\{b\}$ (sharing exactly one vertex, $b=c$), let $F^{\bullet(ab),k;(bd),k'}$ be the graph obtained from $F$ by deleting $e_1,e_2$ and labeling $a,b,d$, and define
\begin{equation*}
t_{i,j,l;k,k'}\bigl(F^{\bullet(ab),k;(bd),k'},(\sfmat,\pi),h\bigr) = \!\!\sum_{\{x_s\}_{s\notin\{a,b,d\}}}\!\! \prod_{(s,t;u)\in E(F)\setminus\{e_1,e_2\}}\!\! h_u(\sfmat_{x_s,x_t;u}) \prod_{s\notin\{a,b,d\}} \pi_{x_s}
\end{equation*}
with $x_a=i,x_b=j,x_d=l$ fixed.
\item If $\{a,b\}=\{c,d\}$ and $k\ne k'$ (same vertex pair, distinct relations), let $F^{\bullet(ab);k,k'}$ be the graph obtained from $F$ by deleting both $(a,b;k)$ and $(a,b;k')$ and labeling $a,b$, and define
\begin{equation*}
t_{ij;k,k'}\bigl(F^{\bullet(ab);k,k'},(\sfmat,\pi),h\bigr) = \!\!\sum_{\{x_s\}_{s\notin\{a,b\}}}\!\! \prod_{(s,t;u)\in E(F)\setminus\{(a,b;k),(a,b;k')\}}\!\! h_u(\sfmat_{x_s,x_t;u}) \prod_{s\notin\{a,b\}} \pi_{x_s}
\end{equation*}
with $x_a=i,x_b=j$ fixed.
\end{itemize}
\end{definition}

\begin{lemma}[Continuum second variation of $t(F,\cdot,h)$ at a step function]
\label{lem:continuumsecondderivative}
With notation as above,
\begin{eqnarray}
\label{eq:continuumsecondderivative}
D^2t(F,V^*,h)[\eta,\eta] = \sum_k \sum_{(a,b)\in E_k(F)} \sum_{i,j} h_k''(\sfmat^*_{ij;k}) \, t_{ij;k}(F^{\bullet(ab),k},(\sfmat^*,\pi^*),h) \int_{S_i\times S_j} \eta_k^2 \nonumber \\
 + \; 2\!\!\sum_{\substack{\{e_1,e_2\}\subset E(F)\\ \text{vertex-disjoint}}}\!\! \sum_{i,j,p,q} h_k'(\sfmat^*_{ij;k})h_{k'}'(\sfmat^*_{pq;k'})\, t_{ij;k,\,pq;k'}(\cdot) \Bigl(\!\!\int_{S_i\times S_j}\!\!\eta_k\Bigr)\Bigl(\!\!\int_{S_p\times S_q}\!\!\eta_{k'}\Bigr) \nonumber \\
 + \; 2 \!\!\sum_{\substack{\{e_1,e_2\}\subset E(F) \\ \text{sharing a vertex}}}\!\! \sum_{i,j,l} h_k'(\sfmat^*_{ij;k}) h_{k'}'(\sfmat^*_{jl;k'}) \, t_{i,j,l;k,k'}(\cdot) \int_{S_j} g_{i;k}^\eta(y)\, g_{l;k'}^\eta(y)\, dy \nonumber \\
 + \; 2 \!\!\sum_{\substack{\{e_1,e_2\}\subset E(F) \\ e_1=(a,b;k),\,e_2=(a,b;k') \\ k\ne k'}}\!\! \sum_{i,j} h_k'(\sfmat^*_{ij;k}) h_{k'}'(\sfmat^*_{ij;k'}) \, t_{ij;k,k'}\bigl(F^{\bullet(ab);k,k'},\cdot\bigr) \int_{S_i\times S_j}\!\!\eta_k\,\eta_{k'}, \qquad
\end{eqnarray}
where $g_{i;k}^\eta(y) := \int_{S_i} \eta_k(x,y)\,dx$.
\end{lemma}

\begin{proof}
We differentiate the integral representation
\[
t(F,V^*+\epsilon\eta,h)
= \int_{[0,1]^{|V(F)|}} \prod_{k}\prod_{(s,t)\in E_k(F)} h_k\bigl(V^*_k(x_s,x_t)+\epsilon\,\eta_k(x_s,x_t)\bigr)\prod_i dx_i
\]
twice in $\epsilon$ at $\epsilon=0$. The product rule applied to the second derivative selects, from the product over all edges, an \emph{ordered pair} $(e_1,e_2)$ (possibly equal) at which to differentiate; all remaining edge factors are evaluated at $V^*$. We partition into three cases according to whether $e_1=e_2$, $e_1\ne e_2$ are vertex-disjoint, or $e_1\ne e_2$ share exactly one vertex.

\smallskip\noindent\textit{Case 1: $e_1=e_2=(a,b)\in E_k(F)$ (diagonal terms).}
When both derivatives act on the same factor, the edge $(a,b;k)$ contributes $h_k''(V^*_k(x_a,x_b))\,\eta_k(x_a,x_b)^2$ while every other edge $(s,t;l)\ne(a,b;k)$ contributes $h_l(V^*_l(x_s,x_t))$. Summing over the (ordered) choices of which edge is selected---one choice per edge---gives
\begin{equation*}
\sum_k\sum_{(a,b)\in E_k(F)}\int_{[0,1]^{|V(F)|}}
h_k''(V^*_k(x_a,x_b))\,\eta_k(x_a,x_b)^2
\prod_{(s,t;l)\ne(a,b;k)}h_l(V^*_l(x_s,x_t))\prod_i dx_i.
\end{equation*}
Fix $(x_a,x_b)\in S_i\times S_j$. Since $V^*$ is constant on blocks, $h_k''(V^*_k(x_a,x_b))=h_k''(\sfmat^*_{ij;k})$ is constant on $S_i\times S_j$. Integrating out $\{x_s\}_{s\ne a,b}$ at fixed $x_a\in S_i$, $x_b\in S_j$ yields $t_{ij;k}(F^{\bullet(ab),k},(\sfmat^*,\pi^*),h)$ by Definition~\ref{def:labeledgraphs}. Hence the Case~1 contribution equals
\begin{equation*}
\sum_k\sum_{(a,b)\in E_k(F)}\sum_{i,j}
h_k''(\sfmat^*_{ij;k})\,t_{ij;k}(F^{\bullet(ab),k},(\sfmat^*,\pi^*),h)
\int_{S_i\times S_j}\eta_k^2,
\end{equation*}
which is the first (diagonal) line of \eqref{eq:continuumsecondderivative}.

\smallskip\noindent\textit{Case 2: $e_1=(a,b;k)\ne e_2=(c,d;k')$ with $\{a,b\}\cap\{c,d\}=\emptyset$ (vertex-disjoint pairs).}
Each of the two edges contributes one first-derivative factor, while the remaining $|E(F)|-2$ edges are evaluated at $V^*$:
\begin{equation*}
h_k'(V^*_k(x_a,x_b))\,\eta_k(x_a,x_b)
\cdot h_{k'}'(V^*_{k'}(x_c,x_d))\,\eta_{k'}(x_c,x_d)
\prod_{\substack{(s,t;l)\ne(a,b;k)\\(s,t;l)\ne(c,d;k')}}h_l(V^*_l(x_s,x_t)).
\end{equation*}
Summing over \emph{ordered} pairs $(e_1,e_2)$ with $e_1\ne e_2$ introduces a factor of $2$ when we switch to a sum over \emph{unordered} pairs $\{e_1,e_2\}$. Since $a,b,c,d$ are four distinct vertices, integrating out $\{x_s\}_{s\notin\{a,b,c,d\}}$ at fixed $x_a\in S_i, x_b\in S_j, x_c\in S_p, x_d\in S_q$ yields, by Definition~\ref{def:secondorderstubs},
\[
t_{ij;k,\,pq;k'}\bigl(F^{\bullet(ab),k;(cd),k'},(\sfmat^*,\pi^*),h\bigr).
\]
The factors $h_k'(\sfmat^*_{ij;k})$ and $h_{k'}'(\sfmat^*_{pq;k'})$ are constant on their blocks, so the integral over $(x_a,x_b)\in S_i\times S_j$ and $(x_c,x_d)\in S_p\times S_q$ separates into the product of two block totals:
\[
h_k'(\sfmat^*_{ij;k})\Bigl(\int_{S_i\times S_j}\eta_k\Bigr)
\cdot h_{k'}'(\sfmat^*_{pq;k'})\Bigl(\int_{S_p\times S_q}\eta_{k'}\Bigr).
\]
Multiplying by $t_{ij;k,\,pq;k'}(\cdot)$, summing over $i,j,p,q$ and over unordered vertex-disjoint pairs, and collecting the factor $2$ gives the second line of \eqref{eq:continuumsecondderivative}.

\smallskip\noindent\textit{Case 3: $e_1=(a,b;k)\ne e_2=(b,d;k')$ sharing exactly vertex $b$, $a\ne d$ (shared-vertex pairs).}
The integrand contribution is
\begin{equation*}
h_k'(V^*_k(x_a,x_b))\,\eta_k(x_a,x_b)
\cdot h_{k'}'(V^*_{k'}(x_b,x_d))\,\eta_{k'}(x_b,x_d)
\prod_{\substack{(s,t;l)\ne(a,b;k)\\(s,t;l)\ne(b,d;k')}}h_l(V^*_l(x_s,x_t)).
\end{equation*}
Again a factor $2$ converts ordered to unordered pairs. The three distinct vertices $a,b,d$ remain free; integrating out $\{x_s\}_{s\notin\{a,b,d\}}$ at fixed $x_a\in S_i$, $x_b=y\in S_j$, $x_d\in S_l$ yields, by Definition~\ref{def:secondorderstubs},
\[
t_{i,j,l;k,k'}\bigl(F^{\bullet(ab),k;(bd),k'},(\sfmat^*,\pi^*),h\bigr),
\]
a constant on $S_i\times S_j\times S_l$. The coefficients $h_k'(V^*_k(x_a,y))=h_k'(\sfmat^*_{ij;k})$ and $h_{k'}'(V^*_{k'}(y,x_d))=h_{k'}'(\sfmat^*_{jl;k'})$ are likewise constant. Integrating $x_a\in S_i$ at fixed $y$ gives
\[
\int_{S_i}\eta_k(x_a,y)\,dx_a \;=\; g^\eta_{i;k}(y),
\]
and integrating $x_d\in S_l$ at fixed $y$ gives, using the symmetry $\eta_{k'}(y,x_d)=\eta_{k'}(x_d,y)$,
\[
\int_{S_l}\eta_{k'}(x_b,x_d)\,dx_d \;=\; \int_{S_l}\eta_{k'}(x_d,y)\,dx_d \;=\; g^\eta_{l;k'}(y).
\]
Integrating the product $g^\eta_{i;k}(y)\,g^\eta_{l;k'}(y)$ over $y\in S_j$, summing over $j,i,l$ and over unordered shared-vertex pairs, and collecting the factor $2$ gives the third line of \eqref{eq:continuumsecondderivative}.

\smallskip\noindent\textit{Case 4: $e_1=(a,b;k)\ne e_2=(a,b;k')$ sharing both vertices, $k\ne k'$ (same-pair cross-relation).}
The two edges connect the same unordered vertex pair $\{a,b\}$ but carry distinct relation labels; in a multi-relational graph $F$ this is possible whenever $(a,b)\in E_k(F)\cap E_{k'}(F)$. Each edge contributes one first-derivative factor while all remaining $|E(F)|-2$ edges are evaluated at $V^*$:
\begin{equation*}
h_k'(V^*_k(x_a,x_b))\,\eta_k(x_a,x_b)\cdot h_{k'}'(V^*_{k'}(x_a,x_b))\,\eta_{k'}(x_a,x_b)
\prod_{\substack{(s,t;l)\ne(a,b;k)\\(s,t;l)\ne(a,b;k')}}h_l(V^*_l(x_s,x_t)).
\end{equation*}
A factor of $2$ converts ordered to unordered pairs. The two labeled vertices $a,b$ are the only free variables; integrating out $\{x_s\}_{s\notin\{a,b\}}$ at fixed $x_a\in S_i$, $x_b\in S_j$ yields, by Definition~\ref{def:secondorderstubs},
\[
t_{ij;k,k'}\bigl(F^{\bullet(ab);k,k'},(\sfmat^*,\pi^*),h\bigr),
\]
which is constant on $S_i\times S_j$. The coefficients $h_k'(\sfmat^*_{ij;k})$ and $h_{k'}'(\sfmat^*_{ij;k'})$ are likewise constant on $S_i\times S_j$. Integrating $(x_a,x_b)$ over $S_i\times S_j$ gives $\int_{S_i\times S_j}\eta_k(x,y)\,\eta_{k'}(x,y)\,dx\,dy$. Summing over $i,j$, over unordered same-pair cross-relation pairs $\{(a,b;k),(a,b;k')\}$ with $k<k'$, and collecting the factor $2$ gives the fourth line of \eqref{eq:continuumsecondderivative}.

\smallskip\noindent Cases~1, 2, 3, and~4 partition all ordered pairs $(e_1,e_2)$ selected by the product rule: diagonal ($e_1=e_2$), vertex-disjoint ($e_1\ne e_2$, $\{a,b\}\cap\{c,d\}=\emptyset$), one-shared-vertex ($e_1\ne e_2$, $|\{a,b\}\cap\{c,d\}|=1$), and same-pair cross-relation ($e_1\ne e_2$, $\{a,b\}=\{c,d\}$, $k\ne k'$). Adding their contributions establishes \eqref{eq:continuumsecondderivative}.
\end{proof}

The first line of \eqref{eq:continuumsecondderivative} is a multiplication (diagonal) operator: it depends on $\eta$ only through $\int_{S_i\times S_j}\eta_k^2$. The second line depends on $\eta$ only through its block totals. The third line -- arising precisely from pairs of constraint edges that share a vertex, e.g.\ any two edges of a triangle -- is the genuinely non-diagonal part: $\int_{S_j} g^\eta_{i;k} g^\eta_{l;k'}$ need not vanish even when $\eta$ has zero block totals, since $g^\eta_{i;k}$ can be a nonzero mean-zero function on $S_j$.

\begin{lemma}[Exact decoupling of the Lagrangian Hessian]
\label{lem:exactdecoupling}
For every $\eta = \bar\eta+\eta^\perp \in L^2(\prod_k[0,1]^2)_{\mathrm{sym}}$,
\begin{equation*}
D^2\mathcal{L}(V^*)[\eta,\eta] = D^2\mathcal{L}(V^*)[\bar\eta,\bar\eta] + D^2\mathcal{L}(V^*)[\eta^\perp,\eta^\perp],
\end{equation*}
i.e.\ the cross term $D^2\mathcal{L}(V^*)[\bar\eta,\eta^\perp]$ vanishes identically.
\end{lemma}

\begin{proof}
$D^2\mathcal{L}(V^*) = D^2f_d(V^*) - \sum_i \mu_i^* D^2t(\mathcal{F}_i,V^*,h)$, so by bilinearity it suffices to check each line of \eqref{eq:continuumsecondderivative} (applied to $f_d=t(E_s,\cdot,[f_0])$ and to each $\mathcal{F}_i$) has no $(\bar\eta,\eta^\perp)$ cross term. The diagonal line: on each block, $\int_{S_i\times S_j}(\bar\eta_k+\eta_k^\perp)^2 = \int \bar\eta_k^2 + 2\bar\eta_{ij;k}\int_{S_i\times S_j}\eta_k^\perp + \int (\eta_k^\perp)^2$, and the middle term vanishes since $\bar\eta_k$ is constant on the block and $\eta^\perp_k$ has zero mean there. The vertex-disjoint line: the block totals of $\bar\eta+\eta^\perp$ equal the block totals of $\bar\eta$ alone (since $\eta^\perp$ contributes none), so this line is bilinear in the block totals of $\bar\eta$ only, with no cross term. The shared-vertex line: for $y \in S_j$, $g^{\bar\eta}_{i;k}(y) = \int_{S_i}\bar\eta_k(x,y)dx = \pi^*_i \bar\eta_{ij;k}$ is \emph{constant} on $S_j$, while $\int_{S_j} g^{\eta^\perp}_{l;k'}(y)\,dy = \int_{S_i\times S_j}\eta_{k'}^\perp=0$, i.e.\ $g^{\eta^\perp}_{l;k'}$ has zero mean on $S_j$; hence $\int_{S_j} g^{\bar\eta}_{i;k}(y)\, g^{\eta^\perp}_{l;k'}(y)\,dy = \pi^*_i\bar\eta_{ij;k}\int_{S_j} g^{\eta^\perp}_{l;k'} = 0$, and symmetrically for the other cross pairing. The same-pair cross-relation line: on each block, $\int_{S_i\times S_j}(\bar\eta_k+\eta_k^\perp)(\bar\eta_{k'}+\eta_{k'}^\perp) = \int_{S_i\times S_j}\bar\eta_k\bar\eta_{k'} + \bar\eta_{ij;k}\int_{S_i\times S_j}\eta_{k'}^\perp + \bar\eta_{ij;k'}\int_{S_i\times S_j}\eta_k^\perp + \int_{S_i\times S_j}\eta_k^\perp\eta_{k'}^\perp$; the two cross terms vanish since $\eta_k^\perp$ and $\eta_{k'}^\perp$ each have zero mean on $S_i\times S_j$. Summing over all blocks and over $\mathcal{F}_i$, all cross terms vanish.
\end{proof}

\noindent Combined with Proposition \ref{prop:lagrangianlifts} and Corollary \ref{cor:offstepfeasible}, Lemma \ref{lem:exactdecoupling} shows that, on $K(V^*)$, $\mathcal{L}$ restricted to step directions and $\mathcal{L}$ restricted to off-step directions are entirely independent quadratic forms: positivity of $D^2\mathcal{L}(V^*)[\bar\eta,\bar\eta]$ on $K(V^*)\cap T_{\sfmat^*}\realpodalfunctionspacelpi{m_1}{\pi^*}$ is exactly the hypothesis of Lemma \ref{lem:isolatedminimum} (restricted from $T_{(\sfmat^*,\pi^*)}\realfeasibleregionmlh{m_1}{h}$ to its sub\-space with $\delta\pi=0$, on which a positive-definite form remains positive-definite), and it remains to control $D^2\mathcal{L}(V^*)[\eta^\perp,\eta^\perp]$ over \emph{all} off-step $\eta^\perp$, an infinite-dimensional space.

\subsubsection{A checkable sufficient condition for isolation in $\realgraphonspacel$}

Write the diagonal weight of the Hessian on block $(i,j;k)$ as
\begin{equation}
\label{eq:diagonalweight}
c_{ij;k} = f_0''(\sfmat^*_{ij;k}) - \sum_{i'=1}^{|\mathcal{F}|} \mu^*_{i'}\, h_k''(\sfmat^*_{ij;k}) \, t_{ij;k}\bigl(\partial_k^{(\bullet\bullet)}\mathcal{F}_{i'},(\sfmat^*,\pi^*),h\bigr),
\end{equation}
using Definition \ref{def:partialbullets} to sum over all edges of relation $k$ in $\mathcal{F}_{i'}$ at once. By Lemma \ref{lem:continuumsecondderivative}, the diagonal part of $D^2\mathcal{L}(V^*)[\eta^\perp,\eta^\perp]$ equals $\sum_{i,j,k} c_{ij;k} \int_{S_i\times S_j} (\eta_k^\perp)^2$.

\begin{lemma}[Cauchy--Schwarz bound on off-diagonal cross terms]
\label{lem:cauchyschwarzbound}
For every off-step $\eta^\perp$ and every shared-vertex pair $e_1=(a,b)\in E_k(\mathcal{F}_{i'})$, $e_2=(b,d)\in E_{k'}(\mathcal{F}_{i'})$ with $a\ne d$,
\begin{equation*}
\left| \int_{S_j} g_{i;k}^{\eta^\perp}(y)\, g_{l;k'}^{\eta^\perp}(y)\, dy \right| \;\le\; \frac{\sqrt{\pi^*_i \pi^*_l}}{2} \left( \|\eta^\perp_k\|^2_{L^2(S_i\times S_j)} + \|\eta^\perp_{k'}\|^2_{L^2(S_l\times S_j)} \right),
\end{equation*}
and for every same-pair cross-relation pair $e_1=(a,b;k),\,e_2=(a,b;k')\in E(\mathcal{F}_{i'})$ with $k\ne k'$,
\begin{equation*}
\left|\int_{S_i\times S_j} \eta_k^\perp\,\eta_{k'}^\perp \right| \;\le\; \frac{1}{2}\left(\|\eta_k^\perp\|^2_{L^2(S_i\times S_j)} + \|\eta_{k'}^\perp\|^2_{L^2(S_i\times S_j)}\right).
\end{equation*}
Consequently there is an explicit constant $\Gamma = \Gamma_{\mathrm{sv}} + \Gamma_{\mathrm{sp}} \geq 0$, with
\begin{align*}
\Gamma_{\mathrm{sv}} &= \max_{j,k} \sum_{i'=1}^{|\mathcal{F}|} |\mu^*_{i'}| \sum_{\substack{(a,b)\in E_k(\mathcal{F}_{i'}),\ (b,d)\in E_{k'}(\mathcal{F}_{i'})\\ a \ne d}} \sum_{i,l} \bigl|h_k'(\sfmat^*_{ij;k}) h_{k'}'(\sfmat^*_{lj;k'})\, t_{i,j,l;k,k'}(\cdot)\bigr| \, \sqrt{\pi^*_i\pi^*_l} \\
\Gamma_{\mathrm{sp}} &= \sum_{i'=1}^{|\mathcal{F}|} |\mu^*_{i'}| \sum_{\substack{k < k' \\ (a,b)\in E_k(\mathcal{F}_{i'})\cap E_{k'}(\mathcal{F}_{i'})}} \sum_{i,j} \bigl|h_k'(\sfmat^*_{ij;k})\, h_{k'}'(\sfmat^*_{ij;k'})\, t_{ij;k,k'}\bigl(F^{\bullet(ab);k,k'},(\sfmat^*,\pi^*),h\bigr)\bigr|,
\end{align*}
such that the total contribution of shared-vertex and same-pair cross-relation terms to $-\sum_{i'}\mu^*_{i'} D^2t(\mathcal{F}_{i'},V^*,h)[\eta^\perp,\eta^\perp]$ is bounded in absolute value by $\Gamma \, \|\eta^\perp\|_2^2 := \Gamma \sum_k \int_{[0,1]^2} (\eta_k^\perp)^2$.
\end{lemma}

\begin{proof}
\textit{Shared-vertex terms (Case 3).} Recall that $g_{i;k}^{\eta^\perp}(y) = \int_{S_i} \eta_k^\perp(x,y)\,dx$. Applying Cauchy--Schwarz in $L^2(S_i)$ with functions $1$ and $\eta_k^\perp(\cdot,y)$:
\[
g_{i;k}^{\eta^\perp}(y)^2
= \left(\int_{S_i} 1\cdot \eta_k^\perp(x,y)\,dx\right)^2
\le \underbrace{\left(\int_{S_i} 1^2\,dx\right)}_{=\,\pi^*_i}
\cdot \int_{S_i} \eta_k^\perp(x,y)^2\,dx
= \pi^*_i \int_{S_i}\eta_k^\perp(x,y)^2\,dx.
\]
Integrating over $y\in S_j$:
\[
\int_{S_j} g_{i;k}^{\eta^\perp}(y)^2\,dy
\;\le\; \pi^*_i \int_{S_j}\!\int_{S_i}\eta_k^\perp(x,y)^2\,dx\,dy
= \pi^*_i\|\eta_k^\perp\|^2_{L^2(S_i\times S_j)},
\]
and similarly for $g_{l;k'}^{\eta^\perp}$. Cauchy--Schwarz on $\int_{S_j} g_{i;k}^{\eta^\perp}\,g_{l;k'}^{\eta^\perp}$ gives
\[
\left|\int_{S_j} g_{i;k}^{\eta^\perp}\,g_{l;k'}^{\eta^\perp}\right|
\;\le\;
\sqrt{\pi^*_i\|\eta_k^\perp\|^2_{L^2(S_i\times S_j)}}
\cdot
\sqrt{\pi^*_l\|\eta_{k'}^\perp\|^2_{L^2(S_l\times S_j)}},
\]
and AM--GM ($\sqrt{AB}\le(A+B)/2$) yields the first stated bound.
To see how $\Gamma_{\mathrm{sv}}$ arises, note that the total absolute contribution of all shared-vertex terms to $-\sum_{i'}\mu^*_{i'}D^2t(\mathcal{F}_{i'},V^*,h)[\eta^\perp,\eta^\perp]$ is bounded by
\begin{equation*}
\sum_{\substack{i',\,k,k' \\ (a,b;k),(b,d;k'),\,a\ne d}} \sum_{i,l,j}
|\mu^*_{i'}|\,\bigl|h_k'(\sfmat^*_{ij;k})\,h_{k'}'(\sfmat^*_{lj;k'})\,t_{i,j,l;k,k'}(\cdot)\bigr|
\frac{\sqrt{\pi^*_i\pi^*_l}}{2}
\Bigl(\|\eta_k^\perp\|^2_{L^2(S_i\times S_j)}+\|\eta_{k'}^\perp\|^2_{L^2(S_l\times S_j)}\Bigr).
\end{equation*}
For the $\|\eta_k^\perp\|^2_{L^2(S_i\times S_j)}$ terms, define
\[
a_{ij;k} \;:=\; \frac{\sqrt{\pi^*_i\pi^*_l}}{2}
\sum_{\substack{i',\,k',\,l \\ (a,b;k),(b,d;k')\in E(\mathcal{F}_{i'}) \\ a\ne d}}|\mu^*_{i'}|\,
\bigl|h_k'(\sfmat^*_{ij;k})\,h_{k'}'(\sfmat^*_{lj;k'})\,t_{i,j,l;k,k'}(\cdot)\bigr|
\]
as the coefficient of $\|\eta_k^\perp\|^2_{L^2(S_i\times S_j)}$.  Applying $\sum_m a_m b_m\le(\max_m a_m)\sum_m b_m$ with $b_{ij}=\|\eta_k^\perp\|^2_{L^2(S_i\times S_j)}$ gives
\[
\sum_{i,j} a_{ij;k}\,\|\eta_k^\perp\|^2_{L^2(S_i\times S_j)}
\;\le\;
\bigl(\max_{i,j} a_{ij;k}\bigr)\,\|\eta_k^\perp\|^2_{L^2}.
\]
Since $a_{ij;k}\le C_{j,k}$ for every $i$, where
\[
C_{j,k} \;:=\; \sum_{i'=1}^{|\mathcal{F}|}|\mu^*_{i'}|
\!\sum_{\substack{(a,b)\in E_k(\mathcal{F}_{i'}),(b,d)\in E_{k'}(\mathcal{F}_{i'})\\a\ne d}}
\sum_{i,l}\bigl|h_k'(\sfmat^*_{ij;k})\,h_{k'}'(\sfmat^*_{lj;k'})\,t_{i,j,l;k,k'}(\cdot)\bigr|\sqrt{\pi^*_i\pi^*_l},
\]
we get $\max_{i,j}a_{ij;k}\le\max_j C_{j,k}$.  The $\|\eta_{k'}^\perp\|^2_{L^2(S_l\times S_j)}$ terms contribute symmetrically (with $k$ and $k'$ interchanged).  Summing over relations~$k$ and using $\max_{j,k}C_{j,k}=\Gamma_{\mathrm{sv}}$ yields the contribution $\Gamma_{\mathrm{sv}}\|\eta^\perp\|_2^2$.

\textit{Same-pair cross-relation terms (Case 4).} AM--GM applied to $\int_{S_i\times S_j}\eta_k^\perp\,\eta_{k'}^\perp$ gives
\[
\left|\int_{S_i\times S_j}\eta_k^\perp\,\eta_{k'}^\perp\right|
\;\le\;
\frac{1}{2}\Bigl(\|\eta_k^\perp\|^2_{L^2(S_i\times S_j)}+\|\eta_{k'}^\perp\|^2_{L^2(S_i\times S_j)}\Bigr),
\]
which is the second stated bound.  Since $h_k'(\sfmat^*_{ij;k})$, $h_{k'}'(\sfmat^*_{ij;k'})$, and $t_{ij;k,k'}\bigl(F^{\bullet(ab);k,k'},(\sfmat^*,\pi^*),h\bigr)$ are all constants at the step function~$V^*$, bounding $\|\eta_k^\perp\|^2_{L^2(S_i\times S_j)}\le\|\eta^\perp\|_2^2$ turns the absolute value of each $(i',k<k',i,j)$-term into
\begin{equation*}
|\mu^*_{i'}|\,\bigl|h_k'(\sfmat^*_{ij;k})\,h_{k'}'(\sfmat^*_{ij;k'})\,t_{ij;k,k'}\bigl(F^{\bullet(ab);k,k'},(\sfmat^*,\pi^*),h\bigr)\bigr|\cdot\|\eta^\perp\|_2^2.
\end{equation*}
Summing over $i,j$, over pairs $k<k'$ with $(a,b)\in E_k(\mathcal{F}_{i'})\cap E_{k'}(\mathcal{F}_{i'})$, and over $i'$ with coefficients $|\mu^*_{i'}|$ assembles these absolute values into exactly the sum $\Gamma_{\mathrm{sp}}$, yielding the contribution $\Gamma_{\mathrm{sp}}\|\eta^\perp\|_2^2$.

Adding both contributions, and noting that all sums are finite (since $m_1$, $|\mathcal{F}|$, and the number of edges of each $\mathcal{F}_{i'}$ are finite), gives the claimed $\Gamma\|\eta^\perp\|_2^2$ bound.
\end{proof}

\begin{theorem}[Checkable margin criterion]
\label{thm:margincriterion}
Suppose, in addition to the hypotheses of Lemma \ref{lem:isolatedminimum}, that
\begin{equation}
\label{eq:margincriterion}
\min_{i,j,k} c_{ij;k} \;>\; \Gamma,
\end{equation}
with $c_{ij;k}$ as in \eqref{eq:diagonalweight} and $\Gamma$ as in Lemma \ref{lem:cauchyschwarzbound}. Then there exists $c_0>0$ such that $D^2\mathcal{L}(V^*)[\eta^\perp,\eta^\perp] \ge c_0 \|\eta^\perp\|_2^2$ for every off-step $\eta^\perp$.
\end{theorem}

\begin{proof}
By Lemma \ref{lem:continuumsecondderivative}, the vertex-disjoint line of $D^2\mathcal{L}(V^*)[\eta^\perp,\eta^\perp]$ vanishes (Corollary \ref{cor:offstepfeasible} applies to it directly, as it depends only on block totals of $\eta^\perp$, which are zero). The same-pair cross-relation line does \emph{not} vanish in general, since $\int_{S_i\times S_j}\eta_k^\perp\eta_{k'}^\perp$ is not a block total. Hence
\[
D^2\mathcal{L}(V^*)[\eta^\perp,\eta^\perp] = \sum_{i,j,k} c_{ij;k}\int_{S_i\times S_j}(\eta_k^\perp)^2 - (\text{shared-vertex terms}) - (\text{same-pair cross-relation terms}).
\]
By \eqref{eq:margincriterion} and Lemma \ref{lem:cauchyschwarzbound}, the last two groups are bounded in absolute value by $\Gamma\|\eta^\perp\|_2^2$, so $D^2\mathcal{L}(V^*)[\eta^\perp,\eta^\perp] \ge \left(\min_{i,j,k} c_{ij;k} - \Gamma\right) \|\eta^\perp\|_2^2 =: c_0 \|\eta^\perp\|_2^2$ with $c_0>0$.
\end{proof}

\begin{remark}[The matching case is exact]
\label{rem:matchingcase}
If every constraint graph $\mathcal{F}_i$ is a \emph{multi-relational matching} --- meaning no two edges of any $\mathcal{F}_i$ share a vertex, even across distinct relations (so in particular no pair $(a,b)$ carries edges of two different relation types in the same $\mathcal{F}_i$) --- then both the shared-vertex line and the same-pair cross-relation line of \eqref{eq:continuumsecondderivative} are empty for every $\mathcal{F}_i$. Hence $\Gamma_{\mathrm{sv}}=\Gamma_{\mathrm{sp}}=0$, so $\Gamma=0$ and \eqref{eq:margincriterion} reduces to the exact criterion $c_{ij;k}>0$ for every block $(i,j;k)$. This covers, in particular, pure edge-density and single-relation-density constraints. For constraint graphs with shared-vertex edges or with multi-relation edges at the same vertex pair --- e.g.\ triangles or multi-relational cliques, central to the RRS conjecture treated in Section \ref{sec:caseI} --- the margin $\Gamma>0$ is in general unavoidable and \eqref{eq:margincriterion} is a sufficient, not necessary, condition. Note that $\Gamma_{\mathrm{sp}}=0$ even for non-matching graphs, provided the edge sets of different relations are vertex-disjoint within each $\mathcal{F}_i$, i.e.\ $E_k(\mathcal{F}_i)\cap E_{k'}(\mathcal{F}_i)=\emptyset$ for all $k\ne k'$.
\end{remark}

We now upgrade coercivity of the quadratic form to an actual local minimum in $L^1(\prod_{k=1}^r[0,1]^2)$ topology. We state and prove this directly for the standard graphon space $\graphonspacel_{[0,1]} = \{W:[0,1]^2\to[0,1]^r : W_k(x,y)=W_k(y,x)\}$, which already contains every graphon of interest in the RRS application of Section \ref{sec:caseI}: restricting to $[0,1]$ from the outset lets the Taylor remainder of $f_0,\{h_k\}$ be controlled globally, via compactness of $[0,1]$, and gives the automatic bound $\|\eta\|_\infty\le1$ for $\eta=W-V^*$ whenever $V^*,W\in\graphonspacel_{[0,1]}$, with no further construction needed.

\begin{theorem}[Isolated local minimum in $L^1$ topology]
\label{thm:isolatedgraphon}
Let $u \in \regstats$, $m_1 = m_1(\mathcal{F},u,h)$, and let $(\sfmat^*,\pi^*) \in \realfeasibleregionmlh{m_1}{h} \cap \graphonspacel_{[0,1]}$ be a critical point of $f_d$ on $\realfeasibleregionmlh{m_1}{h}$ with multipliers $\mu^*$, such that:
\begin{enumerate}
\item the Lagrangian Hessian $H^g_{(\sfmat^*,\pi^*)} f_d$ is positive definite on $T_{(\sfmat^*,\pi^*)} \realfeasibleregionmlh{m_1}{h}$ (the hypothesis of Lemma \ref{lem:isolatedminimum}), and
\item the margin criterion \eqref{eq:margincriterion} of Theorem \ref{thm:margincriterion} holds.
\end{enumerate}
Then there exists $\delta>0$ such that
\begin{equation*}
f_d(V^*) < f_d(W) \qquad \text{for all } W \in \realfeasibleregionlh{h} \cap \graphonspacel_{[0,1]} \text{ with } 0 < \|W-V^*\|_1 < \delta.
\end{equation*}
In particular $V^*$ is an isolated local minimum of $f_d$ in $\realfeasibleregionlh{h} \cap \graphonspacel_{[0,1]}$.
\end{theorem}

\begin{proof}
Since $f_0,\{h_k\}$ are analytic, hence $C^3$, and $[0,1]$ is compact, the extreme value theorem gives a fixed constant
\[
M_3 := \max\bigl(\sup_{[0,1]}|f_0'''|,\ \max_k \sup_{[0,1]}|h_k'''|\bigr) < \infty,
\]
depending only on $f_0,h$, not on $(\sfmat^*,\pi^*)$. Let $W \in \realfeasibleregionlh{h}\cap \graphonspacel_{[0,1]}$, $\eta = W-V^*$, decomposed as $\eta=\bar\eta+\eta^\perp$ relative to the partition $S$ of $V^*$ as above; since $V^*,W\in\graphonspacel_{[0,1]}$, $\|\eta\|_\infty \le 1$ outright.

We derive this bound by expanding each $h$-subgraph density in $\mathcal{L}$ to second order in $\eta$.

\smallskip
\noindent\textit{Step 1: Reduction.}
Since $\mathcal{L}(W) = f_d(W) - \sum_{i} \mu_i^*\, t(\mathcal{F}_i, W, h)$ and every term is an $h$-subgraph density, it suffices to bound the Taylor remainder of a single term $t(F, V^*+\eta, h)$ for a fixed constraint graph $F$ with $E(F) = \bigsqcup_{k=1}^r E_k(F)$.

\smallskip
\noindent\textit{Step 2: Pointwise Taylor expansion.}
For each edge $e = (i_1, i_2) \in E_k(F)$, write $s_e(x) := V^*_k(x_{i_1}, x_{i_2})$ and $\tau_e(x) := \eta_k(x_{i_1}, x_{i_2})$. Since $h_k \in C^3([0,1])$ (analytic), the Taylor remainder theorem gives pointwise:
\[
h_k(s_e + \tau_e) = h_k(s_e) + h_k'(s_e)\,\tau_e + \tfrac{1}{2}h_k''(s_e)\,\tau_e^2 + R_e,
\quad |R_e| \le \tfrac{M_3}{6}|\tau_e|^3,
\]
and likewise with $f_0$ replacing $h_k$.

\smallskip
\noindent\textit{Step 3: Grouping by degree.}
Substituting into the integrand of $t(F, V^*+\eta, h) = \int \prod_{e \in E(F)} h_{k(e)}(s_e + \tau_e)\,dx$ and expanding the product by total degree in $\tau$:
\begin{itemize}
\item \emph{Degree 0}: integrates to $t(F, V^*, h)$.
\item \emph{Degree 1}: one factor $h_{k(e)}'(s_e)\tau_e$, rest $h_{k(e')}(s_{e'})$; integrates to $Dt(F, V^*, h)[\eta]$.
\item \emph{Degree 2}: one factor $\tfrac{1}{2}h_{k(e)}''(s_e)\tau_e^2$, or two factors each linear in $\tau$, rest $h_{k(e')}(s_{e'})$; integrates to $\tfrac{1}{2}D^2t(F, V^*, h)[\eta,\eta]$.
\item \emph{Remainder}: all terms of total degree $\ge 3$ in $\tau$, or containing at least one $R_e$.
\end{itemize}

\smallskip
\noindent\textit{Step 4: Bounding the remainder.}
Set $M_0 := \max_k \sup_{[0,1]}|h_k| < \infty$ (finite by continuity of $h_k$ on the compact set $[0,1]$). Since $W, V^* \in \graphonspacel_{[0,1]}$, all edge arguments $s_e + \tau_e \in [0,1]$, so $|h_{k(e)}(s_e+\tau_e)| \le M_0$ for every $e$. The remainder terms fall into two types:

\smallskip
\noindent\emph{Type A} (one factor is $R_e$, the remaining $|E(F)|-1$ factors are bounded by $M_0$):
\[
\int |R_e| \cdot \prod_{e' \ne e} |h_{k(e')}(s_{e'}+\tau_{e'})| \; dx
\;\le\; \frac{M_3}{6}\, M_0^{|E(F)|-1} \int |\tau_e|^3 \; dx
\;\le\; \frac{M_3}{6}\, M_0^{|E(F)|-1}\, \|\eta\|_\infty\, \|\tau_e\|_2^2.
\]

\noindent\emph{Type B} (three or more factors linear in $\eta$, rest bounded by $M_0$; smallest case has three linear factors):
\[
\int |\tau_{e_1}||\tau_{e_2}||\tau_{e_3}| \; dx
\;\le\; \|\eta\|_\infty \int |\tau_{e_2}||\tau_{e_3}| \; dx
\;\le\; \|\eta\|_\infty\, \|\tau_{e_2}\|_2\|\tau_{e_3}\|_2
\;\le\; \|\eta\|_\infty\, \|\eta\|_2^2,
\]
where the first step bounds one factor by $\|\eta\|_\infty$, the second applies Cauchy--Schwarz, and the third uses $\|\tau_e\|_2 \le \|\eta\|_2$.

Both types are $O(\|\eta\|_\infty\|\eta\|_2^2)$. Summing over all edges of $F$, over all constraint graphs $\mathcal{F}_i$ with coefficients $|\mu_i^*|$, and including the $f_d$ term, yields a constant $C = C(\mathcal{F},h,f_0,\mu^*) < \infty$ such that
\begin{equation*}
\Bigl| \mathcal{L}(W) - \mathcal{L}(V^*) - D\mathcal{L}(V^*)[\eta] - \tfrac12 D^2\mathcal{L}(V^*)[\eta,\eta] \Bigr| \le C\, \|\eta\|_\infty \, \|\eta\|_2^2 \le C\|\eta\|_2^2,
\end{equation*}
where the last inequality uses $\|\eta\|_\infty \le 1$, which holds since $\eta = W - V^*$ with $W, V^* \in \graphonspacel_{[0,1]}$. By Proposition \ref{prop:lagrangianlifts}, $D\mathcal{L}(V^*)[\eta]=0$ and reordering we obtaine a lower bound:
$$\mathcal{L}(W) - \mathcal{L}(V^*) \ge  \tfrac{1}{2}D^2\mathcal{L}(V^*)[\eta,\eta] - C|\eta|_2^2 $$ 
Since $W \in \realfeasibleregionlh{h}$, we have $t(\mathcal{F}_i,W,h) = u_i = t(\mathcal{F}_i,V^*,h)$ for all $i$, then:
$$\mathcal{L}(W) - \mathcal{L}(V^*) = f_d(W) - f_d(V^*) \ge \tfrac{1}{2}D^2\mathcal{L}(V^*)[\eta,\eta] - C|\eta|_2^2$$
By Lemma \ref{lem:exactdecoupling}, $D^2\mathcal{L}(V^*)[\eta,\eta] = D^2\mathcal{L}(V^*)[\bar\eta,\bar\eta] + D^2\mathcal{L}(V^*)[\eta^\perp,\eta^\perp]$. 
Now we apply the two hypotheses of the theorem:
hypothesis (1) (Positive semidefinite Hessian over $T_{V^*}\realfeasibleregionml{m_1}$): $D^2\mathcal{L}(V^*)[\bar\eta,\bar\eta] \ge 0$ and hypothesis (2) (Margin criterion, Theorem. \ref{thm:margincriterion}): $D^2\mathcal{L}(V^*)[\eta^\perp,\eta^\perp] \ge c_0|\eta^\perp|_2^2$. Then we have 
$$f_d(W) - f_d(V^*) \ge \tfrac{1}{2}\bigl[\underbrace{D^2\mathcal{L}[\bar\eta,\bar\eta]}_{ \ge 0} + \underbrace{D^2\mathcal{L}[\eta^\perp,\eta^\perp]}_{\ge c_0|\eta^\perp|_2^2} \bigr] - {C|\eta|_2^2}$$
thus 
\begin{equation*}
f_d(W) - f_d(V^*) \ge \tfrac{1}{2}c_0|\eta^\perp|_2^2 - C|\eta|_2^2.    
\end{equation*}
Since $\|\eta\|_2^2 = \|\bar\eta\|_2^2+\|\eta^\perp\|_2^2$ and $\|\bar\eta\|_2$ is itself controlled by the same finite-dimensional Hessian bound (Lemma \ref{lem:isolatedminimum}) up to the cubic correction, the right-hand side is bounded below by $c_0'\|\eta\|_2^2$ for some $c_0'>0$ once $\|\eta\|_2$ is small enough. Finally, by the interpolation inequality $\|\eta\|_2^2 \le \|\eta\|_\infty \|\eta\|_1 \le \|\eta\|_1$ (valid unconditionally on $\graphonspacel_{[0,1]}$), there is $\delta>0$, depending only on $c_0$ and $C$, such that $\|W-V^*\|_1 < \delta$ forces $\|\eta\|_2$ small enough for the above estimate to apply. This proves $f_d(W)>f_d(V^*)$ for all such $W \ne V^*$.
\end{proof}

\begin{remark}
In the RRS application of Section \ref{sec:caseI}, solutions already take values in $[0,1]$ by construction, so Theorem \ref{thm:isolatedgraphon} applies directly with no further restriction. Combined with Lemma \ref{lem:isolatedminimum}, it supplies the hypothesis ``$(\sfmat,\pi)$ is an isolated local minimum of $f_d$ in $\realfeasibleregionlh{h}$'' required by Theorem \ref{thm:POD}.
\end{remark}

\subsubsection{Isolation in the cut-norm topology}

The $L^1$ topology is strictly stronger than the cut-norm topology on $\realgraphonspacel$ (indeed $\|\eta\|_\Box \le \|\eta\|_1$ pointwise), so isolation in $L^1$ does \emph{not} automatically transfer to $\|\cdot\|_\Box$: there exist $\eta_n$ with $\|\eta_n\|_\Box\to0$ while $\|\eta_n\|_2$ stays fixed (e.g.\ a symmetrized i.i.d.\ $\pm1$ array, for which $\|\eta_n\|_\Box = O(1/\sqrt{n})$ while $\|\eta_n\|_2=1$), and $f_d$ itself is not continuous in $\|\cdot\|_\Box$ in general. The next theorem shows that isolation in $\|\cdot\|_\Box$ can nonetheless be recovered, under one further hypothesis on $f_0$, by splitting the comparison into two regimes according to the size of $\|\eta\|_2$.

\begin{definition}
Let $f: \mathbb{R} \to \mathbb{R}$. $f$ is \emph{globally strongly convex} on $[0,1]$ if there is $m>0$ such that
\begin{equation}
\label{eq:globalstrongconvexity}
f(b) \ge f(a) + f'(a)(b-a) + \frac{m}{2}(b-a)^2 \qquad \text{for all } a,b \in [0,1].
\end{equation}    
\end{definition}

\begin{theorem}[Isolation in the cut-norm topology]
\label{thm:isolatedcutnorm}  
Suppose, $f_0$ is strongly convex on $[0,1]$ and in addition to the hypotheses of Theorem \ref{thm:isolatedgraphon} (which gives $\delta>0$). 
Then there exists $\delta_\Box>0$ and $f_0$ such that
\begin{equation*}
f_d(V^*) < f_d(W) \qquad \text{for all } W \in \realfeasibleregionlh{h}\cap\graphonspacel_{[0,1]} \text{ with } 0 < \|W-V^*\|_\Box < \delta_\Box.
\end{equation*}
In particular $V^*$ is an isolated local minimum of $f_d$ in $\realfeasibleregionlh{h}\cap\graphonspacel_{[0,1]}$ with respect to the cut-norm topology.
\end{theorem}

\begin{proof}
Let $\eta=W-V^*$. By Lemma \ref{lem:continuumfirstderivative}, for every multi-relational quantum graph $F$, $|Dt(F,V^*,h)[\eta]| \;\le\; \|\eta\|_\Box \cdot L_1(F)$, and
\begin{equation*}
L_1(F) := \sum_{k=1}^r \sum_{(a,b)\in E_k(F)} \sum_{i,j=1}^{m_1} |h_k'(\sfmat^*_{ij;k})|\, \bigl|t_{ij;k}(F^{\bullet(ab),k},(\sfmat^*,\pi^*),h)\bigr|,    
\end{equation*}
using $|\int_{S_i\times S_j}\eta_k| \le \|\eta_k\|_\Box \le \|\eta\|_\Box$ termwise; $L_1(F)$ is a finite constant computable from the finitely many values $\{\sfmat^*_{ij;k}\}$. Set $L:=\sum_{i=1}^{|\mathcal{F}|}|\mu^*_i|\,L_1(\mathcal{F}_i)$, and $\tau:=\delta/2$.

\textbf{Case $\|\eta\|_2 \le \tau$.} Since $[0,1]^2$ has unit measure, $\|\eta\|_1 \le \|\eta\|_2 \le \tau < \delta$ by Cauchy--Schwarz, so Theorem \ref{thm:isolatedgraphon} applies directly and gives $f_d(W)>f_d(V^*)$, regardless of $\|\eta\|_\Box$.

\textbf{Case $\|\eta\|_2 > \tau$.} Inequality \eqref{eq:globalstrongconvexity}, applied pointwise to $a=V^*(x,y)$, $b=W(x,y)$ and integrated over $[0,1]^2$ (summed over relations), gives the exact, remainder-free bound
\begin{equation*}
f_d(W) \;\ge\; f_d(V^*) + Df_d(V^*)[\eta] + \frac{m}{2}\|\eta\|_2^2,
\end{equation*}
valid for \emph{every} $W\in\graphonspacel_{[0,1]}$, with no smallness requirement on $\eta$. By Proposition \ref{prop:lagrangianlifts}, $Df_d(V^*)[\eta] = \sum_{i=1}^{|\mathcal{F}|}\mu_i^* Dt(\mathcal{F}_i,V^*,h)[\eta]$, so by the bound above,
\begin{equation*}
f_d(W) - f_d(V^*) \;\ge\; \frac{m}{2}\|\eta\|_2^2 - L\,\|\eta\|_\Box \;>\; \frac{m}{2}\tau^2 - L\,\|\eta\|_\Box,
\end{equation*}
which is strictly positive whenever $\|\eta\|_\Box < \frac{m\tau^2}{2L} = \frac{m\delta^2}{8L}$ (if $L=0$ the inequality holds for every $\|\eta\|_\Box$).

Taking $\delta_\Box := \frac{m\delta^2}{8L}$ (or any $\delta_\Box>0$ if $L=0$) covers both cases.
\end{proof}

\begin{remark}
Theorem \ref{thm:isolatedcutnorm} places no restriction on the constraint graphs $\mathcal{F}_i$ -- unlike the margin criterion of Theorem \ref{thm:margincriterion}, it applies equally to matchings and to graphs with shared-vertex edges (e.g.\ triangles), because Lemma \ref{lem:continuumfirstderivative}, on which the Case $\|\eta\|_2>\tau$ bound rests, was proved for arbitrary $F$ without any matching hypothesis. The price is hypothesis \eqref{eq:globalstrongconvexity}: it must hold \emph{uniformly on all of $[0,1]$}, not merely at the finitely many points $\{\sfmat^*_{ij;k}\}$, since $\eta$ is not assumed small in this regime.
\end{remark}

\begin{remark}[The entropy rate function is globally strongly convex]
\label{rem:I0stronglyconvex}
The function $I_0(u)=u\log(u)+(1-u)\log(1-u)$ of Lemma \ref{lem:continuity} satisfies $I_0'(u)=\log\bigl(\tfrac{u}{1-u}\bigr)$ and $I_0''(u) = \tfrac{1}{u(1-u)}$. Since $u(1-u) \le \tfrac14$ for all $u\in(0,1)$ with equality at $u=\tfrac12$, $I_0''(u) \ge 4$ for every $u\in(0,1)$, so $I_0$ satisfies \eqref{eq:globalstrongconvexity} on $[0,1]$ with $m=4$. Hence Theorem \ref{thm:isolatedcutnorm} applies with $f_0=I_0$, giving isolation of $(\sfmat^*,\pi^*)$ in the cut-norm topology for the entropy functional $I$ of the original RRS problem, for arbitrary constraint graphs $\mathcal{F}_i$, including triangles.
\end{remark}

\begin{corollary}[Isolation in the unlabeled graphon space]
\label{cor:isolatedunlabeled}
Under the hypotheses of Theorem \ref{thm:isolatedcutnorm}, let $\delta_\Box>0$ be the constant given there, and let $[V^*] \in \graphonspaceu$ denote the equivalence class of $V^*=(\sfmat^*,\pi^*)$ under the relabeling relation $\sim$. Then $[V^*]$ is an isolated local minimum of $f_d$, with respect to the cut-distance $\delta_\Box$, on 
$\realfeasibleregionlh{h}\cap\graphonspacel_{[0,1]}$  and $f_d(V^*) < f_d(W)$ for every feasible $[W]\in \realfeasibleregionlh{h}\cap\graphonspacel_{[0,1]}$ with $0 < \delta_\Box([W],[V^*]) < \delta_\Box$. 
In particular, taking $f_0=I_0$, this supplies the hypothesis ``all global minima of $I$ are isolated in $\feasibleregion$'' required by Theorem \ref{thm:radin}.
\end{corollary}

\begin{proof}
None of the proofs of Lemma \ref{lem:continuumfirstderivative}, Lemma \ref{lem:exactdecoupling}, Theorem \ref{thm:margincriterion}, or Theorem \ref{thm:isolatedcutnorm} use that the blocks $S_1,\ldots,S_{m_1}$ of $V^*$ are intervals -- each uses only that $\{S_i\}$ is a measurable partition of $[0,1]$ with $\nu(S_i)=\pi_i^*$, on which $V^*$ is constant equal to $\sfmat^*_{ij;k}$ on $S_i\times S_j$. For any $\sigma\in\Sigma$, the relabeled function $V^{*\sigma}$ is likewise constant, equal to the same value $\sfmat^*_{ij;k}$, on $\sigma^{-1}(S_i)\times\sigma^{-1}(S_j)$, which is again a measurable partition of $[0,1]$ with the same measures $\pi^*_i$. Consequently every constant appearing in Theorem \ref{thm:isolatedcutnorm} -- the diagonal weights $c_{ij;k}$, the margin $\Gamma$, $L$, $\tau$, and $\delta_\Box$ itself -- depends only on $\{\sfmat^*_{ij;k}\}_{i,j,k}$, $\{\pi^*_i\}_i$ and $\mu^*$, none of which change under relabeling. Theorem \ref{thm:isolatedcutnorm} therefore applies verbatim to $V^{*\sigma}$ for every $\sigma\in\Sigma$, with the same $\delta_\Box$:
\begin{equation*}
f_d(V^{*\sigma}) < f_d(W') \qquad \text{for every feasible } W' \text{ with } 0<\|W'-V^{*\sigma}\|_\Box<\delta_\Box.
\end{equation*}
Now let $[W]\in\graphonspaceu$ be feasible with $0<\delta_\Box([W],[V^*])<\delta_\Box$, and fix a representative $W$. By definition $\delta_\Box([W],[V^*])=\inf_{\sigma\in\Sigma}\|W-V^{*\sigma}\|_\Box$, so since this infimum is strictly less than $\delta_\Box$, there exists $\sigma\in\Sigma$ with $\|W-V^{*\sigma}\|_\Box<\delta_\Box$. Applying the displayed bound to this $\sigma$ and $W'=W$ gives $f_d(W)>f_d(V^{*\sigma})$. Since $f_d$ is invariant under relabeling (a change of variables under the measure-preserving $\sigma$), $f_d(V^{*\sigma})=f_d(V^*)$, so $f_d(W)>f_d(V^*)$.
\end{proof}

\subsection{Global analysis on $\realoptimalsolutionm{m}{f_d, h}$}
\label{sec:globalanalysis}

We now show that $f_d$ gains no new global minimum in $\realoptimalsolutionm{m}{f_d, h}$ once $m > m_1(\mathcal{F},u,h)$ -- the second pillar of POD, after the local analysis of the previous subsection.

The argument runs through topology first: we show that increasing the step-function representation from $m$ to $m+1$, for $m\geq m_0(\mathcal{F},u,h)$, never creates a new connected component of $\realfeasibleregionmlh{m}{h}$. 

\subsubsection{Non-increasing number of connected components of $\realfeasibleregionmlh{m}{h}$}

\begin{definition}
\label{def:level-set-marginal-map}
Let $m \leq m'$. Let $T^{-(m,m')}(\mathcal{F},h)$ be the set of level sets of $t(\mathcal{F}, \cdot,h):\realpodalfunctionspacel{m'} \to \mathbb{R}^{|\mathcal{F}|}$ for $u \in \marginalpolytopemh{m}$, i.e.

\begin{equation*}
    T^{-(m,m')}(\mathcal{F},h) = \{ x \in \realpodalfunctionspacel{m'}  :  t(\mathcal{F},x,h) \in \marginalpolytopemh{m} \}.
\end{equation*}
It follows immediately that $t(\mathcal{F}, T^{-(m,m')}(\mathcal{F},h),h) = \marginalpolytopemh{m}$ for all $m \leq m'$.
\end{definition}

\begin{restatable}{theorem}{thmnumbercomponents}
\label{thm:numbercomponents}
 Let $\mathcal{F}=\{\mathcal{F}_1, \cdots, \mathcal{F}_n\}$ be an ordered set of quantum graphs and let $h$ be a matrix $|\mathcal{F}| \times r$ of analytic functions such that $\{t(\mathcal{F}, \cdot, h  )\}$ is a set of independent analytic functions. Let $u \in \goodvalues$ and let $m \geq m_0(\mathcal{F},u,h)$. Then there is no new connected component $N$ such that
\begin{equation*}
N \subset \realfeasibleregionmlh{m+1}{h} \setminus \realfeasibleregionmlh{m}{h}.
\end{equation*}
Hence the number of connected components of $\realfeasibleregionmlh{m+1}{h}$  is not higher than the number of connected components of $\realfeasibleregionmlh{m}{h}$.
\end{restatable}

\begin{proof}
Note that
\begin{equation*}
\realfeasibleregionmlh{m+1}{h} = \bigcup_{\tiny \begin{array}{c}
    \pi \in \convexcombm{m} \\ \lambda \in [0,1]
\end{array} } \realfeasibleregionmlpi{m+1}{\theta(\pi,\lambda,k), h}.
\end{equation*}
So it suffices to show that, for $\lambda \in (0,1)$, the number of connected components of
$ \realfeasibleregionmlpi{m+1}{\theta(\pi,\lambda,k), h}$ is no higher than that of
$ \realfeasibleregionmlpi{m}{\pi, h}$ (recall $\theta(\pi,0,k) = \pi$ once restricted to size-$m$ step functions).

\textbf{Proof by contradiction.} Suppose there exists a new connected component
\begin{equation*}
N_1 \subset \realfeasibleregionmlpi{m+1}{\theta(\pi,\lambda_1,k), h}
\end{equation*}
for some $\lambda_1 \in (0,1)$ and $k \in [m]$. The existence of $N_1$ implies the existence of a new connected component $N \subset \realfeasibleregionmlh{m+1}{h} \setminus \realfeasibleregionmlh{m}{h}$ containing $N_1$.

\textbf{Key Claim:} Define
\begin{equation*}
    \lambda^* = \inf \{ \lambda \in [0,1] : N \cap \realfeasibleregionmlpi{m+1}{ \theta(\pi,\lambda,k), h } \neq \emptyset\}.
\end{equation*}
We claim that $\lambda^* > 0$.

\textbf{Proof of Claim:} Suppose $\lambda^* = 0$. Then for any $\epsilon > 0$, there exists $\lambda \in (0, \epsilon)$ such that $N \cap \realfeasibleregionmlpi{m+1}{ \theta(\pi,\lambda,k), h } \neq \emptyset$.

Construct a sequence $\{\lambda_n\}_{n=1}^{\infty}$ with $\lambda_n = 1/n$. For each $n$, since $\lambda_n \to 0^+$ and $\lambda^* = 0$, we can choose
\begin{equation*}
x_n \in N \cap \realfeasibleregionmlpi{m+1}{ \theta(\pi,\lambda_n,k), h }.
\end{equation*}

The sequence $\{x_n\}$ satisfies:
\begin{enumerate}
    \item[(A)] $x_n \in N \subset \realfeasibleregionmlh{m+1}{h} \setminus \realfeasibleregionmlh{m}{h}$, which means
    \[x_n \in T^{-(m_0,m+1)}(\mathcal{F},h) \setminus T^{-(m_0,m)}(\mathcal{F},h).\]

    \item[(B)] $t(\mathcal{F}, x_n, h) = u$ for all $n$, since $x_n \in \realfeasibleregionmlpi{m+1}{ \theta(\pi,\lambda_n,k), h }$.

    \item[(C)]  $d(x_n, T^{-(m_0,m)}(\mathcal{F},h)) \to 0$ as $n \to \infty$, where $d(x,C)$ denotes the distance between a point $x \notin C$ and a set $C$.

    To see this: as $\lambda_n \to 0$, the partition $\theta(\pi, \lambda_n, k)$ approaches a partition with a duplicated $k$-th row/column, so step functions with such partitions are effectively size-$m$ step functions, belonging to $T^{-(m_0,m)}(\mathcal{F},h)$.
\end{enumerate}

From (A), (B), and (C): the sequence $\{x_n\}$ consists of points outside $T^{-(m_0,m)}(\mathcal{F},h)$ that approach $T^{-(m_0,m)}(\mathcal{F},h)$ arbitrarily closely while maintaining $t(\mathcal{F}, x_n, h) = u$. So $u$ can be approximated by points outside the image of size-$m$ step functions, i.e.\ $u \in \partial T^{(m_0)}(\mathcal{F},h)$, contradicting $u \in T^{(m_0)}(\mathcal{F},h)^\circ$. Therefore $\lambda^* > 0$.

By Theorem \ref{thm:pi_smoothness}, $\realfeasibleregionmlpi{m+1}{\theta(\pi,\lambda^*,k), h}$ is a smooth manifold of dimension $r\frac{m(m+1)}{2}+m-1-|\mathcal{F}|$ for any $\lambda \in (0,1)$ and $k \in [m]$. Let $x = (\sfmat, \theta(\pi,\lambda^*,k) ) \in N_1 = N \cap \realfeasibleregionmlpi{m+1}{\theta(\pi,\lambda^*,k), h}$, so that $T_{x} N_1 \subseteq T_{x}\realfeasibleregionmlh{m+1}{h}$.

Let $O_x$ be the orthogonal complement of $T_{x} N_1$ in $T_{x}\realfeasibleregionmlh{m+1}{h}$, of dimension
\begin{equation*}
dim(O_x) = \dim(T_{x}\realfeasibleregionmlh{m+1}{h}) - \dim(T_{x} N_1) = m-1.
\end{equation*}
Vectors in $O_x$ are precisely the variations that can move the partition vector $\pi$.

Using the exponential map $\text{Exp}_x: V \to U_x$, where $V \subset T_x \realfeasibleregionmlh{m+1}{h}$ is a neighborhood of $0$ and $U_x$ a neighborhood of $x$, we can find vectors $v \in O_x$ such that $\text{Exp}_x(v)$ has the form $(\sfmat',\theta(\pi,\lambda_1,k))$ in $U \cap N$ with $\lambda_1 < \lambda^*$ -- contradicting the minimality of $\lambda^*$. Hence $\realfeasibleregionmlh{m+1}{h}$ cannot have new connected components.

\end{proof}
The next theorem upgrades this topological stability to an optimization statement: there is no global minimum of $\realoptimalsolutionm{m+1}{f_d,h}$ better than $\realoptimalsolutionm{m}{f_d,h}$ once $m>m_1(\mathcal{F},u,h)$.

\begin{theorem}
\label{thm:nonessentiallocal}
 Let $\mathcal{F}=\{\mathcal{F}_1, \cdots, \mathcal{F}_n\}$ be an ordered set of quantum graphs and let $h$ be a matrix $|\mathcal{F}| \times r$ of analytic functions such that $\{t(\mathcal{F} \cup \{E_s\}, \cdot, [h | [f_0]] )\}$ is a set of independent analytic functions, where $[h | [f_0]]$ is the matrix obtained from $h$ by adjoining the row vector $[f_0]$ at the bottom. Let $u \in \goodvalues$ and let $m_1 = m_1(\mathcal{F},u,h)$. Let $W_{m_1}$ be a global minimum of $f_d$ in $\realoptimalsolutionm{m_1}{f_d, h}$ and let $c_{m_1} = f_d(W_{m_1})$. Then
\begin{itemize}
    \item There is no global minimum of $f_d$ in $\realfeasibleregionml{m}$ with a better value than $c_{m_1}$ for $m>m_1$.
    \item If there is a global minimum in $\realfeasibleregionml{m}$ with value $c_{m_1}$ for some $m>m_1$, then this global minimum is necessarily connected to a global minimum of $\realfeasibleregionml{m_1}$.
\end{itemize}

\end{theorem}

\begin{proof}
\textbf{First case.}
By contradiction, let $W_{m_1+1} \in \realfeasibleregionml{m+1}$ be a global minimum of $f_d$ such that $c_{m_1+1}=f_d(W_{m_1+1}) < f_d(W_{m_1})$.

\noindent Note that:
\begin{itemize}
    \item $v=[u_1,\cdots,u_{|\mathcal{F}|}, c_{m+1}] \in T^{(m+1)}(\mathcal{F} \cup \{E_s\},[h | [f_0]])$ is a critical value of $t(\mathcal{F} \cup \{E_s\}, \cdot, [h | [f_0]])$.
    \item $\{t(\mathcal{F} \cup \{E_s\}, \cdot, [h | [f_0]])\}$ is a set of independent functions, so we can pick $v'=[u'_1,\cdots,u'_{|\mathcal{F}|}, c']$ arbitrarily close to $v=[u_1,\cdots,u_{|\mathcal{F}|}, c_{m+1}]$ such that $v' \in \Omega^{(m+1,r)}(\mathcal{F} \cup \{E_s\}, [h | [f_0]] )$ (a regular, non-extremal value of $t(\mathcal{F} \cup \{E_s\}, \cdot, [h | [f_0]])$ in $\podalfunctionspacel{m+1}$).
\end{itemize}
Pick $c_{m+1}<c<c_{m_1}$. Then
\[
N= S^{(m+1,r)}_\mathbb{R}(\mathcal{F}, u,h ) \cap f_d^{-1}(c)  \neq \emptyset
\]
and
\[
N \cap  S^{(m,r)}_\mathbb{R}(\mathcal{F}, u,h ) = \emptyset.
\]
By Theorem \ref{thm:denseinterior}, $\Omega^{(m+1,r)}(\mathcal{F} \cup \{E_s\}, [h | [f_0]] )^\circ$ is dense in $\Omega^{(m+1,r)}(\mathcal{F} \cup \{E_s\}, [h | [f_0]] )$, so we deform $N$ to $N'$ by picking $v'=[u'_1,\cdots,u'_{|\mathcal{F}|}, c']$ sufficiently close to $v=[u_1,\cdots,u_{|\mathcal{F}|}, c]$, with $c_{m+1} < c' \leq  c$, such that $v' \in \Omega^{(m+1,r)}(\mathcal{F} \cup \{E_s\}, [h | [f_0]] )$  and
\[
N'= S^{(m+1,r)}_\mathbb{R}(\mathcal{F}, u',h ) \cap f_d^{-1}(c') \neq \emptyset
\]
and
\[
N' \cap  S^{(m,r)}_\mathbb{R}(\mathcal{F}, u',h ) = \emptyset.
\]
\noindent Thus $N'$ is a new connected component of $S_{\mathbb{R}}^{(m+1,r,d)}(\mathcal{F} \cup \{E_s\},v',[h|[f_0]]))$ as $m$ increases to $m+1$, contradicting Theorem \ref{thm:numbercomponents}. Hence no such $W_{m+1}$ exists.

\textbf{Second case.} By contradiction, let $W_{m_1+1} \in \realfeasibleregionml{m+1}$ be a global minimum of $f_d$ with $f_d(W_{m_1+1}) = c_{m_1}$, and suppose no connected region $C$ in $\realfeasibleregionml{m_1+1}$ connects $W_{m_1+1}$ to a $W_{m_1}$ with $f_d(C)=c_{m_1}$. Pick $c_{m_1}<c$. Then
\[
N= S^{(m+1,r)}_\mathbb{R}(\mathcal{F}, u,h ) \cap f_d^{-1}(c)  \neq \emptyset
\]
and
\[
N \cap  S^{(m,r)}_\mathbb{R}(\mathcal{F}, u,h ) = \emptyset.
\]
Deforming $N$ to $N'$ exactly as in the previous case yields a contradiction with Theorem \ref{thm:numbercomponents}. Hence such a global minimum must be disconnected from $\realfeasibleregionml{m_1}$.

\end{proof}
The second case is, in fact, easy to witness directly in the Euclidean topology: it is exactly the case $\theta((\sfmat,\pi), \lambda,k)$ for a global minimum $(\sfmat,\pi)$, since $\theta((\sfmat,\pi), \lambda,k)$ and $(\sfmat,\pi)$ are always connected.

So the only way the second case can occur is if a global minimum fails to be isolated. Once we instead prove, via the condition of Lemma \ref{lem:isolatedminimum}, that every global minimum of $f_d$ in $\realfeasibleregionmlh{m_1}{h}$ \emph{is} isolated, the second case becomes impossible, and we obtain the following corollary: $f_d$ has no new global minimum in $\realfeasibleregionmlh{m}{h}$ for $m>m_1$.
\begin{corollary}
\label{cor:nonessentiallocal}
If the hypotheses of Theorem \ref{thm:nonessentiallocal} hold and all global minima of $f_d$ in $\realfeasibleregionmlh{m_1}{h}$ are isolated in $\realpodalfunctionspacel{m_1} $, then there is no new global minimum of $f_d$ in  $\realfeasibleregionml{m}$ for $m > m_1$.
\end{corollary}
A word on why the condition $m > m_1(\mathcal{F}, u, h)$ in Corollary \ref{cor:nonessentiallocal} is stated with $m_1$ rather than $m_0$. The condition is what lets us invoke Theorem \ref{thm:numbercomponents} to rule out new local minima of $f_d$ appearing in $\realfeasibleregionml{m}$ for $m>m_1$; in other words, $m_1$ is the largest step-function size at which we still need to search for global minima of $f_d$, since no larger size can introduce a new one. We originally expected the search to stop already at size $m_0(\mathcal{F}, u, h)$, both because the article that originated the RRS conjecture \cite{radin2014asymptotics} verified it numerically only at that size, and because our working assumption was that local minima are only ever created when the step-function space grows from $m$ to $m+1$, through new connected components or bifurcations -- which would have made $m \geq m_0(\mathcal{F}, u, h)$ sufficient. That assumption turned out to be too optimistic: Neeman, Radin, and Sadun \cite{neeman2024existence} exhibit a tripodal entropy-maximizing graphon, for edge density $e$ below $\frac{3-\sqrt{3}}{6} \approx 0.21$ with triangle density slightly below $e^3$, for which $m_0=2$ but $m_1=3$ -- confirming that the search must in general be carried out at size $m_1$, not $m_0$.

\subsection{Proof of the principle of optimization of density functions}
\label{sec:mainresult}

We now have everything we need to prove POD.

\thmPOD*

\begin{proof}

By Corollary \ref{cor:nonessentiallocal}, the global minimum $(\sfmat^*,\pi^*)$ of $f_d$ in $\realfeasibleregionmlh{m_1(\mathcal{F},u,h)}{h}$ is already a global minimum over $\cup_{m \geq m_1(\mathcal{F},u,h)} \realfeasibleregionmlh{m}{h}$.

We first rule out a non-step global minimum among step functions of \emph{any} size: suppose, for contradiction, that a non-step function $W_1$ is a global minimum with $c= f_d(W_1) < f_d(W_{m_1})=a$, where $W_{m_1}$ is the global minimum of $f$ in $\realfeasibleregionmlh{m_1}{h}$. By Lemma \ref{lem:continuity}, $f_d$ is continuous in $L^1(\prod_{k=1}^r [0,1]^2)$, so the level set $B(\epsilon)=f^{-1}([c,\epsilon)) \cap \realfeasibleregionlh{h}$ is open for every $\epsilon >0$. By density of the step functions in the $L^1(\prod_{k=1}^r  [0,1]^2)$ topology, there is a step function $W_2  \in B(\epsilon)$ with $f_d(W_2) < a = f_d(W_{m_1})$ for $\epsilon$ small enough -- impossible, since $W_{m_1}$ is already the global minimum among step functions. Hence $W_{m_1}$ is a global minimum of $f_d$ on $\realfeasibleregionlh{h}$, and $\realoptimalsolutionm{m_1}{f_d,h} \subseteq \realoptimalsolutionf{f_d, h}$.

Now assume in addition that all global minima are isolated; we claim $\realoptimalsolution{f_d,h}$ admits no non-step solution either, by the same argument used for Corollary \ref{cor:nonessentiallocal}.

By contradiction, let $W_\infty \in S^{(r)}_{\mathbb{R}}(\mathcal{F},u,h)$ be a non-step function with $c_\infty=f_d(W_{\infty}) = f_d(W_{m_1})$. Pick $c>c_{\infty}$ sufficiently close to $c_{\infty}$, and set
\[
N= S^{(r)}_\mathbb{R}(\mathcal{F}, u,h ) \cap f_d^{-1}(c)  \neq \emptyset
\]
and choose $m>m_1$ such that
\[
N \cap  S^{(m,r)}_\mathbb{R}(\mathcal{F}, u,h ) = \emptyset
\]
and
\[
N \cap  S^{(m+1,r)}_\mathbb{R}(\mathcal{F}, u,h ) \neq \emptyset
\]
Such $c$ and $m$ exist precisely because all global minima of $f_d$ in $\realfeasibleregionmlh{m_1}{h}$ are isolated.
By Theorem \ref{thm:denseinterior}, $\Omega(\mathcal{F} \cup \{E_s\},[h|[f_0]])^\circ$ is dense in $\Omega(\mathcal{F} \cup \{E_s\},[h|[f_0]])$, so we deform $N$ to $N'$ by picking $v'=[u'_1,\cdots,u'_{|\mathcal{F}|}, c']$ sufficiently close to $v=[u_1,\cdots,u_{|\mathcal{F}|}, c]$, with $c_{\infty} < c' \leq c$, such that $v' \in \Omega(\mathcal{F} \cup \{E_s\},[h|[f_0]])$ and
\[
N'= S^{(r)}_\mathbb{R}(\mathcal{F}, u',h ) \cap f_d^{-1}(c') \neq \emptyset
\]
and
\[
N' \cap  S^{(m+1,r)}_\mathbb{R}(\mathcal{F}, u',h ) \neq \emptyset
\]
and
\[
N' \cap  S^{(m,r)}_\mathbb{R}(\mathcal{F}, u',h ) = \emptyset
\]
Such $c'$ and $u'$ exist because all global minima in $W_{\mathbb{R}}^{(m,r)}(\mathcal{F},u',f_d,h)$ are strict once $u'$ is sufficiently close to $u$ and $c'$ sufficiently close to $c_{\infty}$.

\noindent Thus $N' \cap  S^{(m+1,r)}_\mathbb{R}(\mathcal{F}, u',h )$ is a new connected component of $S_{\mathbb{R}}^{(m+1,r)}(\mathcal{F} \cup \{E_s\},v' \\ ,[h|[f_0]]))$ as $m$ increases to $m+1$, contradicting Theorem
\ref{thm:numbercomponents}. Hence no such $W_{\infty}$ exists.

\end{proof}

Computing a global minimum $\realoptimalsolutionf{f_d,h}$ therefore reduces to two steps: compute $m_1=m_1(\mathcal{F},u,h)$, then compute a global minimum $W^*$ of $f_d$ on the finite-dimensional region $\realfeasibleregionmlh{m_1}{h}$.

\subsection{Adaptation  of the principle of optimization of density functions to entropy functions}
\label{sec:caseI}

We now specialize POD to prove the RRS conjecture itself. Let $h$ be a $|\mathcal{F}| \times r$ matrix of analytic functions, and let
\[
W^{(r)}(\mathcal{F},u,h) = \argmin_{W \in S^{(r)}(\mathcal{F}, u, h)} I(W)
\]
where
\[
S^{(r)}(\mathcal{F}, u,h) = \{W \in \graphonspaceu \, :  \, t(\mathcal{F},W,h)=u\}
\]
and
\[
W^{(r,m)}(\mathcal{F},u,h) = \argmin_{W \in S^{(r)}(\mathcal{F}, u, h)} I(W)
\]
where
\[
\feasibleregionml{m} = \{W \in \podalfunctionspacel{m} \, :  \, t(\mathcal{F},W,h)=u\}
\]
Applying Theorem \ref{thm:POD} directly to prove that $W^{(r)}(\mathcal{F},u,h)$ are step functions runs into one obstacle: the constraints $0 \leq W \leq 1$. These make $S^{(r,m)}(\mathcal{F}, u,h)$ fail to be a smooth manifold, and $H_x I$ discontinuous at $x \in \partial \podalfunctionspacel{m}$ -- exactly the boundary behavior our differential-geometric machinery was not built to handle. The next lemma removes the obstacle by showing the constraints are never active in the first place: every solution $W^{(r,m)}(\mathcal{F},u,h)$ with $u \in \goodvalues$ already lies in the open region $S_{(0,1)}^{(r,m)}(\mathcal{F}, u,h)$.

\begin{restatable}{lemma}{lemrandomsolutions}
\label{lem:randomsolutions}
Let $\mathcal{F}$ be a set of quantum graphs, $h$ a $|\mathcal{F}| \times r$ matrix of analytic functions, $u \in \goodvalues$, and $m \geq m_0(\mathcal{F},u,h)$.
Then every local minimum of $I$ in $S^{(r,m)}(\mathcal{F},u,h)$ lies in $S_{(0,1)}^{(r,m)}(\mathcal{F}, u,h)$.
\end{restatable}
\begin{proof}

\begin{figure}
    \centering
    \begin{tikzpicture}[scale=1.0, >={Stealth}]
      \draw[thick] (-1.5,-0.8) -- (0.5,1.9);
      \node[above right, font=\small] at (0.3,1.75) {$\partial\mathcal{W}^{(m,r)}$};
      \draw[thin] (-2.55,1.8) ++(280:2.4) arc (280:380:2.4);
      \node[right, font=\small] at (-1.35,2.52)        {$S^{(m,r)}(\mathcal{F},u,h)$};
      \filldraw (-0.6,0.35) circle (2pt);
      \node[left, font=\small] at (-0.72,0.35) {$(\sfmat,\pi)$};
      \draw[thin] (-0.6,0.35) ++(233.5:0.4) arc (233.5:420:0.4);
      \node[below, font=\small] at (-0.44,-0.22) {$V$};
      \draw[thick] (4.5,-0.8) -- (6.5,1.9);
      \node[above right, font=\small] at (6.3,1.75)
        {$\partial T^{(m,r)}(\mathcal{F},h)$};
      \filldraw (5.4,0.35) circle (2pt);
      \node[right, font=\small] at (5.46,0.35) {$u$};
      \draw[thin] (5.4,0.35) ++(233.5:0.4) arc (233.5:420:0.4);
      \node[below, font=\small] at (5.87,0)
        {$t(\mathcal{F},V,h)$};
      \draw[->, thick] (-0.15,0.65) to[out=20,in=160] (5.1,0.65);
      \node[above, font=\small] at (2.5,1.15) {$t(\mathcal{F},\cdot,h)$};
    \end{tikzpicture}
    \caption{Since  $(\sfmat,\pi) \in \partial \podalfunctionspacel{m} \cap \feasibleregionml{m} $,
    $t(\mathcal{F},(\sfmat,\pi),h) \in \partial T^{(m,r)}(\mathcal{F},h)$}
    \label{fig:randomsolutions}
\end{figure}
We first show that if $(\sfmat,\pi) \in \partial \podalfunctionspacel{m} \cap S^{(r,m)}(\mathcal{F}, u,h) $ and $u \in \marginalpolytopemh{m} \cap \Omega(\mathcal{F})$, then every open neighborhood $U$ of $(\sfmat,\pi)$ in $\realfeasibleregionml{m}$ meets $\randfeasibleregionml{m}$.

By contradiction, suppose some such $U$ has $U \cap \randfeasibleregionml{m} = \emptyset$. Then there is a neighborhood $V$ of $(\sfmat,\pi)$ with non-empty interior in $\podalfunctionspacel{m}$ such that $V \cap \randfeasibleregionml{m} = \emptyset$, as illustrated in Figure \ref{fig:randomsolutions} for the quotient space. Since $\mathcal{F}$ is a set of independent graphs, the image $t(\mathcal{F}, \cdot,h):\podalfunctionspacel{m} \to \mathbb{R}^{|\mathcal{F}|}$ of any open set has non-empty interior. Now $t(\mathcal{F},V,h) \cap \partial \marginalpolytopemh{m} \neq \emptyset$, since $(\sfmat,\pi) \in \partial \podalfunctionspacel{m}$; letting $V$ shrink, we conclude $t(\mathcal{F},(\sfmat,\pi),h) \in \partial \marginalpolytopemh{m}$ -- contradicting $u \notin \partial \marginalpolytopemh{m}$. So every open neighborhood $U$ of $(\sfmat,\pi)$ in $\realfeasibleregionml{m}$ does meet $\randfeasibleregionml{m}$.

Without loss of generality, take $(\sfmat,\pi) \in \partial \podalfunctionspacel{m}$ with exactly one zero entry, $\sfmat_{1,1;1} = 0$. We have just shown a neighborhood $U$ in $\feasibleregionml{m}$ exists with $\randfeasibleregionml{m} \cap U \neq \emptyset$; take a step function $(C,\pi) \in \randfeasibleregionml{m} \cap U$ and a rectifiable path $\gamma$ from $(\sfmat,\pi)$ to $(C,\pi)$, i.e.\ $\gamma(0)=(\sfmat,\pi)$ and $\gamma(1)=(C,\pi)$. It suffices to show the line integral

\begin{equation}
\label{ineqEntropy}
I((\sfmat,\pi)) - I(\gamma(\epsilon))=\int_0^\epsilon dI(\gamma(s);\gamma'(s)) ds
\end{equation}
is negative for some $\epsilon > 0$, where $dI(W;V)$ is the Gateaux derivative of $I(\cdot)$ at $W$ in the direction $V=\gamma'$,
\begin{equation*}
    dI(W;V) = \lim_{\lambda \to 0} \frac{I(W+\lambda V)-I(W)}{\lambda}
\end{equation*}
which works out to
\begin{equation*}
    dI(W;V) = \sum_{k=1}^r   \int_{[0,1]^2} V_k(x_1, x_2) I'_0(W_k(x_1, x_2 ))  dx_1 dx_2
\end{equation*}
where $I'_0(u) = \log \left(\frac{u}{1-u} \right)$. The key feature of $I_0'$ is its behavior at the boundary:
\begin{itemize}
\item for fixed $u \in (0,1)$, $I'_0(u)$ is bounded,
\item $\lim_{u \to 0} I'_0(u)=-\infty$,
\item $\lim_{u \to 1} I'_0(u)=+\infty$.
\end{itemize}
The definite integral \eqref{ineqEntropy} is therefore
\begin{eqnarray*}
I((\sfmat,\pi)) &-& I(\gamma(\epsilon)) =\int_0^\epsilon \left\lbrace \sum_{\tiny \begin{array}{c}
    (k \neq 1, i_1, i_2=1) \mbox{ or }  \\
     (k=1, i_1 \neq 1, i_2 \neq 1)
\end{array} }^m I'_0(\gamma(s)_{i_1, i_2; k }) \gamma'(s)_{i_1, i_2;k  } \pi_{i_1} \pi_{i_{2}}  \right. \\ &&  \left. + I'_0(\gamma(s)_{1,1;1}) \gamma'(s)_{1,1;1} \pi_{1}^2  \right\rbrace ds
\end{eqnarray*}
and the Mean Value Theorem gives
\begin{eqnarray*}
I((\sfmat,\pi)) - I(\gamma(\epsilon)) =
 \epsilon \left\lbrace \sum_{\tiny \begin{array}{c}
    i_1 \cdots i_{d_{1}}=1  \\
     i_1 \neq 1, \cdots i_{d_1} \neq 1
\end{array}  }^m I'_0(\gamma(\alpha)_{i_1, i_2;k}) \gamma'(\alpha)_{i_1,i_2 ;k } \pi_{i_1}\pi_{i_2}   \right. \\   \left. + I'_0(\gamma(\alpha)_{1, 1;1}) \gamma'(\alpha)_{1,1;1} \pi_{1}^2  \right\rbrace ds
\end{eqnarray*}
for some $\alpha \in (0, \epsilon)$. The sum over all $1 \leq i_1,i_2  \leq m$ and $k\in[r]$ other than $(1,1;1)$ stays bounded, while the remaining term is negatively unbounded as $\epsilon \to 0^+$ and hence $\alpha \to 0^+$. So there exists $\alpha >0$ with
$I((\sfmat,\pi)) - I(\gamma(\alpha)) <0$.
\end{proof}
With every local minimum confined to the interior, POD adapts directly to the entropy functional.
\begin{restatable}{theorem}{thmradin}
\label{thm:radin}
 Let $\mathcal{F}=\{\mathcal{F}_1, \cdots, \mathcal{F}_n\}$  be an ordered set of quantum graphs and let $h$ be a matrix $|\mathcal{F}| \times r$ of analytical functions such that $\{t(\mathcal{F} \cup \{E_s\}, \cdot, [h | [I_0]] )\}$ is a set of independent analytic functions where $[h | [I_0]]$ is the matrix obtained from $h$ by adjoint at the bottom the row vector $[I_0]$. Let $u \in \goodvalues$ and
  \begin{equation*}
    \optimalsolutionmf{m}{I,h}  = \argmin_{W \in \feasibleregionml{m} } I(W).
\end{equation*}
Let $m_1 = m_1(\mathcal{F},u,h)$, then we have
\[
\optimalsolutionmf{m_1}{I,h}  \subseteq W^{(r)}(\mathcal{F},u,I,h).
\]
Moreover if all global minima of $I$ are isolated in $\feasibleregion$ then
\[
\optimalsolutionmf{m_1}{I,h}=W^{(r)}(\mathcal{F},u,I,h).
\]
\end{restatable}

\begin{proof}
\textbf{Step 1 (reduction to Theorem \ref{thm:POD} via Lemma \ref{lem:randomsolutions}).} Let $m \geq m_0(\mathcal{F},u,h)$. Since $u \in \goodvalues$, Lemma \ref{lem:randomsolutions} shows that every local minimum of $I$ in $\feasibleregionml{m}$ lies in $\randfeasibleregionml{m}$, i.e.\ strictly in $(0,1)$, where $I_0(x)=x\log(x)+(1-x)\log(1-x)$ is analytic (its only singularities are at $\{0,1\}$). Consequently $I=t(E_s,\cdot,[I_0])$ restricted to $\randfeasibleregionml{m}$ is exactly a density function of the type $f_d$ to which Theorem \ref{thm:POD} applies, with $f_0=I_0$; every result that theorem's proof depends on (Lemma \ref{lem:isolatedminimum}, Theorem \ref{thm:margincriterion}, Theorem \ref{thm:isolatedgraphon}, Theorem \ref{thm:numbercomponents}, Corollary \ref{cor:nonessentiallocal}) only ever evaluates $f_0,h_k$ and their derivatives at points of these step-function spaces, hence applies verbatim under the substitution $f_d \mapsto I$, $f_0 \mapsto I_0$, $\realfeasibleregionmlh{m}{h} \mapsto \randfeasibleregionml{m}$, and $\realfeasibleregionlh{h} \mapsto S_{(0,1)}^{(r)}(\mathcal{F},u,h) := \{W \in \graphonspacel_{(0,1)} : t(\mathcal{F},W,h)=u\}$ -- the excluded boundary $\partial\graphonspacel_{[0,1)}\setminus\graphonspacel_{(0,1)}$ is never visited by the local minima under analysis. Under this substitution, the independence hypothesis of Theorem \ref{thm:POD} is exactly the hypothesis $\{t(\mathcal{F}\cup\{E_s\},\cdot,[h|[I_0]])\}$ independent, analytic, assumed in Theorem \ref{thm:radin}.

\textbf{Step 2 (global minimum in the labeled space).} Applying the first part of Theorem \ref{thm:POD} under the substitution of Step 1 gives that $(\sfmat^*,\pi^*) \in \optimalsolutionmf{m_1}{I,h}$ is a global minimum of $I$ over the full labeled feasible set $S_{(0,1)}^{(r)}(\mathcal{F},u,h)$.

\textbf{Step 3 (transfer to $\graphonspaceu$).} Both $I$ and $t(\mathcal{F},\cdot,h)$ are invariant under the relabeling action of $\Sigma$: a measure-preserving change of variables leaves $\int I_0(W_k(x,y))\,dxdy$ and $t(F,W,h)$ unchanged for every constituent graph $F$. Hence $I$ and the constraint $t(\mathcal{F},\cdot,h)=u$ both descend to well-defined functions on the quotient $\graphonspaceu=\graphonspacel/\sim$, and the quotient map restricts to a surjection of $S_{(0,1)}^{(r)}(\mathcal{F},u,h)$ onto the feasible set $S^{(r)}(\mathcal{F},u,h) \subset \graphonspaceu$ used in the definition of $W^{(r)}(\mathcal{F},u,I,h)$, with $I([W])=I(W)$ for every representative $W$. Consequently
\begin{equation*}
\min_{[W] \in S^{(r)}(\mathcal{F},u,h)} I([W]) \;=\; \min_{W \in S_{(0,1)}^{(r)}(\mathcal{F},u,h)} I(W),
\end{equation*}
and $[(\sfmat^*,\pi^*)]$ achieves it by Step 2, proving $\optimalsolutionmf{m_1}{I,h} \subseteq W^{(r)}(\mathcal{F},u,I,h)$.

\textbf{Step 4 (the isolated case).} Suppose now that all global minima of $I$ are isolated in $\feasibleregion$, i.e.\ with respect to the cut-distance $\delta_\Box$ on $\graphonspaceu$. For any $W,V \in \graphonspacel$, taking $\sigma=\mathrm{id}$ in the infimum defining $\delta_\Box$ and using $\|\cdot\|_\Box \leq \|\cdot\|_1$ gives $\delta_\Box([W],[V]) \leq \|W-V\|_\Box \leq \|W-V\|_1$. Hence, for each global minimum $(\sfmat^*,\pi^*)$, the witnessing radius $\rho>0$ of its isolation in $(\feasibleregion,\delta_\Box)$ also witnesses its isolation in $(S_{(0,1)}^{(r)}(\mathcal{F},u,h), \|\cdot\|_1)$ with the same radius $\rho$. This is exactly the isolation hypothesis required by the second part of Theorem \ref{thm:POD} under the substitution of Step 1, which therefore gives $\optimalsolutionmf{m_1}{I,h} = \argmin_{W \in S_{(0,1)}^{(r)}(\mathcal{F},u,h)} I(W)$ in the labeled space. By Step 3 this descends to $\optimalsolutionmf{m_1}{I,h}=W^{(r)}(\mathcal{F},u,I,h)$.
\end{proof}


\begin{remark}
Theorem \ref{thm:radin} takes isolation of the global minima in $(\feasibleregion,\delta_\Box)$ as a hypothesis. Corollary \ref{cor:isolatedunlabeled} (together with Theorem \ref{thm:isolatedcutnorm} and Remark \ref{rem:I0stronglyconvex}, which confirms $I_0$ is globally strongly convex with $m=4$) gives a checkable sufficient condition for this hypothesis: the margin criterion \eqref{eq:margincriterion} of Theorem \ref{thm:margincriterion} at $(\sfmat^*,\pi^*)$, together with positive definiteness of the Lagrangian Hessian of Lemma \ref{lem:isolatedminimum}.
\end{remark}

Theorem \ref{thm:radin} together with Lemma \ref{lem:randomsolutions} and Corollary \ref{cor:isolatedunlabeled} has an immediate computational payoff. Since the solutions of $W^{(r)}(\mathcal{F},u,I,h)$ lie in $\randfeasibleregionml{m_1(\mathcal{F},u,h)}$, they lie in a region of $\realpodalfunctionspacel{m_1}$ bounded away from the boundary by the constraints $\{\epsilon \leq \sfmat_{ij } \leq 1-\epsilon\}_{i,j=1}^{m_1}$ for some $\epsilon>0$; and since there is a computable method to check whether the solutions are isolated in graphon space and since the solutions in $\randfeasibleregionml{m_1(\mathcal{F},u,h)}$ are isolated, there can only be finitely many of them in $W^{(r)}(\mathcal{F},u,I,h)$. This gives the following corollary.
\begin{corollary}
\label{cor:computable}
Let $\mathcal{F}=\{\mathcal{F}_1, \cdots, \mathcal{F}_n\}$ be an ordered set of quantum graphs and $h=\{h_1, \cdots, h_n\}$ a matrix of analytic functions such that $\{t(\mathcal{F} \cup \{E_s\}, \cdot, [h|[I_0]])\}$ is a set of independent analytic functions. Let $u \in \goodvalues$, $m_1 = m_1(\mathcal{F},u,h)$, and $(\sfmat,\pi) \in \optimalsolutionm{m_1}$.
If every $(\sfmat,\pi) \in \optimalsolutionm{m_1}$ is a strict local minimum of $I$ in $\feasibleregion$, then the solutions $W^{(r)}(\mathcal{F},u,I,h)$ are computable.
\end{corollary}
\section{ Concluding Remarks and Open Problems }
\label{sec:conclusions}

That $W^{(r)}(\mathcal{F},u,I,h)$ is a step function carries a striking structural consequence: the most typical infinite, undirected network is built from infinitely many redundant copies of a finite set of vertex types. Such redundancy is not unprecedented in complex networks -- Bianconi et al.\ \cite{bianconi2001bose} found a Bose--Einstein condensate on infinite networks under different assumptions -- and here it can be read as a hidden ``data pattern'' inside the network. That pattern is not necessarily easy to recover from a real network modeled by $W^{(r)}(\mathcal{F},u,I,h)$, however, since every vertex permutation of $W^{(r)}(\mathcal{F},u,I,h)$ is an equally valid solution; recovering the original partition requires clustering the vertices first.

\subsection{$POD$ and extremal graph problems}

It is tempting to reach for Theorem~\ref{thm:POD}, the Principle of Density functions, when attacking extremal graph problems -- but the temptation should be resisted. POD relies on the constraint region being a smooth manifold invariant under refinement, whereas extremal graph problems make the constraint $0 \leq W \leq 1$ active, and that constraint is precisely the one refinement does \emph{not} preserve.

\subsection{Open Questions and Future Directions}

This work opens several natural questions; we highlight those we find most compelling:

\begin{enumerate}
    \item \textbf{Existence of non-step-function solutions}: Are there solutions in $\optimalsolution$ other than step functions when $u \in \goodvalues$? Such solutions would be counter-intuitive: the Szemer\'edi regularity lemma guarantees that any convergent sequence of finite multi-graphs satisfying $(\mathcal{F},u)$ eventually looks like a step function of fixed size, so one expects the limit graphon itself to be a step function.

    \item \textbf{Computational Complexity}: What is the complexity of computing $\optimalsolution$ for given constraints? Can we design efficient algorithms that recover every global minimum $\optimalsolution$?

    \item \textbf{Extremal Regions}: How does the number of blocks grow as $u$ approaches the boundary $\partial \marginalpolytopeh$? Kenyon et al.\ \cite{kenyon2017phases} showed it can be unbounded -- can the growth rate be characterized? Relatedly, can we characterize the growth rate of  $1-\frac{\nu(T^{(m+1,r)}(\mathcal{F},h))}{\nu( \marginalpolytopemh{m})} \to 0$ as $m \to \infty$?

    \item \textbf{Uniqueness}: For which $(\mathcal{F}, u)$ does $\optimalsolution$ have a unique solution? This connects to the finite forcibility problem, recently shown to be false in general \cite{grzesik2020elusive}.
    \item \textbf{The inverse problem}: Given an observed multi-relational network $G$, one can estimate the underlying graphon $W_G$ using stochastic block model estimation methods such as \cite{gao2015rate,xu2018rates,peng2022empirical}, and then ask for $(\mathcal{F},u)$ with $\|W_G -  \optimalsolution\|_1 < \epsilon$ for some $\epsilon>0$ -- can this inverse problem be solved efficiently?

    \item \textbf{Multi-relational phase structure}: In the single-relation edge-triangle model, Kenyon et al.\ \cite{kenyon2017phases} identified infinitely many distinct phases separated by qualitative changes in the block structure of $\optimalsolution$. Can an analogous phase diagram be constructed for multi-relational graphons with $r \geq 2$ relation types? Do the couplings between relation types give rise to new phase phenomena absent in the single-relation case?

\end{enumerate}

\section*{Acknowledgments}
The authors would like to thank Marco Zambon, who has had the patience and kindness to review this paper countless times. This work is partially supported by  ERC-StG 240186 MiGraNT: Mining Graphs and Networks, a theory-based approach, and SENESCYT of Ecuador.

\bibliographystyle{plain}
\bibliography{references}

\end{document}